\documentclass[letterpaper,11pt,oneside,reqno]{amsart}
\usepackage{amsmath,amsthm,amsfonts,amssymb,bbm,color}
\usepackage{graphicx,psfrag,subfigure,url}
\usepackage{cite}
\usepackage{mathrsfs}
\usepackage[colorlinks=true]{hyperref}
\usepackage[DIV13]{typearea}
\usepackage{comment}
\usepackage{caption}
\numberwithin{equation}{section}
\newcommand{\bflambda}{\ensuremath{\bar{\lambda}}}
\newcommand{\bfmu}{\ensuremath{\bar{\mu}}}
\newcommand{\bfnu}{\ensuremath{\bar{\nu}}}
\newcommand{\I}{{\rm i}}
\newcommand{\Id}{\mathbbm{1}}
\newcommand{\Or}{\mathcal{O}}

\newcommand{\e}[0]{\varepsilon}

\newcommand{\EE}{\ensuremath{\mathbb{E}}}
\newcommand{\PP}{\ensuremath{\mathbb{P}}}
\newcommand{\R}{\ensuremath{\mathbb{R}}}
\newcommand{\C}{\ensuremath{\mathbb{C}}}
\newcommand{\Z}{\ensuremath{\mathbb{Z}}}

\newcommand{\Sym}{\ensuremath{\mathrm{Sym}}}

\renewcommand{\d}{\mathrm d}
\renewcommand{\Re}{\operatorname{Re}}
\renewcommand{\Im}{\operatorname{Im}}

\newtheorem{theorem}{Theorem}[section]
\newtheorem{proposition}[theorem]{Proposition}

\newtheorem{lemma}[theorem]{Lemma}
\newtheorem{corollary}[theorem]{Corollary}

\newtheorem{remark}[theorem]{Remark}

\title{Anisotropic $(2+1)$d growth and Gaussian limits of $q$-Whittaker processes}

\author[A. Borodin]{Alexei Borodin}
\address{A. Borodin,
Massachusetts Institute of Technology,
Department of Mathematics,
77 Massachusetts Avenue, Cambridge, MA 02139-4307, USA}
\email{borodin@math.mit.edu}

\author[I. Corwin]{Ivan Corwin}
\address{I. Corwin, Columbia University,
Department of Mathematics,
2990 Broadway,
New York, NY 10027, USA}
\email{ivan.corwin@gmail.com}

\author[P.L. Ferrari]{Patrik Ferrari}
\address{P.L. Ferrari,
 Bonn University,
Institute for Applied Mathematics,
Endenicher Allee 60, 53115 Bonn, Germany}
\email{ferrari@uni-bonn.de}

\begin{document}

\begin{abstract}
We consider a discrete model for anisotropic $(2+1)$-dimensional growth of an interface height function. Owing to a connection with $q$-Whittaker functions, this system enjoys many explicit integral formulas. By considering certain Gaussian stochastic differential equation limits of the model we are able to prove a space-time limit of covariances to those of the $(2+1)$-dimensional additive stochastic heat equation (or Edwards-Wilkinson equation) along characteristic directions. In particular, the bulk height function converges to the Gaussian free field which evolves according to this stochastic PDE.
\end{abstract}

\sloppy \maketitle
\setcounter{tocdepth}{1}
\tableofcontents

\section{Introduction}

A key notion in statistical mechanics and probability is that of universality classes. Roughly, this holds that the long-time and large-scale behavior of possibly complex stochastic systems group into broad classes which all show the same scaling exponents and statistics describing fluctuations. The connection between microscopic dynamics and the associated universality class is generally facilitated by a few physically relevant quantities which can be computed on the microscopic side. This article will probe the universality class associated with two-dimensional interface growth.

\subsection{Random growth in $(2+1)$-dimensions}
Random growth models have received significant attention recently. In one spatial dimensional, generic local random growth models with slope dependent growth rates fall into the $(1+1$)-dimensional Kardar-Parisi-Zhang (KPZ) universality class of which quite a lot is now known -- see the reviews and lecture notes~\cite{FS10,Fer10b,Cor11,Cor14,QS15,BG12,Qua11}.
In two spatial dimensions much less is known. It is predicted (see, for example,~\cite{Wol91}) that generic local random growth models with slope dependent growth rates fall into one of two $(2+1)$-dimensional KPZ universality classes -- the isotropic or the anisotropic class. The archetypical $(2+1)$-dimensional model is the continuum KPZ stochastic PDE
$$
\frac{\partial h}{\partial t}(t,x) = \frac{1}{2}\Delta h(t,x) + (\nabla h,Q\nabla h)(t,x) + \xi(t,x).
$$
Here $x=(x_1,x_2)\in \R^2$ and $t\in \R_{\geq 0}$ (despite saying ``space-time'' we put time $t$ before space $x$). The function $h(t,x)\in \R$ is the height above location $x$ at time $t$, the Laplacian $\Delta$ is on $\R^2$ and the noise $\xi$ is space-time white. The quadratic form in $\nabla h$ is defined with respect to a $2\times 2$ matrix $Q$. When the signature of $Q$ is $(+,+)$ or $(-,-)$, the equation is called ``isotropic'' while in the mixed case $(+,-)$ or $(-,+)$ (and the boarder case when one term is 0) it is called ``anisotropic''. Presently this equation has not been shown to be well-posed -- the noise is sufficiently rough so that solutions are distribution valued and hence not regular enough to define the non-linearity by standard means.

The difference between isotropic and anisotropic growth is quite marked. In the isotropic case, all directions are roughly the same, and there is no theoretical prediction for the scaling exponent or fluctuations. Numerics predict fluctuation growth of order $t^{0.24}$ (with $0.24$ only approximate, but different from $1/4$) -- see~\cite{HH12}. The anisotropic case has very different behavior. In particular, it was predicted by Wolf~\cite{Wol91} in 1991 that the anisotropic equation should have fluctuations which grow like $\sqrt{\ln t}$ and behave asymptotically like the equation without non-linearity -- the $(2+1)$-dimensional Edwards-Wilkinson (EW) / additive stochastic heat equation
\begin{equation}\label{EWeq}
\frac{\partial u}{\partial t}(t,x) = \frac{1}{2}\Delta u(t,x) + \xi(t,x).
\end{equation}
This equation is not function valued, but rather takes values in the space of generalized functions. The Gaussian free field is an invariant measure for the equation.

There are some results in the literature confirming the $\sqrt{\ln t}$, and Gaussian free field behavior for certain discrete growth models. Numerics performed by Halpin-Healy and Assdah~\cite{HHA92} support the $\sqrt{\ln t}$ prediction. The first rigorous result was in 1997 by Pr\"{a}hofer and Spohn~\cite{PS97} who showed $\sqrt{\ln t}$ scale fluctuations for the Gates-Westcott model~\cite{GW95} through exact calculations. The Gaussian free field prediction (in addition to the $\sqrt{\ln t}$ scaling) was demonstrated in 2008 by Borodin and Ferrari~\cite{BF08} for a discrete model related to Schur processes (the $q=0$ case of the model we introduce below). Aspects of that result were extended to a slightly more general model recently by Toninelli~\cite{Ton15}. These results (which are essentially the full set of rigorous results for $(2+1)$-dimensional anisotropic KPZ models) deal only with behavior at a single time, and it remained an important open problem to demonstrate the non-trivial temporal limit of such models.

To our knowledge, this present paper, along with the work of Borodin, Corwin and Toninelli~\cite{BCT15} on growth models on the torus (initiated near the completion of this present work, though completed prior to it), is the first work in which a scaling limit to the $(2+1)$-dimensional EW equation has been established for a model in the $(2+1)$-dimensional anisotropic KPZ universality class. The previous work mentioned above have dealt with only a single time. From the outset, let us be clear about two things. First -- we do not prove convergence as space-time processes, but rather work with covariances (avoiding complications related to working with generalized function valued solutions). Second -- our convergence result is for a system of SDEs which arise as limits of a particular discrete growth model. We rely on exact covariance formulas for the SDEs which come from the fact that the discrete growth model is an ``integrable probabilistic system''. Direct analysis of the discrete model is presently beyond our techniques.

This paper is the first instance where bulk asymptotics have been extracted from Macdonald processes (apart from the free-fermionic Schur case). The exact formulas provided by the structure Macdonald processes results in a number of nice formulas as we take certain limits. For instance, in Section~\ref{SecLLN} we demonstrate simple determinant form solutions to a system of ODEs which arises in describing the law of large numbers behavior of the bulk as the Macdonald parameter $q\to 1$. Another example is the explicit covariance formulas given in Section~\ref{SecFluc} for the SDEs which describe the fluctuations around that bulk law of large numbers behavior.

\subsection{The $q$-Whittaker particle system}

\begin{figure}
\begin{center}
\psfrag{11}[l]{$\lambda_1^{(1)}$}
\psfrag{12}[l]{$\lambda_1^{(2)}$}
\psfrag{22}[l]{$\lambda_2^{(2)}$}
\psfrag{1n}[l]{$\lambda_1^{(N)}$}
\psfrag{nn}[l]{$\lambda_N^{(N)}$}
\includegraphics[height=4cm]{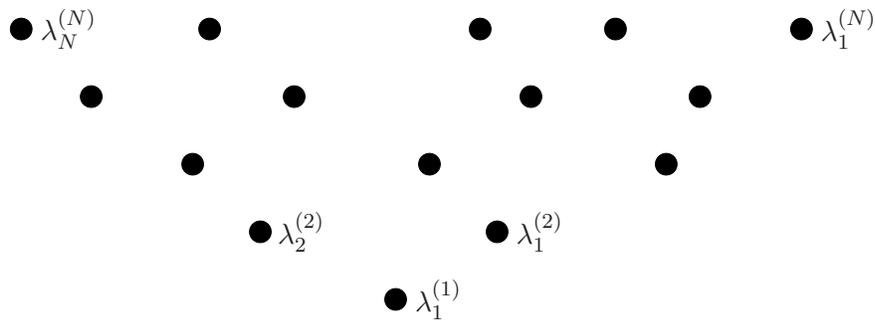}
\caption{The plot of $(\lambda^{(n)}_k,n)$ for $1\leq k\leq n\leq N=5$ yields an interlacing triangular array.}
\label{Figinterlacing}
\end{center}
\end{figure}

Measures on interlacing partitions or Gelfand-Tsetlin patterns defined in terms of symmetric functions have played an important role in asymptotic representation theory and probability, especially related to models of interacting particles, growth, directed polymers, and more broadly the Kardar-Parisi-Zhang universality class. In particular, special properties of families of symmetric functions, such as those in the Macdonald hierarchy (including $q$-Whittaker, Hall-Littlewood, Jack and Schur), enable one to construct interesting Markov dynamics which preserve these classes of measures, and also obtain exact and concise formulas for expectations of many observables under these measures. While there have been some clear successes in this direction (see for example~\cite{BC11, BG12, BP13, Cor14} for some surveys and reviews) there remain many directions untouched and many open problems unresolved. In this present paper we probe the bulk fluctuation behavior of certain $(2+1)$-dimensional growth models associated with these measures, as well as study certain SDE and SPDE limits.

Our investigation starts at the level of $q$-Whittaker processes. These are measures on interlacing sequences of partitions, or equivalently Gelfand-Tsetlin patterns, or interlacing triangular arrays of non-negative integers -- see Figure~\ref{Figinterlacing} for an illustration and Section~\ref{secnotback} for definitions and notations related to the objects we presently discuss. The measures we consider on such interlacing arrays are called $q$-Whittaker processes and are specified by ``specializations'' $\rho$ of $q$-Whittaker functions.

The $q$-Whittaker process under the ``Plancherel'' specialization $\rho$ which is indexed by a parameter $\gamma>0$ can be realized as the time $\gamma$ distribution of a fairly simple Markov dynamic on the interlacing triangular array $\big\{\lambda^{(n)}_k(\gamma):1\leq k\leq n\leq N\big\}$. Start with packed initial data $\lambda^{(n)}_k(0) \equiv 0$ for all $1\leq k\leq n\leq N$. At time $\gamma$, associate to each $\lambda^{(n)}_k(\gamma)$ a rate
$$
\frac{\big(1-q^{\lambda^{(n-1)}_{k-1}(\gamma) - \lambda^{(n)}_{k}(\gamma)}\big) \big(1-q^{\lambda^{(n)}_{k}(\gamma) - \lambda^{(n)}_{k+1}(\gamma)+1}\big)}{1-q^{\lambda^{(n)}_{k}(\gamma) - \lambda^{(n-1)}_{k}(\gamma)+1}}
$$
exponential clock (i.e. in time $d\gamma$ the clock rings with probability given by $d\gamma$ times the above rate). When the $\lambda^{(n)}_{k}(\gamma)$-clock rings, find the longest string $\lambda^{(n)}_k(\gamma) = \lambda^{(n+1)}_k(\gamma) = \cdots = \lambda^{(n+\ell)}_k(\gamma)$ and increase all coordinates in this string by one. Observe that if $\lambda^{(n)}_{k}(\gamma) = \lambda^{(n-1)}_{k-1}(\gamma)$ the jump rate automatically vanishes, hence the interlacing partitions remain interlacing under these dynamics. This ``push-block'' $q$-Whittaker particle system was introduced in~\cite{BC11} and generalizes the $q=0$ dynamics of~\cite{BF08} which relate to Schur processes (instead of $q$-Whittaker processes). Section~\ref{secAlphaDyn} describes this and other dynamics which grow $q$-Whittaker processes for various specializations.

\begin{figure}
\begin{center}
\psfrag{n}[l]{$n$}
\psfrag{x}[l]{$x$}
\includegraphics[height=4cm]{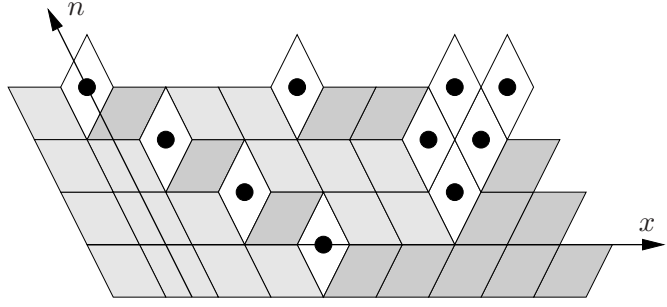}
\caption{An interface associated with the $q$-Whittaker particle system. The coordinates of the particles corresponding to $\lambda_k^{(n)}$ is $(\lambda_k^{(n)}-k+n,n)$.}
\label{green}
\end{center}
\end{figure}

The $q$-Whittaker particle system can also be mapped onto an interface growth model as illustrated in Figure~\ref{green}. In the $q=0$ case, this interface was shown to fluctuate like $\sqrt{\ln \gamma}$ and, after a suitable change of variables, have Gaussian free field statistics at a single long time. The key to that result was the ``determinantal'' structure of Schur processes which provides effective means to extract many asymptotic limits. Presently we do not know how to prove an analogous set of results for the $q$-Whittaker process -- only for certain Gaussian limits of it which arise as we take $q\to 1$. Given similar results in the two extremes of $q=0$ and $q\to 1$, it is reasonable to expect that the asymptotic behavior remains the same for the whole range of $q\in (0,1)$.

In place of the determinantal structure of Schur processes, we instead have the following type of exact formulas which follow from eigenrelations for $q$-Whittaker functions (see~\cite{BC11} or Proposition~\ref{PropTheorem45} for the precise statement of this result along with the contours of integration):
\begin{equation*}
\begin{aligned}
\EE \Bigg[\prod_{i=1}^{m} q^{\lambda^{(n_i)}_{n_i}(\gamma) +\cdots + \lambda^{(n_i)}_{n_i-r_i+1}(\gamma)}\Bigg]
&= \prod_{i=1}^{m} \frac{1}{(2\pi \I)^{r_i} r_i!} \, \oint \cdots \oint \prod_{1\leq i<j\leq m} \, \prod_{k=1}^{r_i} \prod_{\ell=1}^{r_j} \frac{q(z_{i,k} - z_{j,\ell})}{z_{i,k} - q z_{j,\ell}}\\
&\times \prod_{i=1}^{m} \frac{ (-1)^{\frac{r_i(r_i+1)}{2}} \!\!\!\prod\limits_{1\leq k<\ell\leq r_i} (z_{i,k} - z_{i,\ell})^2}{\prod\limits_{k=1}^{r_i} (z_{i,k})^{r_i}} \, \prod_{k=1}^{r_i} \frac{1}{(1-z_{i,k})^{n_{i}}}\, e^{z_{i,k}(q-1)} \, dz_{i,k}.
\end{aligned}
\end{equation*}
While these formulas contain enough information to entirely identify the distribution of the time $\gamma$ distribution, combining them so as to extract useful asymptotics remains a challenge and has really only been successful in the study of $\lambda^{(n)}_n$ or $\lambda^{(n)}_1$.

Simulations of the $q$-Whittaker particle system created soon after its introduction, such as illustrated in Figure~\ref{configs}, revealed an interesting phenomena as $q=e^{-\e}\to 1$ (i.e. $\e\to 0$). If time is scaled so that $\gamma = \e^{-1}\tau$, and the entire picture is centered by $\gamma$ and scaled by $\e^{-1}$, then starting from the right, bands of particles seemed to deterministically peel off. There still seemed to be some randomness, but on a smaller $\e^{-1/2}$ scale. Indeed, the first set of results which we prove in this paper are a characterization of this curious law of large number (LLN) behavior and then a description of the scale $\e^{-1/2}$ central limit theorem (CLT) type fluctuations around the LLN. Following from our $q$-Whittaker dynamics, we can write down ODEs for the LLN and SDEs for the CLT. Following from our integral formulas, we can also write down explicit integral and determinant formula solutions to the LLN ODEs, and explicit integral formulas for the space-time covariance of the SDEs. These LLN results are contained in Section~\ref{SecLLN} (in particular, the integral formulas are given in Proposition~\ref{PropLLN} while the ODEs are verified in Corollary~\ref{ODEcheck}) while the CLT results are contained in Section~\ref{SecFluc} (in particular the integral formulas for the covariance is given in Proposition~\ref{PropGaussian} while the SDEs are derived in Proposition~\ref{propSDE}). These explicit solutions / covariances are quite remarkable and we do not know of a mechanism which would produce them without having had the prelimiting $q$-Whittaker analogs. Of course, once written down, it is possible to directly verify them.

Let us briefly note that there is another $q\to 1$ limit which was considered in~\cite{BC11,BCF12}. In that case, time is scaled like $\e^{-2}$. In this longer time scale, the triangular array LLN crystalizes with spacing of order $\e^{-1}\ln\e^{-1}$ and fluctuation SDEs of order $\e^{-1}$ around that given by the Whittaker process and SDEs introduced in~\cite{OCon09}.

\begin{figure}
\begin{center}
\includegraphics[width=\textwidth]{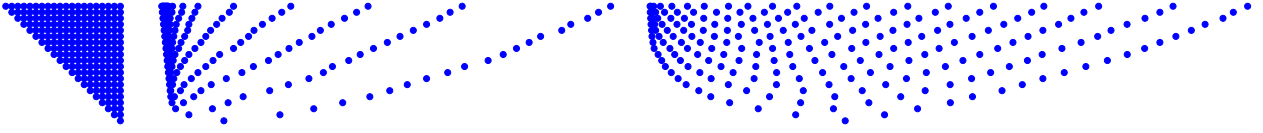}
\caption{Simulation of $q$-Whittaker particle system with $N=20$ particles, $q=e^{-\e}$ and $\e=0.01$. The centered and diffusively scaled particle process $\lambda^{(n)}_k(\e^{-1} \tau)$ is plotted for $\tau=1$ and $\tau=10$.}
\label{configs}
\end{center}
\end{figure}

The system of SDEs derived in this $q\to 1$ limit is given in Proposition~\ref{propSDE}, and exact integral formulas for the covariance of this system is given in Proposition~\ref{PropGaussian}. Given this limiting system, it was again our goal to extract long-time and large-scale ($N\to \infty$) limits and recover the EW equation. The covariance formulas take the form of random matrix type integrals, and unfortunately they were not yet in a form amenable to perform the required asymptotic analysis to reach the goal.

In Section~\ref{seclargetimesim}, we take yet another limit to further reduce the complexity of formulas and SDEs, still keeping the model complex enough to be interesting. We consider the limit of the triangular array for fixed $N$, as time goes to infinity and the SDE solutions (which are already centered) are rescaled diffusively. Calling $\zeta^{(n)}_k(T)$ the resulting particle system, we show in Proposition~\ref{PropLargeLsde} that the $\zeta$ satisfy the system of SDEs (the $W^{(n)}_k(T)$ are independent Brownian motions) indexed by $\{k,n\in \Z: 1\leq k\leq n\}$
$$
d \zeta^{(n)}_k(T) = \left(\frac{-n+1}{T} \zeta^{(n)}_k + \frac{k-1}{T} \zeta^{(n-1)}_{k-1}(T) + \frac{n-k}{T} \zeta^{(n-1)}_{k}(T)\right) dT + dW^{(n)}_k(T)
$$
and in Proposition~\ref{PropLargeL} that their fixed time covariance is given by the following formula which holds for $n_1\geq n_2$, and any $r_i\leq n_i$ for $i=1,2$:
\begin{equation*}
\begin{aligned}
&{\rm Cov}\left(\zeta^{(n_1)}_{n_1}(T)+\ldots+\zeta^{(n_1)}_{n_1-r_1+1}(T);\zeta^{(n_2)}_{n_2}(T)+\ldots+\zeta^{(n_2)}_{n_2-r_2+1}(T)\right)\\
=\,&\frac{\displaystyle\oint \oint \sum_{k=1}^{r_1}\sum_{\ell=1}^{r_2}\frac{1}{z_k-w_\ell} \Big(\prod\limits_{1\leq i<j\leq r_1}(z_j-z_i)^2\prod\limits_{m=1}^{r_1} \frac{e^{T z_m}}{(z_m)^{n_1}}dz_m\Big)
\Big(\prod\limits_{1\leq i<j\leq r_2}(w_j-w_i)^2\prod\limits_{m=1}^{r_2} \frac{e^{T w_m}}{(w_m)^{n_2}}dw_m\Big)}
{\Big(\displaystyle\oint \prod\limits_{1\leq i<j\leq r_1}(z_j-z_i)^2\prod\limits_{m=1}^{r_1} \frac{e^{T z_m}}{(z_m)^{n_1}}dz_m\Big)
\Big(\oint \prod\limits_{1\leq i<j\leq r_2}(w_j-w_i)^2\prod\limits_{m=1}^{r_2} \frac{e^{T w_m}}{(w_m)^{n_2}}dw_m\Big)},
\end{aligned}
\end{equation*}
where the integrals are around $0$ and the $z$-contours contains the $w$-contours.

\begin{figure}
\begin{center}
\psfrag{0}[l]{Time $0$}
\psfrag{s}[l]{Time $S$}
\psfrag{t}[l]{Time $T$}
\psfrag{N}[r]{Size $\sim N$}
\includegraphics[height=4cm]{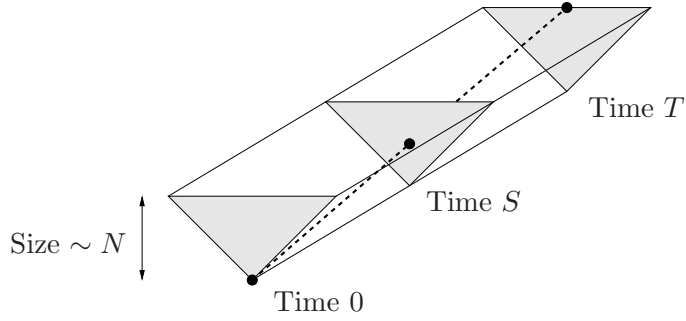}
\caption{A characteristic ray (dotted) comes from the origin at time $0$ at a constant velocity. The dots represent the spatial points associated with time $S$ and time $T$.}
\label{multitime}
\end{center}
\end{figure}

The remainder of the paper focuses on studying the large-scale $N\to \infty$ behavior of this system and its covariance. The SDEs are invariant under diffusive scaling (that is how they arose) so we keep $T$ fixed. Theorem~\ref{thm5.1} and Proposition~\ref{PropCovTwoTimes} contain the limiting covariance for $\zeta$. The asymptotics in Theorem~\ref{thm5.1} deal with the fixed time, two-point covariance and are quite involved and take up a considerable length of the paper. There are essentially three contour integrals and the integrand (in exponential form, centered around the critical points) has one slow-manifold and two fast-manifolds. We first integrate over the fast-manifolds and then consider the asymptotic behavior along the slow manifold. The result of this analysis shows that there are non-trivial fixed time covariances at distances $\Or(N)$ with an explicit limiting covariance function. Proposition~\ref{PropCovTwoTimes} deals with the two-time covariance, and since the temporal dynamics are at this point rather simple, this analysis is fairly straight-forward. It is also due to these simple temporal dynamics that we see that fluctuations never fully decorrelate in time.

Having figured out the asymptotic covariance, we then consider the space-time covariance in $\sqrt{N}$-neighborhoods around rays emanating from the origin at time $T=0$. These rays (illustrated in Figure~\ref{multitime}) are sometimes called characteristics. Fluctuations of space-time points not on such characteristics decorrelate quickly, whereas those on the characteristics display the ``slow decorrelation'' phenomena. This phenomena was observed in $1+1$ growth models~\cite{Fer08,CFP10b}, conjectured (but never proven) for models in $2+1$ dimensions as well~\cite{BF08}. In particular, Corollary~\ref{cor530} shows that
for $d>0$, $a\in (0,1)$, and $T>S>0$ fixed, if we set, for $\eta=(\eta_1,\eta_2)\in \R^2$, $\zeta(T,\eta;N) := N^{1/2} \zeta^{(n)}_k(T)$ where
$$
n=\left(dN+\left(\eta_1\sqrt{(1-a)d}+\eta_2\sqrt{ad}\right)\sqrt{N}\right)T, \qquad k= \left((1-a)dN+\eta_1\sqrt{(1-a)d}\sqrt{N}\right)T,
$$
then for $\eta,\lambda,\mu,\nu\in \R^2$,
\begin{align*}
&\lim_{N\to\infty} {\rm Cov}\left(\zeta(T,\eta;N)-\zeta(T,\lambda;N),\, \zeta(S,\mu;N)-\zeta(S,\nu;N)\right)\\
&\qquad = \frac{S}{\pi d\sqrt{a(1-a)}} \Big(G_{\tau}(|\eta-\mu|)-G_{\tau}(|\eta-\nu|)-G_{\tau}(|\lambda-\mu|)+G_{\tau}(|\lambda-\nu|)\Big).
\end{align*}
The limiting covariance is in terms of
$$\tau=\frac{T-S}{T}, \qquad p_{\tau}(\sigma) = \frac{1}{2\pi \tau} e^{-(\sigma_1^2+\sigma_2^2)/2\tau}, \qquad \Gamma(s,x) = \int_x^{\infty} t^{s-1}e^{-t}dt,$$
and, for $r\in(0,\infty)$,
$$
G_{\tau}(r) = -\Gamma\big(0,\tfrac{r^2}{2\tau}\big) - \ln(r^2) = - \int_{\R^2}p_{\tau}(\sigma)\ln\big(|\sigma-\xi|^2\big)d\sigma
$$
where $\xi=(\xi_1,\xi_2)$, $r= |\xi|$. The reason we take differences of $\zeta$'s above is that the individual covariances grow like $\ln N$, and taking differences cancels this divergence.

The covariance for $\zeta$ can be matched to that of the EW equation (\ref{EWeq}) in equilibrium, which is given by (Section~\ref{secidentify})
$$
{\rm Cov}\big[u(t,x)-u(t,y) ,u(\tilde t,\tilde x) -u(\tilde t,\tilde y)\big] = \frac{1}{4\pi} \Big(G_{t-\tilde t}(|x-\tilde x|) -G_{t-\tilde t}(|x-\tilde y|) - G_{t-\tilde t}(|y-\tilde x|) + G_{t-\tilde t}(|y-\tilde y|)\Big).
$$
In particular, taking $S = \frac{d\sqrt{a(1-a)}}{4}$, $\tilde t= 0$ and $t = \tau$ gives an exact matching of the covariances as we hoped to achieve. It is notable that for $S$ fixed, as $T$ varies from $S$ to infinity, the parameter $\tau$ goes from 0 to $1$ and hence our $\zeta$ process only relates to the EW equation for a fixed time interval $t\in [0,1]$. We do not have an explanation for why this occurs in our limit. We suspect that this apparent time dilation may result from some of the intermediate limits we took of the $q$-Whittaker particle system. The Wolf prediction would seem to suggest that had we taken the $N\to \infty$ of the $q$-Whittaker particle system directly, the EW equation would arise as the limit with time varying as $t\in [0,\infty)$. This is certainly a point which warrants further study.

\subsection*{Acknowledgements}
The authors are grateful to discussions with Julien Dub\'edat, Leonid Petrov, Hao Shen, Fabio Toninelli, and Li-Cheng Tsai. The authors also appreciate funding from the Simons Foundation in the form of the ``The Kardar-Parisi-Zhang Equation and Universality Class'' Simons Symposium, from the Galileo Galilei Institute in the form of the ``Statistical Mechanics, Integrability and Combinatorics'' program, and the Kavli Institute for Theoretical Physics in the form of the ``New approaches to non-equilibrium and random systems: KPZ integrability, universality, applications and experiments'' program. This research was supported in part by the National Science Foundation under grant PHY-1125915.
A. Borodin was partially supported by the NSF grants DMS-1056390 and DMS-1607901, and by a Radcliffe Institute for Advanced Study Fellowship, and a Simons Fellowship.
I. Corwin was partially supported by the NSF through DMS-1208998, Microsoft Research and MIT through the Schramm Memorial Fellowshop, the Clay Mathematics Institute through a Clay Research Fellowship, the Institute Henri Poincar\'e through the Poincar\'e Chair, and the Packard Foundation through a Packard Fellowship for Science and Engineering.
P.L.~Ferrari is supported by the German Research Foundation via the SFB 1060--B04 project.

\section{Notation and background}\label{secnotback}

\subsection{$q$-deformed functions}\label{secqdef}
We will assume throughout that $q\in (0,1)$ and generally associate $q$ with $\e>0$ via $q=e^{-\e}$. The {\it $q$-Pochhammer symbol} is defined as
$$
(a;q)_{n} = \prod_{i=0}^{n-1}(1-q^i a)
$$
with the obvious extension to the infinite product when $n=\infty$. We will make use of the following, readily observed asymptotics that for $a$ fixed, as $\e\to 0$,
$$
\ln (a;e^{-\e})_{\e^{-1}b} \approx \e^{-1} g_a(b) \qquad \textrm{and}\qquad \ln (a;e^{\e})_{\e^{-1}b} \approx \e^{-1} g_a(-b),
$$
where
$$
g_a(b) = \int_{0}^{b} \ln(1-a e^{-s})ds.
$$

A random variable follows the {\it $q$-geometric distribution} with parameters $q,\alpha\in (0,1)$ if it takes values in $s\in \{0,1,\ldots\}$ according to probability
$$
\PP(s) = \frac{\alpha^s}{(q;q)_s} (\alpha;q)_\infty.
$$

Let $c\in \{0,1,\ldots\}\cup \{+\infty\}$ and $s\in \{0,1,\ldots, c\}$. For real parameters $q,\xi,\eta$ define
$$
\varphi_{q,\xi,\eta}(s|y) = \xi^s \frac{(\eta/\xi;q)_s (\xi;q)_{c-s}}{(\eta;q)_c}\, \frac{(q;q)_c}{(q;q)_s (q;q)_{c-s}}
$$
with the obvious extension to $c=\infty$ given by
$$
\varphi_{q,\xi,\eta}(s|\infty) = \xi^s \frac{(\eta/\xi;q)_s (\xi;q)_\infty} {(\eta;q)_{\infty} (q;q)_s}.
$$
For $q\in (0,1)$ and $0\leq \eta\leq \xi<1$, this defines a probability distribution on $s\in \{0,1,\ldots, c\}$ called the {\it $q$-Hahn distribution}~\cite{P13,C14}. We will make use of another set of parameters which also defines a probability distribution using this function. For $q\in (0,1)$ and $a,c\leq b$ nonnegative integers,
$$
\varphi_{q^{-1},q^a,q^b}\big(s|c\big)
$$
defines a probability distribution on $s\in \{0,1,\ldots,c\}$, as observed in~\cite[Section 6.1]{PM15}.

\subsection{Partitions}\label{secpart}

A {\it partition} $\lambda$ is a non-increasing sequence of non-negative integers $\lambda_1\geq \lambda_2\geq \cdots$. The {\it length} of a partition (written as $\ell(\lambda)$) is the number of non-zero parts. We define $|\lambda| = \sum \lambda_i$. A partition $\mu$ {\it interlaces} with $\lambda$ (written as $\mu \preceq \lambda$ or $\lambda \succeq \mu$) if for all $i$, $\lambda_i \geq \mu_i \geq \lambda_{i+1}$. We will consider measures on sequences of interlacing partitions $\bflambda = (\lambda^{(N)}\succeq \lambda^{(N-1)} \succeq \cdots \succeq \lambda^{(1)})$ wherein $\ell\big(\lambda^{(n)}\big)=n$ for $1\leq n\leq N$. For notational convenience, we adopt the convention that $\lambda^{(0)}$ is the empty partition containing only zeros, and $\lambda^{(n)}_k\equiv +\infty$ for $k\leq 0$. In general, we will use a bar (e.g. $\bflambda$, $\bfmu$) to denote a sequence of interlacing partitions such as described above. We call such a sequence $\bflambda$ {\it packed} if $\lambda^{(n)}_k \equiv 0$ for all $1\leq k\leq n \leq N$.

For a collection of variables $\vec{x} = (x_1,x_2,\ldots)$, and any pair of interlacing partitions $\lambda \succeq \mu$, {\it skew $q$-Whittaker symmetric functions} $P_{\lambda/\mu}(\vec{x})$ and $Q_{\lambda/\mu}(\vec{x})$ are symmetric functions in the $\vec{x}$ variables with coefficients which are rational functions of $q$. They are the same as skew Macdonald symmetric functions with the parameter $t=0$ (we will soon use $t$ to represent a time, having nothing to do with the Macdonald $t$ parameter). There exist explicit combinatorial formulas for these functions as well as many beautiful relations and applications. In the present article we quote all necessary results pertaining to these functions. Those unfamiliar with them and interested in learning more can refer to~\cite{Mac79,BC11} for further information and background.

\subsection{$q$-Whittaker processes}
A {\it specialization} of the algebra of symmetric functions $\Sym$ is an algebra homomorphism $\Sym\to \C$. A specialization is {\it $q$-Whittaker non-negative} if it takes non-negative values on all skew $q$-Whittaker $P$ and $Q$ functions. We will work with two types of $q$-Whittaker non-negative specializations -- the {\it alpha}, and the {\it Plancherel} (see~\cite[Section 2.2.1]{BC11} for more on specializations). The alpha specialization is parameterized by a finite collection of positive real $\vec{\alpha}= (\alpha_1,\ldots,\alpha_t)$ and amounts to substituting $x_i=\alpha_i$ for $1\leq i\leq t$ and $x_i=0$ for all other $i$. The Plancherel specialization is parameterized by one positive real $\gamma$ and corresponds to the $t\to \infty$ limit of substituting $x_i= (1-q)\gamma/t$ for $1\leq i\leq t$ and $x_i=0$ for all other $i$. These specializations will be written as $P_{\lambda}(\vec{\alpha})$ or $P_{\lambda}(\gamma)$ (and likewise for $Q$). Equivalently, the specializations can be defined through the values taken by $g_{k}= Q_{(k)}$, the $q$-analog of the complete homogeneous symmetric functions $h_n$, the monomials of which form a linear basis of $\Sym$: for a formal parameter $u$,
\begin{equation}\label{eqPiu}
\Pi(u;\vec{\alpha}):=\sum_{k\geq 0} g_k(\vec{\alpha}) u^n = \prod_{j=1}^{t} \frac{1}{(\alpha_j u;q)_{\infty}},\qquad \textrm{and} \qquad \Pi(u;\gamma):=\sum_{k\geq 0} g_k(\gamma) u^n = e^{\gamma u}.
\end{equation}
We do not consider the {\it beta specialization} (see~\cite[Section 2.2.1]{BC11}) here since it is unclear whether it admits a limit under the scaling we focus on soon. The alpha and Plancherel specializations are $q$-Whittaker nonnegative. Note that it is easy to combine the alpha and Plancherel specializations. We proceed considering each separately, though all subsequent results can also be stated considering both simultaneously.

We now define the {\it $q$-Whittaker process} under the alpha or Plancherel specializations. Consider a sequence of positive reals $\vec{a}=(a_1,\ldots, a_N)$. For the alpha specialization, let $\vec{\alpha}=(\alpha_1,\ldots, \alpha_t)$ be positive reals such that $a_i\alpha_j<1$ for all $i,j$. For the Plancherel specialization let $\gamma$ be a positive real. Denote both of these specialization by the symbol $\rho$. The {\it ascending $q$-Whittaker process}~\cite{BC11} is a measure $\PP_{\vec{a};\rho}$ on interlacing partitions $\bflambda$ defined by
$$
\PP_{\vec{a};\rho}(\bflambda) =\frac{ Q_{\lambda^{(N)}}(\rho) \, P_{\lambda^{(N)}/\lambda^{(N-1)}}(a_N)P_{\lambda^{(N-1)}/\lambda^{(N-2)}}(a_{N-1})\cdots P_{\lambda^{(1)}}(a_1)}{\Pi(\vec{a};\rho)}
$$
where the normalizing constant (i.e. the sum of the numerators over all collections of interlacing partitions) is given by
$$
\Pi(\vec{a};\rho) = \prod_{i=1}^{N} \Pi(a_i;\rho),
$$
with $\rho=\vec{\alpha}$ or $\rho=\gamma$ and $\Pi(u;\rho)$ defined as in (\ref{eqPiu}). Likewise define $\EE_{\vec{a};\rho}$ as the expectation operator with respect to this probability measure. Figure~\ref{Figinterlacing} illustrates how to associate an interlacing triangular array with a sequence of interlacing partitions $\bar{\lambda}$.

\subsection{Integral formulas}\label{secintforms}

Utilizing Macdonald difference operators,~\cite{BC11} provided a general approach to computing expectations of a wide variety of observables with respect to Macdonald measures. This approach was generalized to the full Macdonald process in~\cite{BCGS13}. The integral formulas we recall here follow from the Macdonald parameter $t=0$ degeneration of those results. In particular,
Proposition~\ref{PropTheorem45} follows from~\cite[Theorem 4.5]{BCGS13} whereas Proposition~\ref{PropTheorem46} follows from~\cite[Theorem 4.6]{BCGS13}.

\begin{figure}
\begin{center}
\psfrag{a}[c]{}
\psfrag{zm}[r]{$z_{m,k}$}
\psfrag{zm1}[r]{$z_{m-1,k}$}
\psfrag{z11}[r]{$z_{1,k}$}
\psfrag{zik}[c]{$z_{i,k}$}
\includegraphics[height=5.5cm]{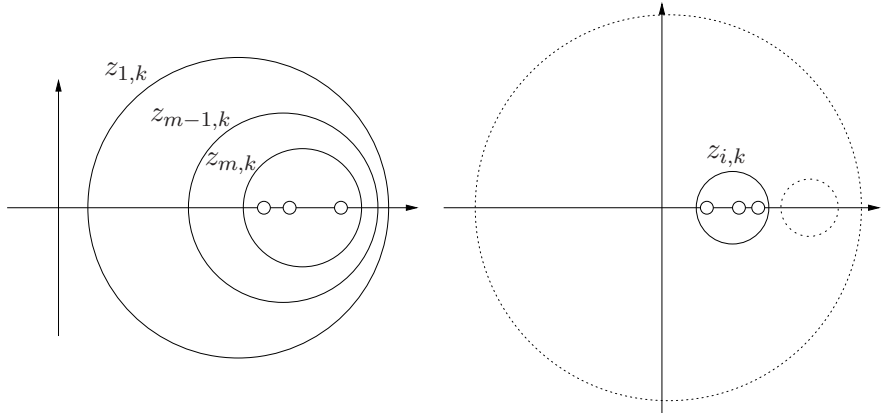}
\caption{Contours for Propositions~\ref{PropTheorem45} (left) and~\ref{PropTheorem46} (right). Left: the inner contour is for all $z_{m,k}$, $1\leq k\leq r_m$. It encloses $a_1,\ldots,a_N$ (represented as white dots). The next contour contains $q$ times this contour and is for all $z_{m-1,k}$, , $1\leq k\leq r_{m-1}$. The final contour, which contains $q$ times all previous ones, is for $z_{1,k}$, , $1\leq k\leq r_1$. Right: all contours are chosen to lie on a small circle enclosing the $a_1,\ldots,a_N$, and not intersecting the image of the contour times $q^{-1}$ (the smaller dotted circle) or the origin. For the alpha specialization we also require the contours be contained in the large dotted circle centered at the origin of radius $q\min_j \alpha_j^{-1}$.}
\label{Figtheorem45contours}
\end{center}
\end{figure}

\begin{proposition}\label{PropTheorem45}
Fix $N\geq 1$, positive reals $\vec{a} = (a_1,\ldots, a_N)$, $m\geq 1$, a sequence of $m$ integers $N\geq n_1\geq n_2\geq \cdots \geq n_m\geq 1$, and $r_1,\ldots, r_m$ such that $0\leq r_i\leq n_i$ for $1\leq i\leq m$. Then for an alpha specialization $\rho = \vec{\alpha}$ with $a_i\alpha_j<1$ for all $i,j$, or for a Plancherel specialization $\rho=\gamma>0$,
\begin{equation}\label{eqMoments}
\begin{aligned}
\EE_{\vec{a};\rho} \Bigg[\prod_{i=1}^{m} q^{\lambda^{(n_i)}_{n_i} +\cdots + \lambda^{(n_i)}_{n_i-r_i+1}}\Bigg]
&= \prod_{i=1}^{m} \frac{1}{(2\pi \I)^{r_i} r_i!} \, \oint \cdots \oint \prod_{1\leq i<j\leq m} \, \prod_{k=1}^{r_i} \prod_{\ell=1}^{r_j} \frac{q(z_{i,k} - z_{j,\ell})}{z_{i,k} - q z_{j,\ell}}\\
&\times \prod_{i=1}^{m} \frac{ (-1)^{\frac{r_i(r_i+1)}{2}} \!\!\!\prod\limits_{1\leq k<\ell\leq r_i} (z_{i,k} - z_{i,\ell})^2}{\prod\limits_{k=1}^{r_i} (z_{i,k})^{r_i}} \, \prod_{k=1}^{r_i} \prod_{\ell=1}^{n_{i}} \frac{-a_\ell}{z_{i,k} - a_\ell}\, \frac{\Pi(q z_{i,k};\rho)}{\Pi(z_{i,k};\rho)} \, dz_{i,k},
\end{aligned}
\end{equation}
with $\Pi(u;\rho)$ as in (\ref{eqPiu}).
We assume that the parameters have been chosen so that the following choice of integration contours exist (see the left-hand side of Figure~\ref{Figtheorem45contours} for an illustration of such contours). The $z_{i,k}$ contour includes all $a_1,\ldots, a_N$ as well as contains the image under multiplication by $q$ of all $z_{j,\ell}$ contours for $j>i$ and arbitrary $\ell$; no contours include 0.
\end{proposition}

\begin{proposition}\label{PropTheorem46}
Fix $N\geq 1$, positive $\vec{a} = (a_1,\ldots, a_N)$, $m\geq 1$, a sequence of $m$ integers $N\geq n_1\geq n_2\geq \cdots \geq n_m\geq 1$, and $r_1,\ldots, r_m$ such that $0\leq r_i\leq n_i$ for $1\leq i\leq m$. Then for an alpha specialization $\rho = \vec{\alpha}$ with $a_i\alpha_j<q^m$ for all $i,j$, or for a Plancherel specialization $\rho=\gamma>0$,
\begin{align*}
\EE_{\vec{a};\rho} \Bigg[\prod_{i=1}^{m} q^{-\lambda^{(n_i)}_{1} - \cdots - \lambda^{(n_i)}_{r_i}}\Bigg]
&= \prod_{i=1}^{m} \frac{1}{(2\pi \I)^{r_i} r_i!} \, \oint \cdots \oint \prod_{1\leq i<j\leq m} \, \prod_{k=1}^{r_i} \prod_{\ell=1}^{r_j} \frac{z_{i,k} - z_{j,\ell}}{z_{i,k} - q^{-1} z_{j,\ell}}\\
&\times \prod_{i=1}^{m} \frac{ (-1)^{\frac{r_i(r_i-1)}{2}} \!\!\!\prod\limits_{1\leq k<\ell\leq r_i} (z_{i,k} - z_{i,\ell})^2}{\prod\limits_{k=1}^{r_i} (z_{i,k})^{r_i}} \, \prod_{k=1}^{r_i} \prod_{\ell=1}^{n_{i}} \frac{z_{i,k}}{z_{i,k} - a_\ell}\, \frac{\Pi(q^{-1} z_{i,k};\rho)}{\Pi(z_{i,k};\rho)} \, dz_{i,k},
\end{align*}
with $\Pi(u;\rho)$ as in (\ref{eqPiu}).
We assume that the parameters have been chosen so that the following choice of integration contours exist (see the right-hand side of Figure~\ref{Figtheorem45contours} for an illustration of such contours). The $z_{i,k}$ contour includes all $a_1,\ldots, a_N$, does not include 0, and is not contained in the image under multiplication by $q^{-1}$ of any of the $z_{j,\ell}$ contours for $j>i$ and arbitrary $\ell$. Additionally, for the alpha specialization, we assume that the contours are contained in the disc centered at the origin of radius $q\min_{j} \alpha_j^{-1}$.
\end{proposition}

\subsection{Dynamics preserving the $q$-Whittaker process}\label{secAlphaDyn}

Besides having many observables whose expectations admit concise formulas, Macdonald processes (or in this case, $q$-Whittaker processes) arise as the fixed time marginals of certain Markov dynamics on interlacing partitions. We will describe two dynamics -- one in discrete time which relates to the alpha specialization and one in continuous time which relates to the Plancherel specialization. In both cases we will initialize $\bflambda(0)$ to the packed configuration (i.e., $\lambda^{(n)}_k \equiv 0$ for $1\leq k\leq n\leq N$). These two so-called ``push-block'' dynamics were introduced in~\cite{BC11}. Their Schur analogs were introduced in~\cite{BF08} based on a construction related to earlier work of~\cite{DF90}. There are other types of discrete and continuous time dynamics which preserve the respective alpha and Plancherel specialized $q$-Whittaker processes, such as considered in~\cite{OP12,PM15,P16}. Appendix~\ref{appother} contains a description of some of these dynamics (including some related to generalizations of the Robinson-Schensted-Knuth or RSK correspondence) as well as some of the parallel analysis in those cases as we preform for the push-block dynamics below.

\subsubsection{Alpha dynamics}\label{alphadyn}
We describe the push-block discrete time alpha dynamic which arises as the degeneration of~\cite[Example 2.3.4 (1)]{BC11} when the Macdonald parameter $t=0$ (not to be confused with the time parameter $t$ we use in what follows). Here the time parameter $t\in \{0,1,\ldots\}$ and we assume $a_i\alpha_j<1$ for all $i,j$. Proposition~\ref{propalphapres} shows that these dynamics preserve the $q$-Whittaker process. For arbitrary $n\geq 1$, given partitions $\mu$ of length $n$ and $\lambda$ of length $n-1$ which interlace as $\mu \succeq \lambda$, define a probability distribution on partitions $\nu$ of length $n$ by
\begin{equation}\label{Paalpha}
P_{a,\alpha}\big(\nu | \lambda,\mu\big) = \begin{cases} \textrm{const} \cdot \phi_{\nu/\mu} \psi_{\nu/\lambda} (a \alpha)^{|\nu|} & \nu\succeq \lambda \textrm{ and } \nu\succeq \mu,\\ 0 & \textrm{otherwise}\end{cases}.
\end{equation}
Here $\textrm{const}$ is a constant (with respect to $\nu$, though depending on all other variables) which makes this a probability distribution on $\nu$ and the factors
\begin{align*}
\phi_{\lambda/\mu} &= \prod_{i=1}^{\ell(\lambda)} \frac{(q^{\lambda_i-\mu_i+1};q)_{\infty} (q^{\mu_i-\lambda_{i+1}+1};q)_{\infty}}{(q;q)_{\infty} (q^{\mu_i-\mu_{i+1}+1};q)_{\infty}},\\
\psi_{\lambda/\mu} &= \prod_{i=1}^{\ell(\lambda)} \frac{(q^{\lambda_i-\mu_i+1};q)_{\infty} (q^{\mu_i-\lambda_{i+1}+1};q)_{\infty}}{(q;q)_{\infty} (q^{\lambda_i-\lambda_{i+1}+1};q)_{\infty}}.
\end{align*}

Given a sequence of interlacing partitions (recall the notation and definitions from Section~\ref{secpart}) $\bflambda= \big(\lambda^{(N)}\succeq \lambda^{(N-1)}\succeq \cdots \succeq \lambda^{(1)}$ with $\ell(\lambda^{(n)}=n$ for all $1\leq n\leq N$, we define a Markov transition matrix to another set of interlacing partitions $\bfmu=\big(\mu^{(N)}\succeq \mu^{(N-1)}\succeq \cdots \succeq \mu^{(1)}$ with $\ell(\mu^{(n)}=n$ for all $1\leq n\leq N$ (recall the convention that $\lambda^{(0)}$ and $\mu^{(0)}$ equal the partition of all zeros) by
$$
P^{\rm{PB}}_{\vec{a};\alpha}\big(\bflambda \to \bfmu\big) = \prod_{n=1}^{N} P_{a_n,\alpha}\big(\mu^{(n)}|\mu^{(n-1)},\lambda^{(n)}\big),
$$
where $P_{a,\alpha}$ is defined in \eqref{Paalpha}.

\begin{proposition}\label{propalphapres}
Define a Markov process indexed by $t$ on interlacing partitions $\bflambda(t)$ with packed initial data and Markov transition between time $t-1$ and $t$ given by $P^{\rm{PB}}_{\vec{a};\alpha_{t}}\big(\bflambda(t-1) \to \bflambda(t)\big)$. Then, for any $t\in \{0,1,\ldots\}$, $\bflambda(t)$ is marginally distributed according to the $q$-Whittaker measure $\PP_{\vec{a};\vec{\alpha}(t)}$ with $\vec{\alpha}(t) = (\alpha_1,\ldots, \alpha_t)$.
\end{proposition}
\begin{proof}
This is a direct consequence of the Macdonald $t=0$ degeneration of the results of~\cite[Section~2.3]{BC11}.
\end{proof}

\subsubsection{Plancherel dynamics}\label{Plandyn}

We define a continuous time push-block dynamic on interlacing partitions $\bflambda(\gamma)$, introduced in~\cite[Definition 3.3.3]{BC11} (and called the $q$-Whittaker growth process therein). In this case, $\gamma>0$ represents time. This dynamic, in fact, arises as the continuous time limit of the alpha push-block dynamics discussed above. Proposition~\ref{propplanpres} shows that this dynamic preserve the $q$-Whittaker process. Given a state $\bflambda$ at time $\gamma$, for $1\leq k\leq n\leq N$, each $\lambda^{(n)}_k$ has its own independent exponential clock with rate
\begin{equation}\label{eqDynamics}
{\rm R}^{(n)}_k=a_n\, \frac{(1-q^{\lambda^{(n-1)}_{k-1} - \lambda^{(n)}_{k}}) (1-q^{\lambda^{(n)}_{k} - \lambda^{(n)}_{k+1}+1})}{1-q^{\lambda^{(n)}_{k} - \lambda^{(n-1)}_{k}+1}}.
\end{equation}
When the $\lambda^{(n)}_{k}$-clock rings we find the longest string $\lambda^{(n)}_k = \lambda^{(n+1)}_k = \cdots = \lambda^{(n+\ell)}_k$ and increase all coordinates in this string by one. Observe that if $\lambda^{(n)}_{k} = \lambda^{(n-1)}_{k-1}$ the jump rate automatically vanishes, hence the interlacing partitions remain interlacing under these dynamics.

\begin{proposition}\label{propplanpres}
Define a continuous time Markov processes $\bflambda(\gamma)$ with the above push-block dynamics started from packed initial data. Then, for any $\gamma>0$, $\bflambda(\gamma)$ is marginally distributed according to the Plancherel specialized $q$-Whittaker process $\PP_{\vec{a};\gamma}$.
\end{proposition}
\begin{proof}
This follows from the results of~\cite[Section 3.3]{BC11}.
\end{proof}

\section{Law of large numbers for the $q$-Whittaker particle system}\label{SecLLN}
As there are many parameters in play, it is possible to consider a variety of different limits of the $q$-Whittaker process and the dynamics which preserve it. In~\cite[Section 4]{BC11} one such limit was considered for the alpha and Plancherel case. That limit involved simultaneously taking $q$ as well as the $a$'s and $\alpha$'s to 1 while taking time to infinity in a suitable manner. That limit led to the Whittaker process as well as certain dynamics preserving it (which turned out to relate to directed polymer models).

Here we consider taking $q=e^{-\e}\to 1$ while fixing the $a$'s and $\alpha$'s. In the alpha case, time remains discrete and is not scaled while in the Plancherel case, time is scaled so that $\gamma = \e^{-1} \tau$ for some new time $\tau$ which stays fixed. In this section we demonstrate how under these scalings, $\lambda^{(n)}_k(t)$ (respectively, $\lambda^{(n)}_{k}(\gamma)$) behave like $\e^{-1} x^{(n)}_k(t)$ (respectively, $\e^{-1} x^{(n)}_{k}(\tau)$) where the $x$'s are deterministic functions. In particular, this explains the curves that are peeling off in Figure~\ref{configs}. We provide integral formulas for exponentials of sums of these $x$'s, as well as equivalent determinantal formulas which are useful in our subsequent analysis of further limits. By appealing to the dynamics described in Section~\ref{secAlphaDyn}, we derive certain difference / differential equations that the $x$'s should satisfy. We do not prove these relations directly from the dynamics, but rather show (for two of them) how one can directly observe that our formulas for the $x$'s satisfy them. Appendix~\ref{appother} contains derivations of other difference / differential equations relates to alternative dynamics defined there in. We do not, in those cases, verify that our formulas indeed satisfy these equations.

In Section~\ref{SecFluc} we probe beyond the law of large number behavior and study the $\e^{-1/2}$ scale fluctuations, deriving a Gaussian limit of the $q$-Whittaker process as well as certain dynamics which act nicely on this Gaussian measure.

In this section we work with both the discrete alpha dynamics and the continuous Plancherel dynamics. In subsequent sections, we only consider limits of the continuous dynamics. We opt to consider the LLN behavior of the discrete dynamics here since it leads to certain formulas which display greater symmetry. 

\subsection{Integral formulas}\label{limint}

\begin{proposition}[Law of large number for the $q$-Whittaker process]\label{PropLLN}
\mbox{}

\noindent{\bf Alpha case:} For $\e>0$, consider the following scalings
\begin{equation}\label{eqScalingLLNalpha}
q=e^{-\e}, \qquad t,\,\, \vec{a}, \,\,\vec{\alpha}\, \textrm{ fixed},\qquad \lambda^{(n)}_{k}(t) = \e^{-1} x^{(n)}_{k}(t;\e).
\end{equation}
Then the following limit in probability exists
\begin{equation*}
\lim_{\e\to 0} x^{(n)}_{k}(t;\e) = x^{(n)}_{k}(t),
\end{equation*}
and the limiting $x^{(n)}_{k}(t)$ satisfies the following (defining) integral formulas
\begin{equation}\label{eq3.3alpha}
e^{-(x^{(n)}_{n}(t) +\cdots + x^{(n)}_{n-r+1}(t))} = \oint\cdots \oint \mathfrak{F}_{\vec{\alpha}(t)}(n,r;z_{1},\ldots, z_{r}) dz_1\cdots dz_{r}
\end{equation}
where, for $\vec{\alpha}(t)=(\alpha_1,\ldots,\alpha_t)$,
\begin{equation}\label{eqnfrakFalpha}
\mathfrak{F}_{\vec{\alpha}(t)}(n,r;z_{1},\ldots,z_{r}) =\frac{(-1)^{\frac{r(r+1)}{2}} \prod\limits_{1\leq k<\ell\leq r} (z_{k} - z_{\ell})^2}{(2\pi \I)^{r} r! \prod\limits_{k=1}^{r} (z_{k})^{r}} \, \prod_{k=1}^{r} \bigg(\prod_{\ell=1}^{n} \frac{a_\ell}{a_\ell-z_{k}}\bigg)\, \prod_{i=1}^{t} (1-\alpha_i z_k),
\end{equation}
and where the integration is along counterclockwise simple loops enclosing $a_1,\ldots,a_n$ but not $0$.

\noindent{\bf Plancherel case:} For $\e>0$, consider the following scalings
\begin{equation}\label{eqScalingLLNPlan}
q=e^{-\e}, \qquad \gamma=\e^{-1}\tau,\qquad \vec{a},\,\, \tau \textrm{ fixed},\qquad \lambda^{(n)}_{k} = \e^{-1} x^{(n)}_{k}(\tau;\e).
\end{equation}
Then the following limit in distribution (and in probability) exists
\begin{equation*}
\lim_{\e\to 0} x^{(n)}_{k}(\tau;\e) = x^{(n)}_{k}(\tau),
\end{equation*}
and the limiting $x^{(n)}_{k}(\tau)$ satisfy the same (defining) integral formulas as in (\ref{eq3.3alpha}) except with $\mathfrak{F}_{\vec{\alpha}(t)}(n,r;z_{1},\ldots,z_{r})$ replaced by $\mathfrak{F}_{\tau}(n,r;z_{1},\ldots,z_{r})$ where the only difference is that in (\ref{eqnfrakFalpha}), we replace the last factor $\prod_{i=1}^{t} (1-\alpha_i z_k)$ with $e^{-z_k \tau}$.
\end{proposition}

\begin{proof}
The idea of the proof is encapsulated in the following example. Consider a sequence of random variables $x_{\e}$ such that $\EE\big[e^{-k x_{\e}}\big]\to e^{-k x}$ for $k=1,2$ and $x$ deterministic. Then $\textrm{var}\big(e^{-x_{\e}}\big)\to 0$ and thus by Chebyshev's inequality, $e^{-x_{\e}}$ converges in probability to the determinstic value $e^{-x}$. This, likewise, implies that $x_{\e}$ converges in probability to $x_{\e}$. We now proceed with the proof.

The proof of both the alpha and Plancherel cases are effectively identical. As such, we will only write the Plancherel case. Under the scaling (\ref{eqScalingLLNPlan}), we have
\begin{equation*}
\EE_{\vec{a};\rho} \Bigg[\prod_{i=1}^{m} q^{\lambda^{(n_i)}_{n_i} +\cdots + \lambda^{(n_i)}_{n_i-r_i+1}}\Bigg]
=\EE_{\vec{a};\rho} \Bigg[\prod_{i=1}^{m} e^{-\left(x^{(n_i)}_{n_i}(\tau;\e) +\cdots + x^{(n_i)}_{n_i-r_i+1}(\tau;\e)\right)}\Bigg],
\end{equation*}
which by the moment formula (\ref{eqMoments}) can be written as a multiple contour integral, where the only $\e$-dependence is in the portion of the integrand given by
\begin{equation*}
\prod_{1\leq i<j\leq m} \, \prod_{k=1}^{r_i} \prod_{\ell=1}^{r_j} \frac{q(z_{i,k} - z_{j,\ell})}{z_{i,k} - q z_{j,\ell}}
 \prod_{i=1}^{m} \prod_{k=1}^{r_i} \frac{\Pi(q z_{i,k};\rho)}{\Pi(z_{i,k};\rho)}.
\end{equation*}
Here, in the Plancherel case, $\rho=\gamma=\e^{-1} \tau$, whereas in the alpha case $\rho= \vec{\alpha}(t)$.

Set $\e_0>0$, then the integration contours in (\ref{eqMoments}) can be chosen to be independent of $\e$ for all $\e\in[0,\e_0]$ by taking nested circles enclosing the given poles. Furthermore, on these given contours, we have the following uniform convergence
\begin{equation*}
\lim_{\e\to 0}\prod_{1\leq i<j\leq m} \, \prod_{k=1}^{r_i} \prod_{\ell=1}^{r_j} \frac{q(z_{i,k} - z_{j,\ell})}{z_{i,k} - q z_{j,\ell}} = 1, \qquad \lim_{\e\to 0}\prod_{i=1}^{m} \prod_{k=1}^{r_i} \frac{\Pi(q z_{i,k};\rho)}{\Pi(z_{i,k};\rho)} = \prod_{i=1}^{m} \prod_{k=1}^{r_i} e^{-z_{k} \tau}.
\end{equation*}
In the alpha case, the uniform convergence still holds, and the term $e^{-z_{k} \tau}$ in the second relation above is replaced by $\prod_{i=1}^{t} (1-\alpha_i z_k)$.

Therefore we have
\begin{equation*}
\EE_{\vec{a};\rho} \Bigg[\prod_{i=1}^{m} e^{-\left(x^{(n_i)}_{n_i}(\tau;\e) +\cdots + x^{(n_i)}_{n_i-r_i+1}(\tau;\e)\right)}\Bigg]
=\prod_{i=1}^{m}e^{-\left(x^{(n_i)}_{n_i}(\tau) +\cdots + x^{(n_i)}_{n_i-r_i+1}(\tau)\right)},
\end{equation*}
i.e., all the mixed moments of the vector $\big\{e^{-\big(x^{(n)}_{n}(\tau;\e) +\cdots + x^{(n)}_{n-r+1}(\tau;\e)\big)}\big\}_{1\leq r \leq n \leq N}$ converge to the mixed moments of the constant vector $\big\{e^{-\big(x^{(n)}_{n}(\tau) +\cdots + x^{(n)}_{n-r+1}(\tau)\big)}\big\}_{1\leq r \leq n \leq N}$. In other words, this implies that all of the variances converge to zero. Hence, by Chebyshev's inequality, we conclude that each random variable $e^{-\big(x^{(n)}_{n}(\tau;\e) +\cdots + x^{(n)}_{n-r+1}(\tau;\e)\big)}$ converges in probability to the deterministic limit of its mean $e^{-\big(x^{(n)}_{n}(\tau) +\cdots + x^{(n)}_{n-r+1}(\tau)\big)}$. This proves \eqref{eqnfrakFalpha} (in the Plancherel case). The convergence of the $x^{(n)}_k(\tau;\e)$ to the deterministic limits $x^{(n)}_k(\tau)$ readily follows from this.
\end{proof}

Note that the above result was stated for a single time $\tau$. However, it is easily extended to hold for all $\tau$ in an interval. This is because the $x^{(n)}_k(\tau;\e)$ are weakly increasing with $\tau$ and the limiting $x^{(n)}_k(\tau)$ is continuous in $\tau$. The convergence in probability can be extended to hold for any finite set of $\tau$'s and in this way we can show that the $L^\infty$ norm of the difference between $x^{(n)}_k(\tau;\e)$ and $x^{(n)}_k(\tau)$ as $\tau$ varies in an interval must go to zero.

\subsection{Determinant formulas}\label{sectDetFormulas}
Our integral expressions for the law of large number (\ref{eq3.3alpha}) can be rewritten with the help of the Cauchy-Binet formula as determinants. These formulas will be useful in subsequent asymptotics -- in particular the proof of Proposition~\ref{PropLargeLsde}.
Appendix~\ref{appother} contains a positivity results in the particular case when $a_1,\ldots, a_N=1$ which relates these determinants to certain non-intersecting lattice path partition functions.

\begin{lemma}\label{LemmaLLNdeterminant}
The alpha case law of large numbers (\ref{eq3.3alpha}) can be rewritten as
\begin{equation}\label{eq3.9}
e^{-(x^{(n)}_{n}(t) +\cdots + x^{(n)}_{n-r+1}(t))} = \det\big[A(i-j)\big]_{i,j=1}^r = (-1)^{r(r+1)/2}
\det\left[\tilde A_{i,j}\right]_{i,j=1}^r
\end{equation}
via function
\begin{equation}\label{eq3.10}
A(s)=\frac{-1}{2\pi\I}\oint z^s \prod_{\ell=1}^n\frac{a_\ell}{a_\ell-z}\,\prod_{i=1}^{t}(1-\alpha_i z)\, \frac{dz}{z},
\end{equation}
with the integral along a contour containing the $a_1,\ldots,a_N$ but not 0, and the matrix entries
\begin{equation}\label{eq3.11}
\tilde A_{i,j}=\frac{1}{2\pi\I}\oint p_{i-1}(z) p_{j-1}(z)\prod_{\ell=1}^n\frac{a_\ell}{a_\ell-z}\, \prod_{i=1}^{t}(1-\alpha_i z) \frac{dz}{z^r},
\end{equation}
where the $p_j(z)$'s are any monic polynomials of degree $j$, and the contour is the same as above.

In the Plancherel case, the law of large numbers with $x^{(n)}_{k}(\tau)$ replacing $x^{(n)}_{k}(t)$ is given in the same form, except with the term $\prod_{i=1}^{t}(1-\alpha_i z)$ in \eqref{eq3.10} and \eqref{eq3.11} replaced by $e^{-\tau z}$.
\end{lemma}
\begin{proof}
Since the alpha and Plancherel cases are quite similar, we only treat the alpha case presently.
Let us first prove the equality with the determinant involving the kernel in (\ref{eq3.10}).
Using the identity
\begin{equation}\label{eq3.13}
\prod_{1\leq k<\ell\leq r}\left(\frac1{z_k}-\frac1{z_\ell}\right) =\frac{(-1)^{r(r-1)/2}\prod_{1\leq k<\ell\leq r}(z_k-z_\ell)}{\prod_{k=1}^r (z_k)^{r-1}}
\end{equation}
and setting $w_n(z):=\prod_{i=1}^{t} (1-\alpha_i z) \frac{1}{z}\prod_{\ell=1}^n\frac{a_\ell}{a_\ell-z}$ we have
\begin{equation*}
\begin{aligned}
e^{-(x^{(n)}_{n}(t) +\cdots + x^{(n)}_{n-r+1}(t))} &=
\frac{(-1)^{r}}{(2\pi \I)^{r} r!}\oint \cdots \oint \prod_{1\leq k<\ell\leq r} (z_{k} - z_{\ell})(z_k^{-1}-z_\ell^{-1}) \prod_{k=1}^r w_n(z_k) dz_k\\
&=\det\left[
\frac{-1}{2\pi \I}\oint w_n(z) z^{i-j} dz
\right]_{i,j=1}^r
\end{aligned}
\end{equation*}
where in the second equality we used the Cauchy-Binet identity. The kernel of this determinant matches that in (\ref{eq3.10}). Turning to the second determinant expression,
let $\tilde w_n(z)=\prod_{i=1}^{t}(1-\alpha_i t)\frac{1}{z^r}\prod_{\ell=1}^n\frac{a_\ell}{a_\ell-z}$ and let $p_j(z)=z^j+...$ be any polynomial of degree $j$ with leading coefficient $z^j$. Then $\prod_{1\leq k<\ell\leq r}(z_\ell-z_k)=\det\big[p_{j-1}(z_i)\big]_{i,j=1}^r$.
The Cauchy-Binet identity gives
\begin{equation*}
e^{-(x^{(n)}_{n}(t) +\cdots + x^{(n)}_{n-r+1}(t))} =
(-1)^{r(r+1)/2}\det\left[
\frac{1}{2\pi\I}\oint dz \tilde w_n(z) p_{i-1}(z)p_{j-1}(z)
\right]_{i,j=1}^r,
\end{equation*}
completing the proof.
\end{proof}
In the case that $a_1=\cdots= a_N=1$, we provide an even more explicit determinant formula. For the statement of the below corollary, we need one piece of notation. For two vectors $\vec{b}=(b_{1},\ldots, b_t)$ and $\vec{c}=(c_1,\ldots, c_t)$, define $e_{i}(\vec{b};\vec{c})$, for $1\leq i\leq t$ by the equality
$$
\prod_{i=1}^{t}(b_i + c_i z) = \sum_{i=0}^{t} e_i(\vec{b};\vec{c}) z^i.
$$
The $e_{i}(\vec{b};\vec{c})$ are separately symmetric in both the $\vec{b}$ and $\vec{c}$ variables.

\begin{corollary}\label{CorPolynomials}
For the alpha case, define (recall $(x)_n=x(x+1)\cdots (x+n-1)$ for $n>0$, $(x)_0\equiv 1$, and $\frac{1}{(n-1)!}=0$ for $n\leq 0$)
\begin{equation*}
G_{r,t}(m):=\sum_{i=0}^t \frac{e_{i}(\vec{1}-\vec{\alpha};\vec{\alpha}) (r)_{m-i-1}}{(m-i-1)!},
\end{equation*}
which is $0$ if $m\leq 0$. Then it holds
\begin{equation}\label{eq3.16}
e^{-(x^{(n)}_{n}(t) +\cdots + x^{(n)}_{n-r+1}(t))} = \det\Big[G_{r,t}(n+1-r+j-i)\Big]_{i,j=1}^r.
\end{equation}

In the Plancherel case, define
\begin{equation*}
G_{r,\tau}(m):=\sum_{i\geq 0} \frac{\tau^i (r)_{m-i-1}}{i!(m-i-1)!}.
\end{equation*}
Then it holds
\begin{equation}\label{eq3.16p}
e^{-(x^{(n)}_{n}(\tau) +\cdots + x^{(n)}_{n-r+1}(\tau))} = e^{-\tau r} \det\Big[G_{r,\tau}(n+1-r+j-i)\Big]_{i,j=1}^r.
\end{equation}
\end{corollary}
\begin{proof}
Let us describe how to derive the alpha case result (\ref{eq3.16}). With $a_1=\ldots=a_N=1$, let us choose $p_k(z)=(z-1)^k$ in (\ref{eq3.11}). By the change of variables $w=1-z$ and setting $m=n+2-i-j$ we then obtain
\begin{equation*}
\begin{aligned}
\tilde A_{i,j}&=\frac{(-1)^{i+j-1}}{2\pi\I}\oint \frac{\prod_{\ell=1}^{t} (1-\alpha_\ell + \alpha_\ell w)}{(1-w)^r}\frac{dw}{w^{m+1}}=\sum_{\ell=0}^{t}\sum_{b\geq 0} e_{\ell}(\vec{1}-\vec{\alpha};\vec{\alpha})\frac{(r)_b}{b!} \frac{(-1)^{i+j-1}}{2\pi\I}\oint\frac{dw}{w^{m+1-\ell-b}}\\
&=(-1)^{i+j-1}\sum_{\ell=0}^{t} \frac{e_{\ell}(\vec{1}-\vec{\alpha};\vec{\alpha}) (r)_{m-\ell-1}}{(m-\ell-1)!},
\end{aligned}
\end{equation*}
where the contour integrals in the first line are along small circles enclosing $0$.
Plugging this into (\ref{eq3.10}), taking out the $(-1)^{i}$ factor and reshuffling the columns with $j\to r+1-j$ we get (\ref{eq3.16}). The derivation of (\ref{eq3.16p}) follows similarly.
\end{proof}

\subsection{Derivation of ODEs satisfied by the law of large numbers}\label{secderodes}
In Section~\ref{limint} we determined the law of large numbers for the alpha and Plancherel $q$-Whittaker processes under certain specified scalings. In Section~\ref{secAlphaDyn} we recalled various types of Markov dynamics which preserve these classes of $q$-Whittaker processes. Therefore, it is natural to hope that taking a suitable limit of the Markov dynamics will lead us to certain deterministic ODEs which the law of large numbers satisfy. (Later in Section~\ref{SecFluc} we will push this further to consider fluctuations as well.) We presently provide heuristic (i.e. without proof) derivations of the ODEs that we expect the law of large numbers satisfies. In Section~\ref{secdirver}, in the case of the push-block ODEs, we provide direct verification that the formulas from Section~\ref{limint} do satisfy these equations. We do not pursue verifying the other cases.

\subsubsection{Alpha ODEs}\label{secalphaode}
We consider the limiting difference equations which follow from the alpha dynamics introduced earlier in Section~\ref{alphadyn}.
First consider how the dynamics observed by $\lambda^{(1)}_1$ behaves as $\e\to 0$ and with the scalings given in (\ref{eqScalingLLNalpha}). In a single time step, the change $\lambda^{(1)}_1(t) - \lambda^{(1)}_1(t-1)$ is distributed according to a $q$-geometric distribution with parameter $\alpha_t a_1$. Thus, it suffices to consider how such a distribution behaves under our scalings. Let us call $b\in (0,1)$ the fixed $q$-geometric parameter. Then, for $q=e^{-\e}$ and $x$ such that $\e^{-1}x$ is a non-negative integer, we have
$$
\PP(X=\e^{-1}x) = \frac{b^{\e^{-1} x}(b;e^{-\e})_{\infty}}{(e^{-\e};e^{-\e})_{\e^{-1}x}}.
$$
The right-hand side will be maximal for $x$ such that
$$
x\ln b -\ln (e^{-\e};e^{-\e})_{\e^{-1}x}
$$
is maximal. As observed in Section~\ref{secqdef}, as $\e\to 0$,
$\ln (e^{-\e};e^{-\e})_{\e^{-1}x} \approx \e^{-1}g_1(x)$.
Thus, $\PP(X=\e^{-1}x)$ should be maximal around the $x$ which maximizes
$ x\ln b - g_1(x)$.
Using the fact that
$
\frac{d}{dx} g_a(x) = \ln(1-a e^{-x})
$
we readily deduce that the maximizing $x$ is the solution to
$$
\ln b = \ln(1-e^{-x}),
$$
or in other words $x= -\ln(1-b)$. From this reasoning we expect the difference equation
\begin{equation}\label{x1eqn}
x^{(1)}_1(t) - x^{(1)}_1(t-1) = -\ln(1-\alpha_t a_1).
\end{equation}
Proving the above law of large numbers for the $q$-geometric distribution should be quite doable, moreover one expects that looking to higher order Taylor approximation terms, we should see a Gaussian fluctuations of order $\e^{-1/2}$. This would be relevant to the fluctuations of the alpha dynamics (though we do not pursue these any further herein).

Let us turn to the general case of $\lambda^{(n)}_k(t)$. Given $x^{(n-1)}(t)$ and $x^{(n)}(t-1)$ we seek to maximize over all $x^{(n)}(t)$ the log of the transition probability (recall the relation between $x$'s and $\lambda$'s from (\ref{eqScalingLLNalpha}) and the definition of the probability distribution used below which is given in (\ref{Paalpha}))
$$
\ln P_{a_n,\alpha_t}(\lambda^{(n)}(t)|\lambda^{(n-1)}(t),\lambda^{(n)}(t-1)),
$$
as $\e\to 0$. In fact, we really only need to identify the argmax of this quantity. In the same spirit as above, we can express the maximization problem as equivalent to finding the $\e\to 0$ limiting argmax over $\{x^{(n)}_{k}(t)\}_{k=1}^{n}$ of
\begin{align*}
&\sum_{k=1}^{n} \Big[g_1\big(x^{(n)}_{k}(t-1) -x^{(n)}_{k+1}(t-1)\big)
- g_1\big(x^{(n)}_{k}(t) -x^{(n)}_{k}(t-1)\big)
- g_1\big(x^{(n)}_{k}(t-1) -x^{(n)}_{k+1}(t)\big)\Big]
\\
&+\sum_{k=1}^{n-1} \Big[g_1\big(x^{(n)}_{k}(t) -x^{(n)}_{k+1}(t)\big)
- g_1\big(x^{(n)}_{k}(t) -x^{(n-1)}_{k}(t)\big)
- g_1\big(x^{(n-1)}_{k}(t) -x^{(n)}_{k+1}(t)\big)\Big]
\,+ \ln(\alpha_t a_n) \sum_{k=1}^{n} x^{(n)}_{k}(t).
\end{align*}
Differentiating in each $x^{(n)}_k(t)$ as $1\leq k\leq n$ varies yields a collection of critical point equations which, after introducing the notation
$$
y^{(n)}_k(t) = e^{-x^{(n)}_k(t)}
$$
(we assume the convention that $y^{(n)}_k(t)\equiv 1$ if $k>n$, and $y^{(n)}_k(t)\equiv 0$ if $k\leq 0$) becomes
\begin{equation}\label{eqsymstar}
a_n \frac{\left(1-\frac{y^{(n-1)}_{k-1}(t)}{y^{(n)}_k(t)}\right)\left(1-\frac{y^{(n)}_k(t)}{y^{(n)}_{k+1}(t)}\right)}
{1-\frac{y^{(n)}_k(t)}{y^{(n-1)}_k(t)}} =\frac{1}{\alpha_t} \frac{\left(1-\frac{y^{(n)}_k(t)}{y^{(n)}_k(t-1)}\right)\left(1-\frac{y^{(n)}_{k-1}(t)}{y^{(n)}_k(t)}\right)}{ 1-\frac{y^{(n)}_{k-1}(t-1)}{y^{(n)}_k(t)}}.
\end{equation}

In Section~\ref{secdirver} we provide a direct verification and proof that the integral formulas from Section~\ref{limint} satisfy the above equation. Notice that when $n=1$ this reduces to
$$
a_1\alpha_t = 1- \frac{y^{(1)}_1(t)}{y^{(1)}_1(t-1)},
$$
which is equivalent to (\ref{x1eqn}) derived above.

\subsubsection{Plancherel ODEs}\label{secplancode}
We consider the limiting ODE which follow from the push-block Plancherel dynamic introduced earlier in Section~\ref{Plandyn}.
Recall the scalings given in (\ref{eqScalingLLNPlan}) and consider any particle $\lambda^{(n)}_k(t)$. The rate at which it increases by one is given by $R^{(n)}_k$ in (\ref{eqDynamics}). Substituting the scalings from (\ref{eqScalingLLNPlan}) we get a rate of (recall that $\lambda^{(n)}_k(\e^{-1}\tau) = \e^{-1}x^{(n)}_k(\tau;\e)$)
\begin{equation}\label{eqnresc}
a_n\, \frac{(1-e^{x^{(n)}_{k}(\tau;\e)-x^{(n-1)}_{k-1}(\tau;\e)}) (1-q e^{x^{(n)}_{k+1}(\tau;\e)-x^{(n)}_{k}(\tau;\e)})}{1-q e^{x^{(n-1)}_{k}(\tau;\e)-x^{(n)}_{k}(\tau;\e)}}.
\end{equation}
If we assume that all $x^{(n)}_k(\tau;\e)$ vary only on a time scale of order one in $\tau$, then in time of order $\e^{-1}$, this rate will remain essentially unchanged over a time interval of length $\e^{-1}\delta \tau$ for some small $\delta \tau$. Hence the law of large numbers for Poisson random variables suggests that the change in $x^{(n)}_k(\tau;\e)$ over that $\e^{-1}\delta\tau$ time interval is exactly given by (\ref{eqnresc}). As $\e\to 0$ this suggests the limiting system of ODEs (with the conventions $x^{(n)}_k(\tau) \equiv 0$ if $k>n$, and $x^{(n)}_k(\tau)\equiv +\infty$ if $k\leq 0$)
\begin{equation}\label{eqnpushode}
\frac{d x^{(n)}_k(\tau)}{d\tau}= a_n\frac{\left(1-e^{x^{(n)}_{k}(\tau)-x^{(n-1)}_{k-1}(\tau)}\right) \left(1-e^{x^{(n)}_{k+1}(\tau)-x^{(n)}_k(\tau)}\right)}{1-e^{x^{(n-1)}_k(\tau)-x^{(n)}_k(\tau)}}.
\end{equation}
In Section~\ref{secdirver} we provide a direct verification and proof that the integral formulas from Section~\ref{limint} satisfies the above equation.

\subsection{Direct verification of the push-block ODEs}\label{secdirver}

In this section we give a direct verification that the integral formulas from Section~\ref{limint} satisfy the push-block ODEs heuristically derived above in the alpha and Plancherel cases. From (\ref{eq3.9}) we have that $e^{-x^{(n)}_{n-r+1}(t)}$ is a ratio of two Toeplitz determinants with same symbol but different sizes. The following result is a general identity concerning such ratios of specific types of Toeplitz determinants.
\begin{proposition}\label{PropMainRelLLN}
For a function $\varphi(z)$ of a complex variable, denote by
\begin{equation*}
D_r(\varphi)=\det\left[\varphi_{i-j}\right]_{i,j=1}^r,\quad \varphi_k=\frac{1}{2\pi\I}\oint_{\Gamma} \varphi(z)\frac{dz}{z^{k+1}},
\end{equation*}
where $\Gamma$ is any fixed contour.
For a function $F(z)$ of a complex variable, define
\begin{equation*}
y^{(n)}_k(t)=-\frac{D_{n+1-k}\left(\prod_{\ell=1}^{n} \frac{a_\ell}{a_\ell-z} \prod_{i=1}^t (1-\alpha_i z) F(z)\right)}{D_{n-k}\left(\prod_{\ell=1}^{n} \frac{a_\ell}{a_\ell-z} \prod_{i=1}^t (1-\alpha_i z) F(z)\right)}.
\end{equation*}
Then, (\ref{eqsymstar}) holds with the above redefinition of the $a$, $\alpha$, and $y$ parameters / functions.
\end{proposition}

We delay the proof of this result for a moment and record a corollary which shows that the limiting law of large numbers $x^{(n)}_k(t)$ computed in Proposition~\ref{PropLLN} satisfy the heuristically derived ODEs from Section~\ref{secderodes}. We do not speculate here on whether these are the unique solutions to these ODEs, or if some addition conditions are required to ensure uniqueness.

\begin{corollary}\label{ODEcheck}
Recalling the respective Plancherel and alpha limiting law of large numbers $x^{(n)}_k(t)$ and $x^{(n)}_k(\tau)$ from Proposition~\ref{PropLLN}, and defining
$$
y^{(n)}_k(t) = e^{-x^{(n)}_k(t)}, \qquad y^{(n)}_k(\tau) = e^{-x^{(n)}_k(\tau)},
$$
we have that $y^{(n)}_k(t)$ satisfies (\ref{eqsymstar}),
and that $y^{(n)}_k(\tau)$ satisfies (\ref{eqnpushode}).
Recall the conventions that $x^{(n)}_k=0$ for $k>n$ and $x^{(n)}_k=+\infty$ for $k\leq 0$.
\end{corollary}

\begin{proof}
The fact that $y^{(n)}_k(t)$ satisfies (\ref{eqsymstar}) follows immediately in light of Lemma~\ref{LemmaLLNdeterminant} and Proposition~\ref{PropMainRelLLN}. The fact that $y^{(n)}_k(\tau)$ satisfies (\ref{eqnpushode}) follows from a simple limiting procedure. Let $\alpha_n=\e\to 0$. Then $1-\frac{y^{(n)}_{k-1}(t-1)}{y^{(n)}_{k}(t-1)}$ on the right-hand side of (\ref{eqsymstar}) cancels with the corresponding time $t$ term (still on the right-hand side). The remaining right-hand side term
$1-\frac{y^{(n)}_{k}(t)}{y^{(n)}_{k}(t-1)}$ limits to minus the logarithmic derivative of $y^{(n)}_k(t)$ in time, thus yielding the resulting left-hand side of (\ref{eqnpushode}). The left-hand side of (\ref{eqsymstar}) does not change and becomes the right-hand side of (\ref{eqnpushode}).
\end{proof}

To prove Proposition~\ref{PropMainRelLLN} we need two identities for Toeplitz matrices. These are derived by the following linear algebraic identities. The proof will be presented in Appendix~\ref{AppMatrixIdentites}.
\begin{proposition}\label{propMatrixIdentities}
Let $B$ and $C$ be two matrices of sizes at least $(M+2)\times(M+2)$ and assume that $B$ and $C$ are related by
\begin{equation}\label{eqApp11}
C_{i,j}=B_{i,j}+\gamma B_{i,j+1},\quad 1\leq i,j\leq M+1.
\end{equation}
Then
\begin{equation}\label{eqApp12}
\begin{aligned}
\det[B_{i,j}]_{i,j=1}^{M+1} \det[C_{i+1,j+1}]_{i,j=1}^M - \det[C_{i,j}]_{i,j=1}^{M+1} \det[B_{i+1,j+1}]_{i,j=1}^M& \\
+ \gamma \det[B_{i,j+1}]_{i,j=1}^{M+1} \det[C_{i+1,j}]_{i,j=1}^M &= 0,
\end{aligned}
\end{equation}
and
\begin{equation}\label{eqApp13}
\begin{aligned}
\det[B_{i,j}]_{i,j=1}^{M+1} \det[C_{i+1,j+1}]_{i,j=1}^{M-1} - \det[C_{i,j}]_{i,j=1}^M \det[B_{i+1,j+1}]_{i,j=1}^M&\\
+\det[B_{i,j+1}]_{i,j=1}^M \det[C_{i+1,j}]_{i,j=1}^M &=0.
\end{aligned}
\end{equation}
\end{proposition}

Examples of matrices satisfying (\ref{eqApp11}) are Toeplitz matrices with symbols $\varphi(z)$ for $B$ and $(1+\gamma z)\varphi(z)$ for $C$. Then (\ref{eqApp12}) and (\ref{eqApp13}) become the following identities.
\begin{lemma}\label{LemToeplitzIdentities}
It holds
\begin{equation}\label{eqApp14}
\begin{aligned}
D_{M+1}\big(\varphi(z)\big) D_M\big((1+\gamma z)\varphi(z)\big)- D_{M+1}\big((1+\gamma z)\varphi(z)\big) D_M\big(\varphi(z)\big)&\\
+\gamma D_{M+1}\big(z\varphi(z)\big) D_M\big((1+\gamma z)z^{-1} \varphi(z)\big) &=0,
\end{aligned}
\end{equation}
and
\begin{equation}\label{eqApp15}
\begin{aligned}
D_{M+1}\big(\varphi(z)\big) D_{M-1}\big((1+\gamma z)\varphi(z)\big)- D_{M}\big((1+\gamma z)\varphi(z)\big) D_M\big(\varphi(z)\big)&\\
+\gamma D_{M}\big(z\varphi(z)\big) D_M\big((1+\gamma z)z^{-1} \varphi(z)\big) &=0.
\end{aligned}
\end{equation}
\end{lemma}

\begin{proof}[Proof of Proposition~\ref{PropMainRelLLN}]
Let us use the notations
\begin{equation*}
\begin{aligned}
f(z)&=\prod_{\ell=1}^{n} \frac{a_\ell}{a_\ell-z} \prod_{i=1}^{t-1} (1-\alpha_i z) F(z), \qquad g(z)=(1-\alpha_t z) f(z),\\
h(z)&=\left(\frac{a_n}{a_n-z}\right)^{-1},\qquad g(z)=\left(\frac{a_n}{a_n-z}\right)^{-1} (1-\alpha_t z) f(z).
\end{aligned}
\end{equation*}

We factorize the 6 terms in (\ref{eqsymstar}) using either (\ref{eqApp14}) or (\ref{eqApp15}). The terms in the lhs. of (\ref{eqsymstar}) are
\begin{equation*}
\begin{aligned}
a_n\left(1-\frac{y^{(n-1)}_{k-1}(t)}{y^{(n)}_k(t)}\right) = a_n &\frac{D_{n-k}(h) D_{n+1-k}(g) - D_{n+1-k}(h) D_{n-k}(g)}{D_{n-k}(h)D_{n+1-k}(g)}\\
\stackrel{(\ref{eqApp14})}{=}&
\frac{D_{n+1-k}(z g) D_{n-k}(z^{-1}h)}{D_{n+1-k}(g) D_{n-k}(h)}
\end{aligned}
\end{equation*}
and
\begin{equation*}
1-\frac{y^{(n)}_k(t)}{y^{(n)}_{k+1}(t)} = \frac{(D_{n-k}(g))^2-D_{n+1-k}(g) D_{n-1-k}(g)}{(D_{n-k}(g))^2} \stackrel{(\ref{eqApp13})}{=} \frac{D_{n-k}(z g) D_{n-k}(z^{-1} g)}{(D_{n-k}(g))^2}
\end{equation*}
and
\begin{equation*}
\frac{1}{1-\frac{y^{(n)}_k(t)}{y^{(n-1)}_k(t)}} = \frac{D_{n-k}(g) D_{n-k}(h)}{D_{n-k}(g) D_{n-k}(h)-D_{n+1-k}(g) D_{n-1-k}(h)}\stackrel{(\ref{eqApp15})}{=}\frac{D_{n-k}(g) D_{n-k}(h)}{D_{n-k}(zg) D_{n-k}(z^{-1} h)}.
\end{equation*}
The terms in the rhs. of (\ref{eqsymstar}) are
\begin{equation*}
\begin{aligned}
\alpha_t^{-1} \left(1-\frac{y^{(n)}_k(t)}{y^{(n)}_k(t-1)}\right) = \alpha_t^{-1} &\frac{D_{n-k}(g) D_{n+1-k}(f)-D_{n+1-k}(g) D_{n-k}(f)}{D_{n-k}(g) D_{n+1-k}(f)} \\ \stackrel{(\ref{eqApp14})}{=} &\frac{D_{n+1-k}(zf)D_{n-k}(z^{-1}g)}{D_{n-k}(g) D_{n+1-k}(f)}
\end{aligned}
\end{equation*}
and
\begin{equation*}
1-\frac{y^{(n)}_{k-1}(t)}{y^{(n)}_k(t)}=\frac{(D_{n+1-k}(g)^2-D_{n+2-k}(g) D_{n-k}(g)}{(D_{n+1-k}(g))^2} \stackrel{(\ref{eqApp13})}{=} \frac{D_{n+1-k}(z g) D_{n+1-k}(z^{-1} g)}{(D_{n+1-k}(g))^2}
\end{equation*}
and
\begin{equation*}
\begin{aligned}
\frac{1}{1-\frac{y^{(n)}_{k-1}(t-1)}{y^{(n)}_k(t)}} = \frac{D_{n+1-k}(f)D_{n+1-k}(g)}{D_{n+1-k}(f) D_{n+1-k}(g)-D_{n+2-k}(f)D_{n-k}(g)} \stackrel{(\ref{eqApp15})}{=}\frac{D_{n+1-k}(f) D_{n+1-k}(g)}{D_{n+1-k}(zf) D_{n+1-k}(z^{-1} g)}.
\end{aligned}
\end{equation*}
Multiplying the above expressions one immediately gets (\ref{eqsymstar}).
\end{proof}

\section{Stochastic differential equation limit for the fluctuations}\label{SecFluc}

From this point on we focus entirely on the Plancherel push-block dynamics. In Sections~\ref{secderodes} and~\ref{secdirver} we derived (based on these dynamics) and then directly verified ODEs satisfied by the law of large numbers for our $q$-Whittaker processes under certain prescribed scaling. In this section we probe the fluctuations around the law of large number behavior. We start by determining the Gaussian fluctuation limit for the $q$-Whittaker process and then describe the system of SDEs which come from the Plancherel push-block dynamics and which preserves the covariance structure of this Gaussian fluctuation limit.

Let us start by fixing the scalings considered hereafter to be
\begin{equation}\label{eqScalingGaussian}
q=e^{-\e}, \qquad t=\e^{-1}\tau,\qquad \vec{a} \textrm{ fixed},\qquad \lambda^{(n)}_{k} = \e^{-1} x^{(n)}_{k}(\tau) + \e^{-1/2} \xi^{(n)}_{k}(\tau;\e)
\end{equation}
with $x^{(n)}_{k}(\tau)$ given by the results of Proposition~\ref{PropLLN}. In other words, $x^{(n)}_{k}(\tau)$ is the law of large numbers (on the $\e^{-1}$ scale) and $\xi^{(n)}_k(\tau;\e)$ is the fluctuation around it (on the $\e^{-1/2}$ scale).

\subsection{Fixed time Gaussian limit}
We start by proving a Gaussian limit to the $q$-Whittaker process under the scaling (\ref{eqScalingGaussian}) for $\tau>0$ fixed.

Let us first explain the strategy of the proof. Using the scaling (\ref{eqScalingGaussian}) one has
\begin{equation}\label{eq4.1}
\begin{aligned}
&\EE_{\vec{a};\rho} \Bigg[\prod_{i=1}^{m} \left(q^{\lambda^{(n_i)}_{n_i} +\cdots + \lambda^{(n_i)}_{n_i-r_i+1}}
-\EE_{\vec{a};\rho}\Big[ q^{\lambda^{(n_i)}_{n_i} +\cdots + \lambda^{(n_i)}_{n_i-r_i+1}}\Big]\right)\Bigg] \\
=& \prod_{i=1}^m e^{-(x^{(n_i)}_{n_i} +\cdots + x^{(n_i)}_{n_i-r_i+1})} \EE_{\vec{a};\rho} \Bigg[\prod_{i=1}^{m} \left(e^{-\sqrt{\e}(\xi^{(n_i)}_{n_i} +\cdots + \xi^{(n_i)}_{n_i-r_i+1})}
-\EE_{\vec{a};\rho}\Big[ e^{-\sqrt{\e}(\xi^{(n_i)}_{n_i} +\cdots + \xi^{(n_i)}_{n_i-r_i+1})}\Big]\right)\Bigg].
\end{aligned}
\end{equation}
As $\e\to 0$, the r.h.s.\ of (\ref{eq4.1}) behaves as
\begin{equation}\label{eq4.1b}
\prod_{i=1}^m e^{-(x^{(n_i)}_{n_i} +\cdots + x^{(n_i)}_{n_i-r_i+1})} \e^{m/2} \EE_{\vec{a};\rho}\Bigg[\prod_{i=1}^{m} \left((\xi^{(n_i)}_{n_i} +\cdots + \xi^{(n_i)}_{n_i-r_i+1})-\EE_{\vec{a};\rho}\left[(\xi^{(n_i)}_{n_i} +\cdots + \xi^{(n_i)}_{n_i-r_i+1})\right]\right)\Bigg].
\end{equation}
Using the moment formula (\ref{eqMoments}) we can analyze the $\e\to 0$ behavior of the integral and obtain that
\begin{equation*}
\left\{\xi^{(n_i)}_{n_i} +\cdots + \xi^{(n_i)}_{n_i-r_i+1}-\EE_{\vec{a};\rho}\left[\xi^{(n_i)}_{n_i} +\cdots + \xi^{(n_i)}_{n_i-r_i+1}\right]\right\}_{1\leq i\leq m}
\end{equation*}
have centered Gaussian moments. One can then easily extend this to
\begin{equation*}
\left\{\xi^{(n_i)}_{k_i}-\EE_{\vec{a};\rho}\left[\xi^{(n_i)}_{k_i}\right]\right\}_{1\leq i\leq m}
\end{equation*}
by linear combination, see Remark~\ref{Rem4.2}.

\begin{proposition}\label{PropGaussian}
Consider the Plancherel specialization. Fix $N\geq 1$, positive $\vec{a} = (a_1,\ldots, a_N)$, $m\geq 1$, a sequence of $m$ integers $N\geq n_1\geq n_2\geq \cdots \geq n_m\geq 1$, and $r_1,\ldots, r_m$ such that $0\leq r_i\leq n_i$ for $1\leq i\leq m$.

Then, under the scaling (\ref{eqScalingGaussian}) the following limit exists in distribution
\begin{equation*}
\xi^{(n_i)}_{n_i}(\tau):=\lim_{\e\to 0}\xi^{(n_i)}_{n_i}(\tau;\e),
\end{equation*}
and the vector $\big\{\xi^{(n)}_{k}(\tau)\big\}_{1\leq k \leq n \leq N}$ is a centered Gaussian (i.e. $\EE(\xi^{(n)}_{k}(\tau))= 0$)
with covariance determined by the following formula. For $n_i\geq n_j$,
\begin{equation}\label{eq4.8}
\begin{aligned}
{\bf C}(n_i,r_i;n_j,r_j):=\,&{\rm Cov}\left(\xi^{(n_i)}_{n_i}(\tau)+\ldots+\xi^{(n_i)}_{n_i-r_i+1}(\tau);\xi^{(n_j)}_{n_j}(\tau)+\ldots+\xi^{(n_j)}_{n_j-r_j+1}(\tau)\right)\\
=\,&\frac{\displaystyle\oint \oint \mathfrak{Cr}(r_{i},\vec{z}_{i};r_{j},\vec{z}_{j}) \mathfrak{F}(n_i,r_i,\vec{z}_{i}) \mathfrak{F}(n_j,r_j,\vec{z}_{j}) d\vec{z}_{i} d\vec{z}_{j}}
{\displaystyle\oint \mathfrak{F}(n_i,r_i,\vec{z}_{i}) d\vec{z}_{i} \oint \mathfrak{F}(n_j,r_j,\vec{z}_{j}) d\vec{z}_{j}}
\end{aligned}
\end{equation}
where $\vec{z}_i := (z_{i,k})_{k=1}^{r_i}$, $d\vec{z}_{i}:=\prod_{k=1}^{r_i} dz_{i,k}$, and
\begin{equation*}
\begin{aligned}
&\mathfrak{Cr}(r_{i},\vec{z}_{i};r_{j},\vec{z}_{j})=-\sum_{k=1}^{r_i}\sum_{\ell=1}^{r_j}\frac{z_{i,k}}{z_{i,k}-z_{j,\ell}},\\
&\mathfrak{F}(n_i,r_i,\vec{z}_{i}) =\frac{(-1)^{\frac{r_i(r_i+1)}{2}} \prod\limits_{1\leq k<\ell\leq r_i} (z_{i,k} - z_{i,\ell})^2}{(2\pi \I)^{r_i} r_i! \prod\limits_{k=1}^{r_i} (z_{i,k})^{r_i}} \, \prod_{k=1}^{r_i} \bigg(\prod_{\ell=1}^{n_{i}} \frac{a_\ell}{a_\ell-z_{i,k}}\bigg)\, e^{-z_{i,k} \tau}.
\end{aligned}
\end{equation*}
Here the integrals are counterclockwise oriented simple loops including the following poles: the contour for $z_{i,k}$ includes $a_1,\ldots,a_{n_i}$ as well as all $z_{j,\ell}$ for $i<j$ -- see Figure~\ref{Figtheorem45contours} for an illustration of such contours.
\end{proposition}

\begin{remark}\label{Rem4.2}
By linearity we immediately get that the vector $\xi=(\xi_k^{(n)})_{1\leq k\leq n\leq N}$ is also centered Gaussian with covariance matrix given by
\begin{equation*}
{\rm Cov}(\xi^{(n)}_{n-r+1}(\tau);\xi^{(n')}_{n'-r'+1}(\tau)) =
{\bf C}(n,r;n',r') - {\bf C}(n,r-1;n',r')-{\bf C}(n,r;n',r'-1)+{\bf C}(n,r-1;n',r'-1).
\end{equation*}
Also, notice that for a given $i$, the collection of variables $\vec{z}_{i}$ comes into the integrand (\ref{eq4.8}) symmetrically, and that the contours for these variable can be taken to be the same. This means that we can replace $\mathfrak{Cr}(r_{i},\vec{z}_{i};r_{j},\vec{z}_{j})$ with $-r_i r_j \frac{z_{i,1}}{z_{i,1}-z_{j,1}}$.
\end{remark}

\begin{remark}
The alpha specialization version of Proposition~\ref{PropGaussian} is proved identically to the Plancherel case (in fact, since the proof is written using the notation $\Pi(z;\rho)$ there is essentially no change needed). In the alpha case, we employ the scaling (\ref{eqScalingLLNalpha}) except refine it by setting $\lambda^{(n)}_{k}(t) = \e^{-1} x^{(n)}_{k}(t)+ \e^{-1/2} \xi^{(n)}_k(t;\e)$. All that changes in the above proposition is that in $\mathfrak{F}(n_i,r_i,\vec{z}_{i})$ the term $e^{-z_{i,k} \tau}$ is replaced by $\prod_{s=1}^{t} (1-\alpha_s z_{i,k})$.
\end{remark}

In order to prove Proposition~\ref{PropGaussian} we will use the following lemma.

\begin{lemma}\label{Lemma4.2}
Consider the setting of Proposition~\ref{PropGaussian} and define
\begin{equation}\label{eqn411}
Y_i^\e=\frac{1-e^{-\sqrt{\e}\big(\xi^{(n_i)}_{n_i}(\tau;\e) +\cdots + \xi^{(n_i)}_{n_i-r_i+1}(\tau;\e)\big)}}{\sqrt{\e}}.
\end{equation}
Then (using the same contours as in Proposition~\ref{PropGaussian})
\begin{equation}\label{eqncovar}
\begin{aligned}
\lim_{\e\to 0}\EE_{\vec{a};\rho} \Bigg[\prod_{i=1}^{m} \left(Y_i^\e-\EE_{\vec{a};\rho}\left[Y_i^\e\right]\right)\Bigg] &= \prod_{i=1}^m e^{x^{(n_i)}_{n_i}(\tau) +\cdots + x^{(n_i)}_{n_i-r_i+1}(\tau)}\\
\times& \sum_{\begin{subarray}{c}B\textrm{ perfect }\\ \textrm{matching of }\\\{1,\ldots,m\}\end{subarray}} \oint \cdots \oint \prod_{(j_1,j_2)\in B} \mathfrak{Cr}(r_{j_1},\vec{z}_{j_1};r_{j_2},\vec{z}_{j_2}) \prod_{i=1}^m \mathfrak{F}(n_i,r_i,\vec{z}_{i}) d\vec{z}_{i},
\end{aligned}
\end{equation}
if $m$ is even and $0$ if $m$ is odd. When we write perfect matching of $\{1,\ldots,m\}$ we mean a set of $m/2$ ordered pairs of $\{1,\ldots,m\}$ which contains all elements of $\{1,\ldots,m\}$, i.e.,
$$\Big\{\big\{(i_1,j_1),\ldots,(i_{m/2},j_{m/2})\big\}\subset\{1,\ldots,m\}^2 \,|\, i_k<j_k, \textrm{ for all }k, \{i_1,j_1,\ldots,i_{m/2},j_{m/2}\}=\{1,\ldots m\} \Big\}.
$$
Further, for all $i=1,\ldots,m$,
\begin{equation*}
\lim_{\e\to 0}\EE_{\vec{a};\rho} \left[Y_i^\e\right]=0.
\end{equation*}
Consequently, the $\big\{Y_i^{\e}\big\}_{1\leq i\leq m}$ converge as $\e\to 0$ in distribution to the Gaussian vector with limiting covariance between $\lim_{\e\to 0}Y_i^{\e}$ and $\lim_{\e\to 0}Y_j^{\e}$ given by (\ref{eq4.8}).
\end{lemma}

\begin{proof}
First note that the Gaussian convergence result stated at the end of the lemma follows immediately from the method of moments (i.e. if all moments converge to those of a Gaussian, this implies convergence in distribution to that Gaussian). Wick's theorem shows that the limiting moments in (\ref{eqncovar}) are those of a Gaussian vector, and noting that
 \begin{equation*}
e^{-(x^{(n_i)}_{n_i}(\tau) +\cdots + x^{(n_i)}_{n_i-r_i+1}(\tau))} = \oint \mathfrak{F}(n_i,r_i,\vec{z}_{i}) d\vec{z}_{i},
\end{equation*}
as obtained in Proposition~\ref{PropLLN} (see (\ref{eq3.3alpha})), we readily identify this limiting Gaussian with the covariance given in (\ref{eq4.8}).

The proof of the limiting moment formulas follows the general approach of~\cite[Lemma~4.2]{BG14}. We present this in some detail for completeness. We start from (\ref{eqMoments}) and rewrite it as
\begin{equation}\label{eq4.2}
\EE_{\vec{a};\rho} \Bigg[\prod_{i=1}^{m} q^{\lambda^{(n_i)}_{n_i} +\cdots + \lambda^{(n_i)}_{n_i-r_i+1}}\Bigg] =
\oint \cdots \oint \prod_{1\leq i<j\leq m} {\rm Cr}_\e(r_i,\vec{z}_{i};r_j,\vec{z}_{j})\prod_{i=1}^m \mathfrak{F}_\e(n_i,r_i,\vec{z}_{i})d \vec{z}_{i},
\end{equation}
where
\begin{equation*}
{\rm Cr}_\e(r_i,\vec{z}_{i};r_j,\vec{z}_{j}) = \prod_{k=1}^{r_i} \prod_{\ell=1}^{r_j} \frac{q(z_{i,k} - z_{j,\ell})}{z_{i,k} - q z_{j,\ell}},
\end{equation*}
and
\begin{equation*}
\mathfrak{F}_\e(n_i,r_i,\vec{z}_{i}) =\frac{(-1)^{\frac{r_i(r_i+1)}{2}} \prod\limits_{1\leq k<\ell\leq r_i} (z_{i,k} - z_{i,\ell})^2}{(2\pi \I)^{r_i} r_i! \prod\limits_{k=1}^{r_i} (z_{i,k})^{r_i}} \, \prod_{k=1}^{r_i} \bigg(\prod_{\ell=1}^{n_{i}} \frac{-a_\ell}{z_{i,k} - a_\ell}\bigg)\, \frac{\Pi(q z_{i,k};\rho)}{\Pi(z_{i,k};\rho)}.
\end{equation*}
Observe that by expanding the product on the left-hand side,
\begin{equation}\label{eq4.5}
\begin{aligned}
&\EE_{\vec{a};\rho} \Bigg[\prod_{i=1}^{m} \left(q^{\lambda^{(n_i)}_{n_i} +\cdots + \lambda^{(n_i)}_{n_i-r_i+1}}
-\EE_{\vec{a};\rho}\Big[ q^{\lambda^{(n_i)}_{n_i} +\cdots + \lambda^{(n_i)}_{n_i-r_i+1}}\Big]\right)\Bigg]\\
&=\sum_{A\subset\{1,\ldots,m\}} (-1)^{m-|A|} \EE_{\vec{a};\rho} \Bigg[\prod_{i\in A} q^{\lambda^{(n_i)}_{n_i} +\cdots + \lambda^{(n_i)}_{n_i-r_i+1}}\Bigg] \prod_{j\not\in A} \EE_{\vec{a};\rho} \Bigg[q^{\lambda^{(n_j)}_{n_j} +\cdots + \lambda^{(n_j)}_{n_j-r_j+1}}\Bigg].
\end{aligned}
\end{equation}
For a set $A$, we denote $A^{(2)}:=\{(i,j)\,|\, i<j, i,j\in A\}$. Since the term ${\rm Cr}$ is not present when $m=1$, using (\ref{eq4.2}) we find that
\begin{equation}\label{eq4.10}
(\ref{eq4.5})= \oint \cdots \oint \sum_{A\subset\{1,\ldots,m\}} (-1)^{m-|A|} \prod_{(j_1,j_2)\in A^{(2)}} {\rm Cr}_\e(r_{j_1},\vec{z}_{j_1};r_{j_2},\vec{z}_{j_2})
\prod_{i=1}^m \mathfrak{F}_\e(n_i,r_i,\vec{z}_{i}) d\vec{z}_{i}.
\end{equation}
In the $\e\to 0$ limit, ${\rm Cr}_\e \to 1$ linearly in $\e$. With this in mind, define $\mathfrak{Cr}_\e$ by
\begin{equation*}
{\rm Cr}_\e(r_{j_1},\vec{z}_{j_1};r_{j_2},\vec{z}_{j_2}) = 1+\e \,\mathfrak{Cr}_\e(r_{j_1},\vec{z}_{j_1};r_{j_2},\vec{z}_{j_2}).
\end{equation*}
Using this expansion we can rewrite (\ref{eq4.10}) to show that
\begin{equation}\label{eq4.12}
\begin{aligned}
(\ref{eq4.5}) =& \oint \cdots \oint \sum_{B\subset\{1,\ldots,m\}^{(2)}} \e^{|B|}\prod_{(j_1,j_2)\in B} \mathfrak{Cr}_\e(r_{j_1},\vec{z}_{j_1};r_{j_2},\vec{z}_{j_2})\\
&\times \prod_{i=1}^m \mathfrak{F}_\e(n_i,r_i,\vec{z}_{i}) d\vec{z}_{i}
\sum_{\begin{subarray}{c}
A\subset\{1,\ldots,m\},\\
{\rm Supp}(B)\subset A
\end{subarray}}(-1)^{m-|A|}
\end{aligned}
\end{equation}
where ${\rm Supp}$ is the set of all elements of $\{1,\ldots, m\}$ which show up in $B$.

Next, using
$$\sum_{A: I_1\subset A\subset I_2} (-1)^{|I_2|-|A|}=\Id_{I_2=I_1},$$
we obtain
\begin{equation}\label{eq4.14}
(\ref{eq4.5})= \oint \cdots \oint \sum_{\begin{subarray}{c}B\subset\{1,\ldots,m\}^{(2)},\\ {\rm Supp}(B)=\{1,\ldots,m\}\end{subarray}} \e^{|B|}\prod_{(j_1,j_2)\in B} \mathfrak{Cr}_\e(r_{j_1},\vec{z}_{j_1};r_{j_2},\vec{z}_{j_2})
\prod_{i=1}^m \mathfrak{F}_\e(n_i,r_i,\vec{z}_{i}) d\vec{z}_{i}.
\end{equation}
We now seek to study the limit $\lim_{\e \to 0} \e^{-m/2} \cdot (\ref{eq4.5})$. Note that $B$ is a set of pairs and the condition ${\rm Supp}(B)=\{1,\ldots,m\}$ implies that $|B|\geq \lceil m/2 \rceil$. If $m$ is odd, then $\e^{-m/2} \e^{|B|}\to 0$ as $\e\to 0$ for all $B$, while if $m$ is even, the only non-vanishing terms in the $\e\to 0$ limit are the perfect matchings of $\{1,\ldots,m\}$ (with the second entries in the pairing to be larger than the first entries).
We also have the following uniform convergence to continuous functions (the contours are fixed)
\begin{equation}\label{eq4.15}
\begin{aligned}
&\lim_{\e\to 0}\mathfrak{Cr}_\e(r_{j_1},\vec{z}_{j_1};r_{j_2},\vec{z}_{j_2}) = \mathfrak{Cr}(r_{j_1},\vec{z}_{j_1};r_{j_2},\vec{z}_{j_2})\\
&\lim_{\e\to 0}\mathfrak{F}_\e(n_i,r_i,\vec{z}_{i})=\mathfrak{F}(n_i,r_i,\vec{z}_{i}).
\end{aligned}
\end{equation}
Thus we can take the $\e\to 0$ limit inside the integrals with the result
\begin{equation*}
\begin{aligned}
&\lim_{\e \to 0} \e^{-m/2} (\ref{eq4.5}) = \lim_{\e\to 0} \e^{-m/2}\EE_{\vec{a};\rho} \Bigg[\prod_{i=1}^{m} \left(q^{\lambda^{(n_i)}_{n_i} +\cdots + \lambda^{(n_i)}_{n_i-r_i+1}}
-\EE_{\vec{a};\rho}\Big[ q^{\lambda^{(n_i)}_{n_i} +\cdots + \lambda^{(n_i)}_{n_i-r_i+1}}\Big]\right)\Bigg]
\\
&= \sum_{\begin{subarray}{c}B\textrm{ perfect }\\ \textrm{matching of }\\\{1,\ldots,m\}\end{subarray}} \oint \cdots \oint \prod_{(j_1,j_2)\in B} \mathfrak{Cr}(r_{j_1},\vec{z}_{j_1};r_{j_2},\vec{z}_{j_2}) \prod_{i=1}^m \mathfrak{F}(n_i,r_i,\vec{z}_{i}) d\vec{z}_{i}.
\end{aligned}
\end{equation*}

Recalling the scalings of (\ref{eqScalingGaussian}) we have that
\begin{equation}\label{eq4.6}
\lim_{\e \to 0} \e^{-m/2}(\ref{eq4.5})=(-1)^m\prod_{i=1}^m e^{-(x^{(n_i)}_{n_i}(\tau) +\cdots + x^{(n_i)}_{n_i-r_i+1}(\tau))} \lim_{\e\to 0}\EE_{\vec{a};\rho} \Bigg[\prod_{i=1}^{m} \left(Y_i^\e-\EE_{\vec{a};\rho}\left[Y_i^\e\right]\right)\Bigg].
\end{equation}

Thus we have proven that
\begin{equation*}
\begin{aligned}
\lim_{\e\to 0}\EE_{\vec{a};\rho} \Bigg[\prod_{i=1}^{m} &\left(Y_i^\e-\EE_{\vec{a};\rho}\left[Y_i^\e\right]\right)\Bigg] =(-1)^m \prod_{i=1}^m e^{x^{(n_i)}_{n_i}(\tau) +\cdots + x^{(n_i)}_{n_i-r_i+1}(\tau)} \\
\times\,& \sum_{\begin{subarray}{c}B\textrm{ perfect }\\ \textrm{matching of }\\\{1,\ldots,m\}\end{subarray}} \oint \cdots \oint \prod_{(j_1,j_2)\in B} \mathfrak{Cr}(r_{j_1},\vec{z}_{j_1};r_{j_2},\vec{z}_{j_2}) \prod_{i=1}^m \mathfrak{F}(n_i,r_i,\vec{z}_{i}) d\vec{z}_{i},
\end{aligned}
\end{equation*}
which by Wick's theorem it is the $m$th moment of a Gaussian process with covariance given by~(\ref{eq4.8}).

Finally let us show that the average of $Y_i^\e$ goes to zero. We have
\begin{equation}\label{eq4.25}
e^{-(x^{(n_i)}_{n_i}(\tau) +\cdots + x^{(n_i)}_{n_i-r_i+1}(\tau))} \EE_{\vec{a};\rho} \left[Y_i^\e\right] = \frac{e^{-(x^{(n_i)}_{n_i}(\tau) +\cdots + x^{(n_i)}_{n_i-r_i+1}(\tau))}-\EE_{\vec{a};\rho} \big[q^{\lambda^{(n_i)}_{n_i}+\cdots \lambda^{(n_i)}_{n_i-r_i+1}}\big]}{\sqrt{\e}}
\end{equation}
and using (\ref{eq4.2}) for $m=1$ we have
\begin{equation*}
(\ref{eq4.25}) = \frac{(-1)^{\frac{r_i(r_i+1)}{2}}}{{(2\pi \I)^{r_i} r_i!}} \oint \frac{\prod_{1\leq k<\ell\leq r_i} (z_{i,k} - z_{i,\ell})^2}{ \prod_{k=1}^{r_i} (z_{i,k})^{r_i}} \, \prod_{k=1}^{r_i} \bigg(\prod_{\ell=1}^{n_{i}} \frac{-a_\ell}{z_{i,k} - a_\ell}\bigg)\, \frac{\frac{\Pi(q z_{i,k};\rho)}{\Pi(z_{i,k};\rho)}-1}{\sqrt{\e}} d \vec{z}_{i}.
\end{equation*}
Since $\frac{\frac{\Pi(q z_{i,k};\rho)}{\Pi(z_{i,k};\rho)}-1}{\sqrt{\e}}\to 0$ as $\e\to 0$ uniformly in the $z$'s on the chosen contours, we can take the limit $\e\to 0$ inside and obtain that $\lim_{\e\to 0}\EE_{\vec{a};\rho} \left[Y_i^\e\right]=0$.
\end{proof}

\begin{proof}[Proof of Proposition~\ref{PropGaussian}]
From Lemma~\ref{Lemma4.2} we have that each $Y_i^{\e}$, for $1\leq i\leq m$, is tight as $\e \to 0$ (joint tightness of these random variables together follows immediately as well). This means that for all $\delta>0$ there exists large $M(\delta)>0$ and small $\e(\delta)>0$ such that for all $\e<\e(\delta)$,
$$\PP_{\vec{a};\rho}\big(|Y_i^{\e}|>M(\delta)\big)<\delta.$$

Now fix any constant $c\in (0,1)$ and observe that there exist constants $c',c''>0$ depending on $c$ such that for all $x\in[0,c]$
$$-\ln(1-x)<c' x, \qquad -\ln(1+x)>-c'' x.$$
Define $\tilde{\e}(\delta) =\min\big(\e(\delta),(c/M(\delta))^2\big)$ and $\tilde{M}(\delta) = M(\delta)\max(c',c'')$. Set $\tilde Y_i^\e:=\xi^{(n_i)}_{n_i}(\tau;\e) +\cdots + \xi^{(n_i)}_{n_i-r_i+1}(\tau;\e)$. Then, it follows from the above inequalities and the fact that \mbox{$\tilde{Y}_i^{\e} = -\e^{-1/2} \ln(1-\e^{1/2} Y_i^{\e})$} that
for all $\delta>0$ there exists large $\tilde{M}(\delta)>0$ (as given above) and $\tilde{\e}(\delta)>0$ (as given above) such that for all $\e<\tilde{\e}(\delta)$,
$$\PP_{\vec{a};\rho}\big(|\tilde{Y}_i^{\e}|>\tilde{M}(\delta)\big)<\delta.$$
It follows from this and Prokhorov's theorem that every sequence in $\e$ of $\big\{\tilde{Y}_i^{\e}\big\}_{1\leq i\leq m}$ has a convergent subsequence. On such a convergent subsequence of $\e$'s, Taylor expansion shows that the differences $Y_i^{\e}-\tilde{Y}_i^{\e}$ converge to zero in probability. Therefore, in light of the final statement of Lemma~\ref{Lemma4.2}, on said subsequences the $\big\{\tilde{Y}_i^{\e}\big\}_{1\leq i\leq m}$ converge in distribution to the centered Gaussian with covariance given in (\ref{eq4.8}). Since every subsequential limit point is the same, this implies convergence as $\e\to 0$ of $\big\{\tilde{Y}_i^{\e}\big\}_{1\leq i\leq m}$ to this Gaussian as well, completing the proof of the proposition.
\end{proof}

\begin{remark}
Given the explicit form of the $q$-Whittaker process, one might hope to observe the Gaussian limit from asymptotics of that formula. We did not see how to do this, which is why we pursued this alternative route using the moment method.
\end{remark}

\subsection{Derivation of SDEs satisfied by Gaussian fluctuations}
In Proposition~\ref{PropGaussian} we have shown that at fixed time $\tau$ the limit process $\xi$ is Gaussian. Here we consider the SDE induced by the Plancherel push-block dynamics started with packed initial conditions and with all $a_1=\ldots=a_N=1$. We do not consider the limits of the various other $q$-Whittaker preserving dynamics considered earlier.

\begin{proposition}\label{propSDE}
Let us denote by $y_k^{(n)}(\tau):=e^{-x_k^{(n)}(\tau)}$ for $1\leq k \leq n \leq N$. Then, the evolution of the vector $\xi(\tau)=\big(\xi^{(n)}_k(\tau)\big)_{1\leq k\leq n\leq N}$ starting from $\xi(0)=0$ satisfies the system of SDE's (all terms $y, \xi, \sigma, a, b, c, W$ are functions of time $\tau$, though we suppress them to shorten expressions)
\begin{equation}\label{eq4.32}
d \xi^{(n)}_k = -a_k^{(n)} (\xi_k^{(n)}-\xi_{k-1}^{(n-1)}) d\tau + b_k^{(n)} (\xi_{k}^{(n)}-\xi_{k+1}^{(n)}) dt-c_k^{(n)} (\xi_{k}^{(n)}-\xi_k^{(n-1)}) d\tau + \sigma_k^{(n)} dW_k^{(n)},
\end{equation}
where $W_k^{(n)}(\tau)$, $1\leq k \leq n \leq N$, are independent standard Brownian motions,
\begin{equation}\label{eq4.33}
\sigma_k^{(n)}=\sqrt{\frac{\left(1-\frac{y^{(n-1)}_{k-1}}{y^{(n)}_{k}}\right) \left(1-\frac{y^{(n)}_{k}}{y^{(n)}_{k+1}}\right)}{1-\frac{y^{(n)}_k}{y^{(n-1)}_k}}},
\end{equation}
and
\begin{equation}\label{eq4.34}
a_k^{(n)} = \frac{y_{k-1}^{(n-1)}}{y_k^{(n)}}\frac{1-\frac{y^{(n)}_k}{y^{(n)}_{k+1}}}{1-\frac{y^{(n)}_k}{y^{(n-1)}_k}}, \quad
b_k^{(n)} = \frac{y_{k}^{(n)}}{y_{k+1}^{(n)}}\frac{1-\frac{y^{(n-1)}_{k-1}}{y^{(n)}_{k}}}{1-\frac{y^{(n)}_k}{y^{(n-1)}_k}}, \quad
c_k^{(n)} = \frac{y_{k}^{(n)}}{y_k^{(n-1)}}\frac{\left(1-\frac{y^{(n-1)}_{k-1}}{y^{(n)}_{k}}\right) \left(1-\frac{y^{(n)}_{k}}{y^{(n)}_{k+1}}\right)}{\left(1-\frac{y^{(n)}_k}{y^{(n-1)}_k}\right)^2}.
\end{equation}
\end{proposition}

This proposition can be proved via the argument of~\cite[Theorem 1]{BCT15}. Instead of reproducing that proof, let us informally derive this result. First, however, note that this above SDE almost surely has a unique continuous solution. To see this, note that the $\sigma,a,b,c$ variables are all bounded as $\tau$ varies in a compact interval. The drift for the SDE is uniformly (as $\tau$ varies in a compact interval) Lipschitz in the $\xi$ variables, and the covariance is independent of the $\xi$ variables. Thus, standard uniqueness results for SDE apply.

Under the Plancherel push-block dynamics, each particle jumps rate according to independent Poisson clocks whose intensity depends on the distances of (up to) three neighboring particles, see (\ref{eqDynamics}). Under the scaling (\ref{eqScalingGaussian}), we expect to obtain a diffusion process where the $\xi^{(n)}_k(\tau)$ are driven by independent Brownian motions. What remains is to determine the diffusion coefficients as well as the drifts.

To determine the drift, we need to determine, as $\e\to 0$, the $\Or(d\tau)$ term in
\begin{equation*}
\EE\left[\xi^{(n)}_k(\tau+d\tau;\e)-\xi^{(n)}_k(\tau;\e)\right] = - \frac{x^{(n)}_k(\tau+d\tau)-x^{(n)}_k(\tau)}{\sqrt{\e}} + \sqrt{\e} \EE\left[\lambda^{(n)}_k(\e^{-1}\tau+\e^{-1}d\tau)-\lambda^{(n)}_k(\e^{-1}\tau)\right].
\end{equation*}
The jump rate for $\lambda^{(n)}_k$ is given by (\ref{eqDynamics}) (with $a_n=1$), i.e.,
\begin{equation}\label{eq4.36}
R^{(n)}_k=\frac{(1-c_1 e^{-\sqrt{\e}b_1})(1-c_2 e^{-\sqrt{\e}b_2})}{1-c_3 e^{-\sqrt{\e}b_3}}
\end{equation}
with the short-hand notations
\begin{equation*}
c_1=e^{x^{(n)}_{k}(\tau)-x^{(n-1)}_{k-1}(\tau)},\quad c_2= e^{x^{(n)}_{k+1}(\tau)-x^{(n)}_k(\tau)},\quad c_3=e^{x^{(n-1)}_k(\tau)-x^{(n)}_k(\tau)}
\end{equation*}
and
\begin{equation*}
b_1=\xi^{(n-1)}_{k-1}(\tau;\e)-\xi^{(n)}_k(\tau;\e),\quad b_2=\xi^{(n)}_k(\tau;\e)-\xi^{(n)}_{k+1}(\tau;\e),\quad b_3=\xi^{(n)}_k(\tau;\e)-\xi^{(n-1)}_k(\tau;\e).
\end{equation*}
By (\ref{eqDynamics}) we have
\begin{equation}\label{eq4.39}
-\frac{x^{(n)}_k(\tau+d\tau)-x^{(n)}_k(\tau)}{\sqrt{\e}} = -\frac{d\tau}{\sqrt{\e}} \frac{(1-c_1)(1-c_2)}{1-c_3}+\Or(d\tau^2)
\end{equation}
and by (\ref{eq4.36}) we have
\begin{equation}\label{eq4.40}
\begin{aligned}
&\sqrt{\e} \EE\left[\lambda^{(n)}_k(\e^{-1}\tau+\e^{-1}d\tau)-\lambda^{(n)}_k(\e^{-1}\tau)\right] \\
&= \sqrt{\e}\frac{d\tau}{\e} \frac{(1-c_1)(1-c_2)}{1-c_3}\left(1+\sqrt{\e}\frac{c_1 b_1}{1-c_1}+\sqrt{\e}\frac{c_2 b_2}{1-c_2}-\sqrt{\e}\frac{c_3 b_3}{1-c_3}\right) + \Or(d\tau^2).
\end{aligned}
\end{equation}
Combining (\ref{eq4.39}) and (\ref{eq4.40}) we see that the $\Or(d\tau)$ term has a limit as $\e\to 0$, which is the drift term in (\ref{eq4.32}).

To determine the diffusion coefficient, we need to compute the $\Or(d\tau)$ term in
\begin{equation}\label{eq4.42}
\EE\left[\xi^{(n)}_k(\tau+d\tau;\e)-\xi^{(n)}_k(\tau;\e)\right]^2-\left(\EE\left[\xi^{(n)}_k(\tau+d\tau;\e)-\xi^{(n)}_k(\tau;\e)\right]\right)^2.
\end{equation}
We have
\begin{equation*}
(\ref{eq4.42}) = \e \, \EE\left[\lambda^{(n)}_k(\e^{-1}\tau+\e^{-1}d\tau)-\lambda^{(n)}_k(\e^{-1}\tau)\right]^2 + \Or(d\tau^2) =\e \frac{d\tau}{\e} R^{(n)}_k+\Or(d\tau^2).
\end{equation*}
As $\e\to 0$, $R^{(n)}_k\to \frac{(1-c_1)(1-c_2)}{1-c_3}$, which is the square of $\sigma^{(n)}_k$ in (\ref{eq4.32}) as stated in (\ref{eq4.33}).

\section{Slow decorrelation and Edwards-Wilkinson asymptotics}
In Proposition~\ref{PropGaussian} we have obtained the covariance of the process for sums of random variables counted from the edges. In this section we consider the limit when $n_1,n_2\to \infty$ simultaneously but also $r_1,r_2\to\infty$ with the same speed, i.e., we would like to concentrate in the bulk of the macroscopic picture. For that purpose we need to have manageable expressions (not $(r_1,r_2)$-fold integrals) for the covariance of the single random variables. This is possible if we first consider the large time limit.

In this section we consider the Plancherel case with all $a_i=1$. Also, as it will be used often below, we introduce the following notation: given a set $S$, $\frac{1}{2\pi\I} \oint_{\Gamma_S} dz f(z)$ means that the integral path goes around the poles of $S$ only, i.e., the integral is the sum of the residues of $f$ at the elements of the set $S$.

\subsection{Large time simplification}\label{seclargetimesim}
Since $\xi^{(1)}_1$ is a standard Brownian motion, at times $\tau=T L$ we need to consider the scaling $L^{-1/2}\xi^{(1)}_1(\tau=L T)$ as $L\to\infty$ in order to still see a Brownian motion. This suggest that we consider the same scaling for the set $\{\xi^{(n)}_k,1\leq k\leq n\leq N\}$. Propositions~\ref{PropLargeL} and~\ref{PropLargeLsde}, show how in the $L\to\infty$ limit our Gaussian process simplifies considerably.

\begin{proposition}\label{PropLargeL}
For any fixed $T>0$, the limit
\begin{equation*}
\zeta^{(n)}_k(T)=\lim_{L\to\infty} L^{-1/2} \xi^{(n)}_k(L T)
\end{equation*}
exists and $\zeta=\{\zeta^{(n)}_k(T),1\leq k\leq n\leq N\}$ is a centered Gaussian process with covariance given through the following formula. For $n_1\geq n_2$,
\begin{equation}\label{eq5.2}
\begin{aligned}
&{\rm Cov}\left(\zeta^{(n_1)}_{n_1}(T)+\ldots+\zeta^{(n_1)}_{n_1-r_1+1}(T);\zeta^{(n_2)}_{n_2}(T)+\ldots+\zeta^{(n_2)}_{n_2-r_2+1}(T)\right)\\
=\,&\frac{\displaystyle\oint \oint \sum_{k=1}^{r_1}\sum_{\ell=1}^{r_2}\frac{1}{z_k-w_\ell} \Big(\prod\limits_{1\leq i<j\leq r_1}(z_j-z_i)^2\prod\limits_{m=1}^{r_1} \frac{e^{T z_m}}{(z_m)^{n_1}}dz_m\Big)
\Big(\prod\limits_{1\leq i<j\leq r_2}(w_j-w_i)^2\prod\limits_{m=1}^{r_2} \frac{e^{T w_m}}{(w_m)^{n_2}}dw_m\Big)}
{\Big(\displaystyle\oint \prod\limits_{1\leq i<j\leq r_1}(z_j-z_i)^2\prod\limits_{m=1}^{r_1} \frac{e^{T z_m}}{(z_m)^{n_1}}dz_m\Big)
\Big(\oint \prod\limits_{1\leq i<j\leq r_2}(w_j-w_i)^2\prod\limits_{m=1}^{r_2} \frac{e^{T w_m}}{(w_m)^{n_2}}dw_m\Big)},
\end{aligned}
\end{equation}
where the integrals are around $0$ and the $z$-contours contain the $w$-contours.
\end{proposition}
\begin{proof}
Consider (\ref{eq4.8}) in which we set $i=1$, $j=2$, and we do the change of variables $z_{1,k}\to 1-L^{-1} z_k$, $1\leq k\leq r_1$ and $z_{2,\ell}\to 1-L^{-1} w_\ell$, $1\leq \ell\leq r_2$. Now let us fix the integration contours for $w_\ell$ to be circles centered at $0$ of a given radius and the ones for $z_k$ to be circles centered at $0$ of a larger radius. Then
\begin{equation*}
\lim_{L\to\infty} L^{-1} \mathfrak{Cr}(r_{1},\vec{z}_{1};r_{2},\vec{z}_{2})= \sum_{k=1}^{r_1}\sum_{\ell=1}^{r_2} \frac{1}{z_k-w_\ell}
\end{equation*}
and
\begin{equation}\label{eq5.5}
\lim_{L\to\infty} (-1)^{r_1(r_1+1)/2} e^{L T r_1} L^{r_1(r_1-1)}L^{-n_1 r_1}\mathfrak{F}(n_1,r_1,\vec{z}_{1})= \frac{\prod\limits_{1\leq i<j\leq r_1}(z_i-z_j)^2\prod\limits_{m=1}^{r_1} e^{T z_m} z_m^{-n_1}}{(2\pi\I)^{r_1} r_1!},
\end{equation}
where the limits are uniform in the integration paths. The $L$-dependent prefactors in (\ref{eq5.5}) cancel out with the denominator in (\ref{eq4.8}), leading to the statement of Proposition~\ref{PropLargeL}.
\end{proof}

From (\ref{eq5.2}) it is possible to considerably simplify the covariance for the single random variables $\zeta^{(n)}_{n-r+1}$ by using random-matrix type technology. Indeed, first notice that since all the $z$-contours can be chosen to be the same (as well as all the $w$-contours are the same), we can by symmetry replace
$\sum_{k=1}^{r_1}\sum_{\ell=1}^{r_2}\frac{1}{z_k-w_\ell}$ with just $r_1 r_2 \frac{1}{z_1-w_1}$. Then only $z_1$ and $w_1$ interact and we can put together the integration over $z_2,\ldots,z_{r_1}$ into the notation
\begin{equation*}
\rho^{r_1,n_1}(z_1)=\frac{\frac{1}{(2\pi\I)^{r_1}} \oint_{\Gamma_0} dz_2\ldots dz_{r_1} \prod_{1\leq i<j\leq r_1}(z_j-z_i)^2\prod_{m=1}^{r_1} \frac{e^{T z_m}}{(z_m)^{n_1}}}
{\frac{1}{(2\pi\I)^{r_1}} \oint_{\Gamma_0} dz_1 dz_2\ldots dz_{r_1} \prod_{1\leq i<j\leq r_1}(z_j-z_i)^2\prod_{m=1}^{r_1} \frac{e^{T z_m}}{(z_m)^{n_1}}}.
\end{equation*}
Then, one recognizes that $\rho^{r_1,n_1}(z_1)$ is the first correlation function of the determinantal measure (orthogonal polynomial ensemble) $\prod_{1\leq i<j\leq r_1}(z_j-z_i)^2\prod_{m=1}^{r_1} \frac{e^{T z_m}}{(z_m)^{n_1}}$ normalized to $1$. Therefore, if we find the orthogonal polynomials with respect to the dot product
\begin{equation*}
\langle f,g\rangle_n=\frac{1}{2\pi\I}\oint_{\Gamma_0} f(z) g(z) \frac{e^{T z}}{z^n} dz
\end{equation*}
we can readily get $\rho^{r,n}(z)$. Indeed, let $p^n_k(z)$ be the orthogonal polynomial of degree $k$, then (see e.g.\cite{Koe04,Deift00})
\begin{equation}\label{eq5.7}
\rho^{r,n}(z)=\frac{e^{T z}}{z^{n}}\sum_{k=0}^{r-1}\frac{(p^n_k(z))^2}{\langle p^n_k,p^n_k\rangle_n}.
\end{equation}
By similar reasoning, we also have, for $n_1\geq n_2$,
\begin{equation}\label{eq5.8}
\begin{aligned}
{\rm Cov}\Big(\zeta^{(n_1)}_{n_1}(T)+\ldots+\zeta^{(n_1)}_{n_1-r_1+1}(T)&;\zeta^{(n_2)}_{n_2}(T)+\ldots+\zeta^{(n_2)}_{n_2-r_2+1}(T)\Big)\\
&=\oint_{\Gamma_0} \frac{dw}{2\pi\I}\oint_{\Gamma_{0,w}}\frac{dz}{2\pi\I} \frac{1}{z-w} \rho^{r_1,n_1}(z)\rho^{r_2,n_2}(w).
\end{aligned}
\end{equation}
Further, due to linearity of the covariance, using the sum formula in (\ref{eq5.7}) we directly get
\begin{equation}\label{eq5.9}
{\rm Cov}\Big(\zeta^{(n_1)}_{n_1-r_1+1}(T),\zeta^{(n_2)}_{n_2-r_2+1}(T)\Big)=\oint_{\Gamma_0} \frac{dw}{2\pi\I}\oint_{\Gamma_{0,w}}\frac{dz}{2\pi\I} \frac{1}{z-w} \frac{e^{T z}e^{T w}}{z^{n_1}w^{n_2}}\frac{(p^{n_1}_{r_1-1}(z))^2}{\langle p^{n_1}_{r_1-1},p^{n_1}_{r_1-1}\rangle_{n_1}} \frac{(p^{n_2}_{r_2-1}(w))^2}{\langle p^{n_2}_{r_2-1},p^{n_2}_{r_2-1}\rangle_{n_2}}.
\end{equation}

\begin{remark}
The random matrix methods could be used also on the formula of Proposition~\ref{PropGaussian} (i.e., with finite $L$) to get $\rho^{r,n}(z)$. However, the weight for the orthogonalization depends on $r$ as well and therefore the step from (\ref{eq5.8}) to (\ref{eq5.9}) does not work. This is the reason why in the $L\to\infty$ limit the system is simpler (although still quite nontrivial).
\end{remark}

The orthogonal polynomials from (\ref{eq5.7}) can be explicitly computed as follows. Note that we are free to choose the norm of each polynomial.
\begin{lemma}\label{lemmaOrthoPolyn}
The functions
\begin{equation*}
p^n_k(z)=\frac{k!}{(n-k-1)!}\sum_{\ell=0}^k\frac{(n-1-\ell)!}{(k-\ell)!\, \ell!}(-T z)^\ell
\end{equation*}
are orthogonal under $\langle\cdot,\cdot\rangle_n$ (as $k$ varies in $\{0,1,\ldots, N-1\}$) with norm squared
\begin{equation*}
\langle p^n_k,p^n_k\rangle_n = (-1)^k \frac{k!}{(n-1-k)!}T^{n-1}.
\end{equation*}
Further, the following integral representations hold:
\begin{equation*}
\begin{aligned}
p^n_k(z)&=\frac{T^n(-z)^k}{(n-1-k)!}\int_0^\infty \left(1-\frac{y}{z}\right)^k y^{n-1-k} e^{-Ty} dy\\
&=-e^{-T z} z^{n-1-k} k! \frac{1}{2\pi\I} \oint_{\Gamma_z} \left(1-\frac{v}{z}\right)^{-k-1} v^{k-n} e^{T v} dv
\end{aligned}
\end{equation*}
from which
\begin{equation*}
\frac{(p^n_k(z))^2}{\langle p^n_k,p^n_k\rangle_n} = -\frac{z^{n-1} T}{e^{T z}}\int_0^\infty \left(1-\frac{y}{z}\right)^k y^{n-1-k}e^{-T y} dy
\,\,\frac{1}{2\pi\I}\oint_{\Gamma_z} \left(1-\frac{v}{z}\right)^{-k-1} v^{k-n} e^{T v} dv.
\end{equation*}
\end{lemma}
\begin{proof}
We establish the desired results by comparing to the monic Laguerre polynomials (see e.g.~\cite{KS96})
$$\tilde{p}^{(a)}_n(x) = (-1)^n \frac{\Gamma(n+a+1)}{\Gamma(a+1)}\, {}_1F_1 \left(\begin{array}{c} -n \\ a+1 \end{array};x\right).$$
These are orthogonal with respect to the dot product $(f,g)_{a}=\int_0^{\infty}dx f(x)g(x) x^ae^{-x}/\Gamma(a+1)$ on $[0,\infty)$ with norm squared $(\tilde{p}^{(a)}_k,\tilde{p}^{(a)}_k)_{a} = k!\tfrac{\Gamma(k+a+1)}{\Gamma(a+1)}$. This holds (by analytic continuation in $a$) for all $a$. We may easily compute the pairing of the constant function $1$ and the power $x^k$ under both $\langle \cdot,\cdot\rangle_{n}$ and $(\cdot,\cdot)_a$, finding
$$
\langle 1,x^k\rangle_n = \frac{T^{n-k-1}}{(n-k-1)!},\qquad (1,x^k)_a = \frac{\Gamma(k+a+1)}{\Gamma(a+1)}.
$$
Considering $(1,x^k)_a$, if we replace $a$ by $-n$, $x$ by $-Tx$, and multiply the result by an overall factor of $T^{n-1}/(n-1)!$, we arrive at the above expression for $\langle 1,x^k\rangle_n$. From this consideration and the explicit form of $\tilde{p}^{(a)}_n(x)$, we conclude the $p^n_k(z)$ are orthogonal with the above specified norm-squared. The integral representations are proved by employing the Binomial theorem and then performing the explicit integration of each resulting term.
\end{proof}
Using Lemma~\ref{lemmaOrthoPolyn} we finally get the formula for the covariance that we are going to use in the asymptotic analysis as well.
\begin{corollary}\label{CorCovFixedTime}
The covariance at $T=1$ is given by
\begin{multline}\label{eq5.1}
{\rm Cov}\Big(\zeta^{(n_1)}_{n_1-r_1+1}(T=1),\zeta^{(n_2)}_{n_2-r_2+1}(T=1)\Big)\\
=\frac{1}{(2\pi\I)^2}\oint_{\Gamma_0} dw\oint_{\Gamma_{0,w}}dz \frac{1}{z-w}
\bigg(\int_0^\infty dx (w-x)^{r_2-1} x^{n_2-r_2} e^{-x}\frac{1}{2\pi\I} \oint_{\Gamma_w} du \frac{e^{u}}{(w-u)^{r_2}u^{n_2-r_2+1}}\bigg)\\
\times\bigg(\int_0^\infty dy (z-y)^{r_1-1} y^{n_1-r_1} e^{-y}\frac{1}{2\pi\I} \oint_{\Gamma_z} dv \frac{e^{v}}{(z-v)^{r_1}v^{n_1-r_1+1}}\bigg).
\end{multline}
\end{corollary}

We turn now to time dynamics. By Proposition~\ref{propSDE} the vector \mbox{$\xi(T)=\{\xi^{(n)}_k(T),1\leq k\leq n \leq N\}$} satisfies a system of linear SDEs. Therefore, one expects that the limit vector \mbox{$\zeta(T)=\{\zeta^{(n)}_k(T),1\leq k\leq n \leq N\}$} likewise satisfies such a (perhaps simplied) system of linear SDEs.

\begin{proposition}\label{PropLargeLsde}
For any $0<T_0<T_1$, the limit
\begin{equation*}
\zeta^{(n)}_k(T)=\lim_{L\to\infty} L^{-1/2} \xi^{(n)}_k(L T)
\end{equation*}
exists in the topology $\mathcal{C}\big([T_0,T_1],\mathcal{C}(\R)\big)$ (i.e., continuous space-time processes on $[T_0,T_1]\times \R$) and $\zeta(T)=\{\zeta^{(n)}_k(T),1\leq k\leq n\leq N\}$
satisfies the system of SDEs
\begin{equation}\label{eqzetasde}
d \zeta^{(n)}_k(T) = \sum_{(k',n')}A(T)_{(k,n),(k',n')} \zeta^{(n')}_{k'}(T) dT + dW_k^{(n)}(T),
\end{equation}
where $W_k^{(n)}(T)$, $1\leq k \leq n \leq N$, are independent standard Brownian motions and the matrix $A(T)$ has as its only non-zero entries
\begin{equation*}
A(T)_{(k,n),(k,n)}=\frac{-n+1}{T},\quad A(T)_{(k,n),(k-1,n-1)}=\frac{k-1}{T},\quad A(T)_{(k,n),(k,n-1)}=\frac{n-k}{T}.
\end{equation*}
\end{proposition}
\begin{remark}
As one sees from the form of $A(T)$, as $T\to 0$ the matrix entries diverge. While it may be possible (using a logarithmic time change) to prove convergence for all $T\geq 0$, we opt to deal only with $T\in [T_0,T_1]$. After all, in light of Proposition~\ref{PropLargeL} we know the Gaussian process limit at time $T_0$, and thus do not need to start the SDEs above at $T=0$.
\end{remark}
\begin{proof}
From Corollary~\ref{CorPolynomials} we know that for large $\tau$ we have (setting $r=n+1-k$)
\begin{equation*}
y^{(n)}_k(\tau)=e^{-x^{(n)}_k(\tau)} = e^{-\tau} \frac{\det\left(\frac{\tau^{n+j-r-i}}{(n+j-r-i)!}\right)_{i,j=1}^r}{\det\left(\frac{\tau^{n+j-r-i+1}}{(n+j-r-i+1)!}\right)_{i,j=1}^{r-1}}(1+\Or(\tau^{-1})).
\end{equation*}
This is obvious from the intersecting path interpretation since in (\ref{eq3.19}) the leading term for large $\tau$ is the one with $c=i$.
We claim that
\begin{equation}\label{eq5.18}
\frac{\det\left(\frac{\tau^{n-r+j-i}}{(n-r+j-i)!}\right)_{i,j=1}^{r}}{\det\left(\frac{\tau^{n-r+1+j-i}}{(n-r+1+j-i)!}\right)_{i,j=1}^{r-1}} = \tau^{n+1-2r}\frac{(r-1)!}{(n-r)!}.
\end{equation}
It is easy to see that $\tau$ factors out of the determinants and gives the right power. What remains is to show the following evaluation formula (from which the desired ratio $M(n,r)/M(n,r-1)=(r-1)!/(n-r)!$ is clear) for the remaining Toeplitz determinant
\begin{equation*}
M(n,r)=\det\left(\frac{1}{(n-r+j-i)!}\right)_{i,j=1}^r =\prod_{k=0}^{r-1} \frac{(r-1-k)!}{(n-r+k)!}.
\end{equation*}
To prove this it suffices to check (as is readily done after some cancelations) that it satisfies the Desnanot--Jacobi identity, which is equivalent to
\begin{equation*}
M(n,r)M(n-2,r-2) = M(n,-1,r-1)^2 - M(n,r-1)M(n-2,r-1).
\end{equation*}

With the identity (\ref{eq5.18}) we have
\begin{equation*}
y^{(n)}_k(\tau) = e^{-\tau} \tau^{2k-n-1}\frac{(n-k)!}{(k-1)!} (1+\Or(\tau^{-1})).
\end{equation*}
Inserting this into the diffusion coefficient (\ref{eq4.33}) and into the drift coefficients in (\ref{eq4.34}) we get
\begin{equation*}
\sigma^n_k=1+\Or(\tau^{-1}),\quad
a^n_k=\tau^{-1}(k-1)(1+\Or(\tau^{-1})),\quad
b^n_k=\Or(\tau^{-2}),\quad c^n_k=\tau^{-1} (n-k)(1+\Or(\tau^{-1})).
\end{equation*}
Taking $L\to\infty$ and remembering that $d\tau= L dT$ we arrive at the linear system of SDEs given in (\ref{eqzetasde}).

It remains to deduce that the convergence of the drift and standard deviation matrices implies convergence of the full space-time process. To put this in a more general context, consider matrices $A_{\epsilon}(T)$ and $B_{\epsilon}(T)$ with entries uniformly bounded over all $\epsilon\in (0,1)$ and $T\in [T_0,T_1]$ and which have likewise bounded limits $A(T)$ and $B(T)$. Also, consider Gaussian initial data (at time $T_0$) given by $X_{\epsilon}(T_0)$ and assume it has a limit in distribution to some Gaussian vector $X(T_0)$. Let $X_{\e}(T)$ denote the solution to the system of SDEs $dX_{\epsilon}(T) = A_{\epsilon}(T)X_{\epsilon}(T) dT + B_{\epsilon}(T)d W_{\epsilon}(T)$ with initial data $X_{\epsilon}(T_0)$. Then, we claim, $X_{\epsilon}\to X$ in the topology on $\mathcal{C}\big([T_0,T_1],\mathcal{C}(\R)\big)$ where $X(T)$ solves the SDEs with matrices $A(T)$ and $B(T)$ with initial data $X(T_0)$. This result is not hard to prove. For instance, one can use the results from Appendix~\ref{AppGaussians} to write down the multi-time covariance of $X_{\epsilon}$. Given the hypotheses, this clearly has a limit as $\epsilon\to 0$, from which follows convergence of finite-dimensional-distributions to those of $X$. Tightness of $X_{\epsilon}$ is quite readily checked (for instance, also from the multi-time covariance).
\end{proof}

In order to determine the covariance of $\xi's$ at different times, we need to determine the propagator (see (\ref{eqAppPropagator})), i.e., the solution $Y(T)$ of
\begin{equation*}
\frac{d Y(T)}{dT}=A(T) Y(T),\quad Y(0)=\Id.
\end{equation*}
Since $A(T)$ depends on $T$ only by multiplication of $1/T$, let us write $A(T)=\widehat A T^{-1}$. Then, taking $S=\ln(T)$ we obtain $T \frac{d}{dT}=\frac{d}{dS}$, from which
\begin{equation}\label{eq5.12}
\frac{dY}{dS} = \widehat A Y.
\end{equation}
In what follows we will determine the covariance at a fixed time $T_0$ and then propagate from $T_0$, so that the $T^{-1}$ singularity of $A$ does not pose a problem. Notice that setting $S=\ln(T/T_0)$ we have still to find a solution of (\ref{eq5.12}) if we propagate from $T_0$.
\begin{lemma}\label{lemmaPropagator}
For any $1\leq k \leq n$, $1\leq k'\leq n'$, we have
\begin{equation*}
[\exp(S \widehat A)]_{(k,n),(k',n')}= e^{S(1-n)}(e^{S}-1)^{n'-n} \binom{k-1}{k'-1}\binom{n-k}{n'-k'}.
\end{equation*}
\end{lemma}
\begin{proof}
First make the change variables $n-k=x$, $n'-k'=x'$ and likewise $k-1=y$, $k'-1=y'$. The Markov chain on $n,k$ with generator $\widehat A$ turns into the Markov chain on $x,y$ which, in fact, factorizes into two independent Markov chains where $x$ transitions to $x-1$ at rate $x$ and $y$ transitions to $y-1$ at rate $y$. Call the generator of one of these two chains (as they have the same generator) $L$. Then our lemma reduces (through the change of variables) to showing that
$$
[\exp(S L)]_{x,x'} = \binom{x}{x'}e^{-Sx}(e^S-1)^{x-x'}.
$$
First note that for $S=0$, the above expression equals the indicator function for $x=x'$ (if $x'>x$ the binomial coefficient is zero, and if $x'<x$ the term $(e^S-1)^{x-x'}$ is zero). We must then show that for any function $f$, we have
$$
\frac{d}{dS}\big(\!\exp(S L) f\big)(x) = \big(L\exp(SL)f \big)(x).
$$
Checking this is elementary. Let $g(x) = \big(\!\exp(S L) f\big)(x)=\sum_{x'\leq x} \binom{x}{x'}e^{-Sx}(e^S-1)^{x-x'}f(x')$ and compute
$$
\frac{d}{dS}g(x) = \sum_{x'\leq x} \binom{x}{x'}e^{-Sx}(e^S-1)^{x-x'}f(x') \left(\frac{x-x' e^{S}}{e^S-1}\right)
$$
and likewise
$$
Lg(x) = x\sum_{x'\leq x-1} \binom{x-1}{x'}e^{-S(x-1)}(e^S-1)^{x-1-x'}f(x') - \sum_{x'\leq x} \binom{x}{x'}e^{-Sx}(e^S-1)^{x-x'}f(x').
$$
It is now easy to show, for each $x'$, the equality of the summands in $\frac{d}{dS}g(x)$ and $Lg(x)$, thus completing the proof.
\end{proof}
As a consequence of Lemma~\ref{lemmaPropagator} we readily have the propagator from time $T_0$.
\begin{corollary}\label{corPropagator}
The propagator matrix from time $T_0>0$ is given by
\begin{equation}\label{eqPropagator}
\left[Y^{T_0}(T)\right]_{(k,n),(k',n')}=\left(\frac{T_0}{T}\right)^{n-1}\left(\frac{T-T_0}{T_0}\right)^{n-n'} \binom{k-1}{k'-1}\binom{n-k}{n'-k'},
\end{equation}
for $1\leq k\leq n\leq N$ and $1\leq k'\leq n'\leq N$, and it solves
\begin{equation*}
\frac{dY^{T_0}}{dT}=A(T) Y^{T_0}(T),\quad Y(T_0)=\Id.
\end{equation*}
\end{corollary}
According to the general theory of Gaussian processes (see Appendix~\ref{AppGaussians}), if we denote by ${\rm Cov}(T_0)$ the covariance matrix at time $T_0$, then
\begin{equation}\label{eq5.29}
Y^{T_0}(T) {\rm Cov}(T_0)
\end{equation}
is the matrix of covariances between time moments $T_0$ and $T\geq T_0$.

\subsection{Bulk scaling limit at fixed time}
Now we want to determine the $N\to\infty$ limit of the covariance at a fixed time. Without loss of generality we can first consider $T=1$ and by scaling one can get the covariance for any fixed time $T>0$. Indeed, $\zeta^{(n)}_{k}(T)\stackrel{(d)}{=}\sqrt{T}\zeta^{(n)}_k(T)$ since it is the scaling limit as $L\to\infty$, see Proposition~\ref{PropLargeLsde}.
Our result is the following.
\begin{theorem}\label{thm5.1}
Let us denote
\begin{equation*}
\Omega(c,b)=c(1-2b+2 \I \sqrt{b(1-b)}).
\end{equation*}
Take any $a,b\in (0,1)$, $d>0$ and $c\in (0,d]$ such that $\Omega(c,b)\neq\Omega(d,a)$. Then, the large $N$ limit of the covariance is given by
\begin{equation}\label{eqCov5.31}
\begin{aligned}
&\lim_{N\to\infty} N {\rm Cov}\Big(\zeta^{(dN)}_{(1-a)dN}(T=1),\zeta^{(cN)}_{(1-b)cN}(T=1)\Big)\\
&=\frac{16}{(2\pi\I)^2}\int_{\overline \Omega(c,b)}^{\Omega(c,b)} dW \int_{\overline \Omega(d,a)}^{\Omega(d,a)} dZ \frac{1}{Z-W}\frac{1}{\sqrt{(W-\Omega(c,b))(W-\overline\Omega(c,b))}\sqrt{(Z-\Omega(d,a))(Z-\overline\Omega(d,a))}},
\end{aligned}
\end{equation}
where the path $Z$ stays to the right of $W$ (see Figure~\ref{FigContoursCovariance} for an illustration).
\end{theorem}
\begin{figure}
\begin{center}
\psfrag{Z}[c]{$Z$}
\psfrag{W}[c]{$W$}
\psfrag{c}[b]{$c$}
\psfrag{1}[b]{$d$}
\psfrag{w}[c]{$\Omega(c,b)$}
\psfrag{wb}[c]{$\overline \Omega(c,b)$}
\psfrag{z}[c]{$\Omega(d,a)$}
\psfrag{zb}[c]{$\overline \Omega(d,a)$}
\includegraphics[height=6cm]{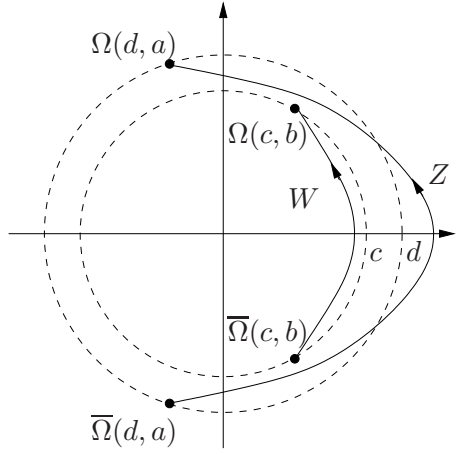}
\caption{Illustration of the integration contours of the limiting covariance (\ref{eqCov5.31}).}
\label{FigContoursCovariance}
\end{center}
\end{figure}

Surprisingly, there is a concise way of expressing the limiting covariance in terms of a complete elliptic integral.
\begin{proposition}\label{PropCovElliptic}
It holds
\begin{equation}\label{eqCovElliptic}
\begin{aligned}
&\frac{16}{(2\pi\I)^2}\int_{\overline \Omega(c,b)}^{\Omega(c,b)} dW \int_{\overline \Omega(d,a)}^{\Omega(d,a)} dZ \frac{1}{Z-W}\frac{1}{\sqrt{(W-\Omega(c,b))(W-\overline\Omega(c,b))}\sqrt{(Z-\Omega(d,a))(Z-\overline\Omega(d,a))}}\\
&=\frac{4 \kappa}{\pi \sqrt{\Im\Omega(d,a) \Im\Omega(c,b)}}\mathbb{K}(\kappa),
\end{aligned}
\end{equation}
with $\kappa=\frac{2\sqrt{\Re\Omega(d,a) \Re\Omega(c,b)}}{\sqrt{(\Re\Omega(d,a)-\Re\Omega(c,b))^2+(\Im\Omega(d,a)+\Im\Omega(c,b))^2}}$ and $\mathbb{K}$ is the complete elliptic integral of the first kind, namely
\begin{equation*}
\mathbb{K}(\kappa)=\int_0^1 \frac{d x}{\sqrt{(1-x^2)(1-\kappa^2 x^2)}}.
\end{equation*}
\end{proposition}
\begin{proof}
Let us first consider the case $\Re\Omega(d,a)>\Re\Omega(c,b)$. We use the notations $\Omega(d,a)=R_1+\I I_1$ and $\Omega(c,b)=R_2+\I I_2$. Doing the change of variables $Z=R_1+\I Y$ and $W=R_2+\I X$ we get
\begin{equation}\label{eq5.71}
\begin{aligned}
(\ref{eqCov}) &=\frac{4}{\pi^2} \int_{-I_2}^{I_2} d X \int_{-I_1}^{I_1} d Y \frac{1}{R_1-R_2+\I(Y-X)}\frac{1}{\sqrt{I_1^2-Y^2}\sqrt{I_2^2-X^2}}\\
&=\frac{4}{\pi^2} \int_0^\infty \d\lambda e^{-\lambda(R_1-R_2)} \int_{-I_2}^{I_2} d X \frac{e^{\I \lambda X}}{\sqrt{I_2^2-X^2}} \int_{-I_1}^{I_1} d Y \frac{e^{-\I\lambda Y}}{\sqrt{I_1^2-Y^2}}\\
&=4 \int_0^\infty \d\lambda e^{-\lambda(R_1-R_2)} J_0(I_1\lambda) J_0(I_2\lambda),
\end{aligned}
\end{equation}
with the Bessel function $J_0$ (see e.g.\ Eq.~2.5.3.3 of~\cite{Pru1}). Using the identity (see Eq.~2.12.38.1 of~\cite{Pru2}) we get
\begin{equation*}
(\ref{eq5.71})= \frac{4 \kappa}{\pi \sqrt{I_1 I_2}}\mathbb{K}(\kappa),\quad \kappa=\frac{2\sqrt{I_1 I_2}}{\sqrt{(R_1-R_2)^2+(I_1+I_2)^2}}.
\end{equation*}

For $\Re\Omega(d,a)\leq \Re\Omega(b,c)$, notice that for fixed $\Omega(d,a)$, both sides of the equations are analytic as a functions of $\Re\Omega(b,c)$ and $\Im\Omega(b,c)>0$ away from $\Omega(b,c)=\Omega(d,a)$. Thus the identity holds also for the general case by analytic continuation.
\end{proof}

\subsection{Proof of Theorem~\ref{thm5.1}}
By simple rescaling, it is enough to consider the case $d=1$, and clearly a shift by $1$ in the $r_i$ has no influence in the limiting result. Let us denote $M=cN$. We do the change of variables $w=N W$, $x=N X$, $u=N(U+W)$, and $z=N Z$, $y=N Y$, $v=N(V+Z)$ in (\ref{eq5.1}) with the result
\begin{multline}\label{eqCovBulk}
{\rm Cov}\Big(\zeta^{(N)}_{N-aN+1}(1),\zeta^{(M)}_{M-bM+1}(1)\Big)=\frac{N}{(2\pi\I)^2}\oint_{\Gamma_0} dW\oint_{\Gamma_{0,W}}dZ \frac{1}{Z-W}\\
\times\bigg(\int_0^\infty dX \frac{e^{N F(c,b,W,X)}}{X-W}\frac{1}{2\pi\I} \oint_{\Gamma_0} dU \frac{e^{N G(c,b,W,U)}}{W+U}\bigg)
\bigg(\int_0^\infty dY \frac{e^{N F(1,a,Z,Y)}}{Y-Z}\frac{1}{2\pi\I} \oint_{\Gamma_0} dV \frac{e^{N G(1,a,Z,V)}}{Z+V}\bigg),
\end{multline}
where
\begin{equation*}
\begin{aligned}
F(c,b,W,X)&=b c \ln(X-W)+(1-b)c\ln(X)-X,\\
G(c,b,W,U)&=-b c \ln(U)-(1-b)c\ln(W+U)+U+W.
\end{aligned}
\end{equation*}

\subsubsection{Analysis of the slow manifold}

In order to analyze the $N\to\infty$ limit of the covariance, we need to understand the properties of the functions $F$ and $G$. Notice that
\begin{equation}\label{eqRescFs}
F(c,b,W,X)=c F(1,b,W/c,X/c)-c \ln(c),\quad G(c,b,W,U)=c G(1,b,W/c,U/c)+ c \ln(c).
\end{equation}
Therefore if we analyze the situation for $c=1$, the case for $c<1$ is obtained by scaling linearly with $c$ all the variables $W,X,U$. Therefore, below we consider first $c=1$ and we write $F(b,W,X)$ for $F(1,b,W,X)$ (and similarly for $G$).

\begin{lemma}[Critical points]\label{lem5.1}
For given $W$, the critical points of $F$ and $G$ are given by
\begin{equation*}
X_\pm=\frac{1+W\pm \Delta}{2},\quad\textrm{and}\quad U_\pm=X_\pm-W=\frac{1-W\pm \Delta}{2},
\end{equation*}
where $\Delta=\sqrt{(1-W)^2+4bW}$.
\end{lemma}
\begin{proof}
It is an elementary computation.
\end{proof}

The critical points are double critical point only for two values of $W$.
\begin{corollary}[Double critical points]\label{cor5.2}
We have $U_+=U_-$ (and $X_+=X_-$) if and only if \mbox{$W\in\{W_c,\overline{W}_c\}$}, where
\begin{equation*}
W_c=\Omega(1,b)\equiv1-2b+2\I \sqrt{b(1-b)}.
\end{equation*}
Further, $|W_c|=1$ and we define $\varphi_c=\arg(W_c)$, and
\begin{equation*}
\begin{aligned}
U_c&=U_\pm(W_c)=b-\I \sqrt{b(1-b)},\\
X_c&=X_\pm(W_c)=1-b+\I\sqrt{b(1-b)}.
\end{aligned}
\end{equation*}
\end{corollary}
\begin{proof}
It is an elementary computation as well.
\end{proof}

We have an integral over $(W,X,U)$ of the function $e^{M(F(b,W,X)+G(b,W,U))}$ times $M$-independent terms. If we compute the eigenvalues of the Jacobi matrix of $\Phi(W,X,U):=F(b,W,X)+G(b,W,U)$ one sees that there is a zero eigenvalue. This means that for the steep descent analysis there is a slow mode and two fast modes. Thus we have to study the submanifold where $\Re \Phi(W,X,U)=0$, that we call \emph{slow manifold}. We will then take $W$ on this submanifold and integrate with respect to $U,X$ (which will be essentially Gaussian integrals) and only after that we integrate over $W$. The next goal is thus to say something about the slow manifold.

\begin{lemma}\label{lem5.3}
We have the following identities. First of all,
\begin{equation}\label{eq5.3}
F(b,W,X_\pm)+G(b,W,U_\pm)=0.
\end{equation}
Denoting $H(W):=F(b,W,X_+)+G(b,W,U_-)$ and $\tilde H(W):=F(b,W,X_-)+G(b,W,U_+)$ we have
\begin{equation}\label{eq5.4}
H(W)+\tilde H(W)=0.
\end{equation}
\end{lemma}
\begin{proof}
It is simple to verify the identity (\ref{eq5.3}) and then (\ref{eq5.4}) follows directly from (\ref{eq5.3}).
\end{proof}

The next goal is to see that $W$ on the slow manifold $(W,X_\pm(W),U_\pm(W))$ can be parameterized as a function in polar coordinates and it takes values inside the circle of radius one.
\begin{lemma}\label{lem5.4}
Let $W=r e^{\I\varphi}$. Then $\Delta^2=\alpha+\I\beta=\rho e^{\I\theta}$, with
\begin{equation*}
\begin{aligned}
\alpha&=2 r \cos(\varphi)(r\cos(\varphi)-1+2b)+1-r^2,\\
\beta&= 2 r \sin(\varphi)(r \cos(\varphi)-1+2b),\\
\rho&=\sqrt{\alpha^2+\beta^2},\quad \theta=2\arctan\left(\frac{\rho-\alpha}{\beta}\right).
\end{aligned}
\end{equation*}
Further, it holds
\begin{equation*}
\Re \Delta=\frac{\sqrt{\rho}}{\sqrt{1+\left(\frac{\rho-\alpha}{\beta}\right)^2}},\quad \Im\Delta=\frac{\rho-\alpha}{\beta}\frac{\sqrt{\rho}}{\sqrt{1+\left(\frac{\rho-\alpha}{\beta}\right)^2}}.
\end{equation*}
\end{lemma}
\begin{proof}
An elementary computation gives the values of $\alpha$ and $\beta$. Then, using $\cos(\arctan(x))=1/\sqrt{1+x^2}$ and $\sin(\arctan(x))=x/\sqrt{1+x^2}$ we get the formulas for $\Delta=\sqrt{\rho} e^{\I\theta/2}$.
\end{proof}

To understand where $\Re H(W)=0$, first notice that
\begin{equation*}
H(W)=b\ln\left(\frac{1-W+\Delta}{1-W-\Delta}\right)+(1-b)\ln\left(\frac{1+W+\Delta}{1+W-\Delta}\right)-\Delta.
\end{equation*}
Thus at least for the two critical points $W\in \{W_c,\overline{W}_c\}$ we have $H(W)=0$. Further,
\begin{equation*}
\frac{d H}{dW}=-\frac{\Delta}{W}.
\end{equation*}

By Corollary~\ref{cor5.2} we thus know that $\frac{d H}{dW}=0$ if and only if $W\in\{W_c,\overline{W}_c\}$, i.e., where $\Delta=0$.
\begin{lemma}\label{lem5.5}
We have
\begin{equation*}
\Re\Delta=0\textrm{ if and only if }\beta=0\textrm{ and }\alpha\leq 0.
\end{equation*}
\end{lemma}
\begin{proof}
If $\alpha=\beta=0$, then $\rho=0$ and $\Delta=0$. Otherwise assume $\alpha\neq 0$. Then $\Re\Delta=0$ possibly only if $\beta\downarrow 0$.\\
First case: $\alpha>0$. Then as $\beta\downarrow 0$, $\Re\Delta\simeq \sqrt{\alpha}+O(\beta^2)$, thus $\Re\Delta\not\to 0$,\\
Second case: $\alpha<0$. Then as $\beta\downarrow 0$, $\Re\Delta\simeq \beta/(2\sqrt{|\alpha|})$ and $\Im\Delta\simeq \beta/(2\sqrt{|\alpha|})$, which gives the result.
\end{proof}

From this it follows that
\begin{corollary}\label{cor5.6}
It holds
\begin{equation*}
\Re\Delta=0\textrm{ if and only if }r \cos\varphi = 1-2b\textrm{ and }r>1.
\end{equation*}
\end{corollary}
\begin{proof}
$\beta=0$ if and only if $r\cos\varphi = 1-2 b$. Setting this into $\alpha$ leads to $\alpha=1-r^2$, which is negative if and only if $r>1$.
\end{proof}

Next we consider the derivative of $\Re H$ with respect to $r$ and $\varphi$.
\begin{lemma}\label{lem5.7}
We have
\begin{equation*}
\frac{d\Re H}{dr} = -\frac{\Re\Delta}{r}\leq 0\quad\textrm{and}\quad \frac{d\Re H}{d\varphi}=\Im\Delta.
\end{equation*}
In particular, on the circle of radius one, $W=e^{\I\varphi}$,
$\frac{d\Re H}{d\varphi}>0$ on $\varphi\in (0,\varphi_c)$ and $\frac{d\Re H}{d\varphi}<0$ on $\varphi\in (\varphi_c,\pi)$. This implies that on the circle of radius one, $\Re H<0$ except at the two critical points $\{W_c,\overline{W}_c\}$.
\end{lemma}
\begin{proof}
The identities follow from
\begin{equation*}
\frac{d\Re H}{dr}+\I\frac{d\Im H}{dr} = \frac{d H}{dr} = \frac{dH}{dW}\frac{dW}{dr}=-\frac{\Delta}{r}=-\frac{\Re\Delta}{r}-\I\frac{\Im\Delta}{r},
\end{equation*}
and
\begin{equation*}
\frac{d\Re H}{d\varphi}+\I\frac{d\Im H}{d\varphi} = \frac{d H}{d\varphi} = \frac{dH}{dW}\frac{dW}{d\varphi}=-\I\Delta=-\I\Re\Delta+\Im\Delta.
\end{equation*}
Using the formula for $\Im\Delta$ in Lemma~\ref{lem5.4} we get $\Im\Delta>0$ if and only if $\beta>0$, i.e., if and only if $r\cos\varphi>1-2b$, and for $r=1$, $\cos\varphi>1-2b$ if and only if $\varphi\in (0,\varphi_c)$ (it suffices to consider $\varphi\in [0,\pi]$ only by symmetry).
\end{proof}

\begin{lemma}\label{lem5.8}
We have $\lim_{r\to 0} \Re H=\infty$.
\end{lemma}
\begin{proof}
As $r\to 0$, $U_+\to 1$, $U_-\to -br e^{\I\varphi}$, $X_+\to 1$, $X_-\to (1-b)r e^{\I\varphi}$. Thus, as $r\to 0$, $\Re H\sim -b\ln(b)-(1-b)\ln(1-b)-\ln(r)\to \infty$.
\end{proof}

The above lemmas imply the following.
\begin{proposition}\label{prop5.9}
There exists a simple path $\Pi=\Pi(b)$ around $0$ such that for $W\in\Pi$
\begin{enumerate}
 \item $\Re H(W)=0$,
 \item $W=r(\varphi) e^{\I\varphi}$: where $r(\varphi)$ is a unique solution of $\Re H(W)=0$ in $[0,1]$,
 \item $r(\varphi_c) e^{\I\varphi_c}=W_c$,
\end{enumerate}
see Figure~\ref{FigSlowManyfold} for an illustration. We denote by $\Pi_+(b)$ the part of $\Pi(b)$ with $\varphi\in [-\varphi_c,\varphi_c]$ and $\Pi_-(b)$ the remaining part.
\begin{figure}
\begin{center}
\psfrag{+}[c]{$+$}
\psfrag{-}[c]{$-$}
\psfrag{wc}[l]{$W_c$}
\psfrag{wcb}[l]{$\overline{W}_c$}
\psfrag{gamma}[l]{$\Pi$}
\includegraphics[height=5cm]{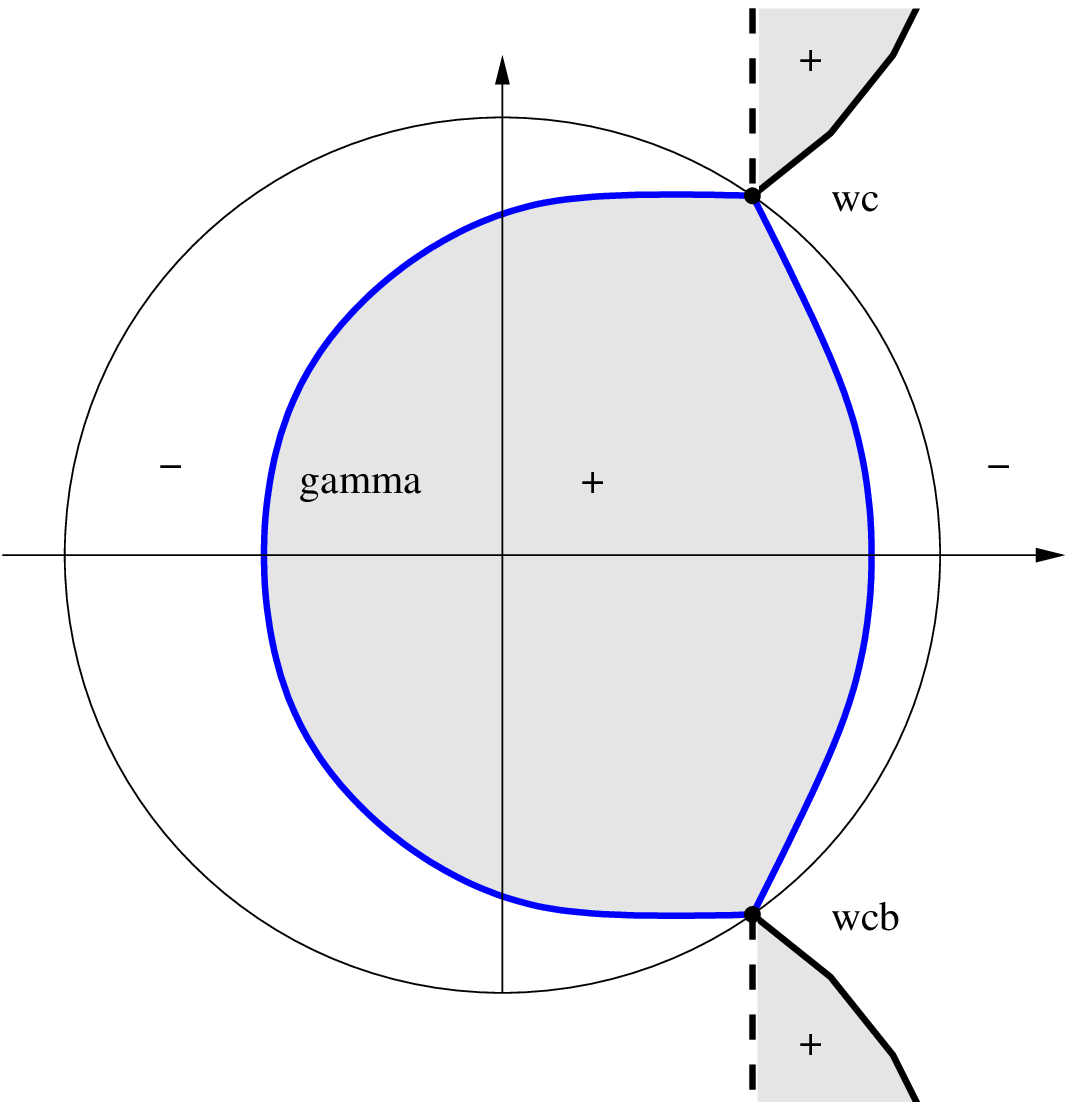}
\caption{The thick blue curve is the path $\Pi$ of Proposition~\ref{prop5.9}. On the solid black lines $\Re H=0$ as well. The dashed black lines are the discontinuities of $\Re H$, which correspond to $\Re\Delta=0$. The gray regions are the ones with positive real part of $H$.}
\label{FigSlowManyfold}
\end{center}
\end{figure}
\end{proposition}
\begin{proof}
By Lemma~\ref{lem5.7} we know that $\Re H(r e^{\I\varphi})$ is monotonically decreasing as $r$ increases, and goes from $\infty$ at $r=0$ (Lemma~\ref{lem5.8}) to $\Re H(W)\leq 0$ for $W$ on a circle of radius $1$ (with equality if and only if $\{W_c,\overline{W}_c\}$) (Lemma~\ref{lem5.8}), thus it has to cross $0$ at some values $r=r(\varphi)\in (0,1]$ (with $r(\varphi)=1$ if and only if $\varphi=\pm \varphi_c$).
\end{proof}

Along the path $\Pi$ of Proposition~\ref{prop5.9}, the real part of $H$ is equal to zero and the imaginary part is not constant.
\begin{lemma}\label{lem5.10}
Along the path $\Pi$ of Proposition~\ref{prop5.9},
\begin{equation*}
\frac{d\Im H}{d\varphi}=-\frac{|\Delta|^2}{\Re\Delta}\leq 0,
\end{equation*}
with equality if and only if $\Delta=0$.
\end{lemma}
\begin{proof}
Let $r=r(\varphi)$ and $W=r(\varphi) e^{\I\varphi}$. Then,
\begin{equation*}
\frac{d\Re H}{d\varphi}+\I\frac{d\Im H}{d\varphi} = \frac{d H}{d\varphi} = \frac{dH}{dW}\frac{dW}{d\varphi} = \left(\Im\Delta-\Re\Delta \frac{d\ln r(\varphi)}{d\varphi}\right)-\I \left(\Re\Delta+\Im\Delta \frac{d\ln r(\varphi)}{d\varphi}\right).
\end{equation*}
Since the real part is by definition of $r(\varphi)$ equal to $0$, we have $\frac{d\ln r(\varphi)}{d\varphi}=\frac{\Im\Delta}{\Re\Delta}$. Thus
\begin{equation*}
\frac{d\Im H}{d\varphi} = -\frac{|\Delta|^2}{\Re\Delta}.
\end{equation*}
\end{proof}

Finally, as a simple corollary we have:
\begin{corollary}\label{cor5.11}
Along the path $\Pi$ of Proposition~\ref{prop5.9},
\begin{equation*}
\Re F(b,W,X_+)=-\Re G(b,W,U_-)=\Re F(b,W,X_-)=-\Re G(b,W,U_+).
\end{equation*}
\end{corollary}

As we shall see, when integrating over $W$ and $Z$, we will have to choose in some situations contours which do not satisfy the original constraints ($Z$ encloses $c W$). In these cases we will have to control the residue term as well. For that purpose we need to see where the paths $\Pi$ for different values of $a,b,c$ intersects. Since in the following the path $\Pi$ as well as $W_c$ and related quantities depends on $b$, we will write $b$ explicitly when needed.
\begin{lemma}\label{lem5.16}
Fix $0<a<b<1$, from which we know that $\varphi_c(a)<\varphi_c(b)$. We have:
\begin{itemize}
\item[(a)] For $\varphi\in[0,\varphi_c(a)]$, $r_a(\varphi)>r_b(\varphi)$.
\item[(b)] There exists a unique $\varphi\in(\varphi_c(a),\varphi_c(b))$ such that $r_a(\varphi)=r_b(\varphi)$.
\item[(c)] For $\varphi\in[\varphi_c(b),\pi]$, $r_a(\varphi)<r_b(\varphi)$.
\end{itemize}
See Figure~\ref{FigNonIntContours} for an illustration.
\end{lemma}
\begin{proof}
By Lemma~\ref{lem5.7} and Lemma~\ref{lem5.4} we get
\begin{equation*}
\frac{d r}{d \varphi} = r \frac{\Im\Delta}{\Re\Delta} = r \frac{\sqrt{\alpha^2+\beta^2}-\alpha}{\beta}.
\end{equation*}
Let $\varphi\in (0,\varphi_c(b))$. Then $\beta>0$ and
\begin{equation}\label{eq5.44}
\frac{d r}{d \varphi} = r \left(\sqrt{(\alpha/\beta)^2+1}-\alpha/\beta\right)>0,
\end{equation}
where
\begin{equation*}
\frac{\alpha}{\beta}=\frac{1}{\tan\varphi} + \frac{1-r^2}{2 r \sin(\varphi)(r \cos(\varphi)-1+2b)}.
\end{equation*}
From (\ref{eq5.44}) together with the fact that $r(\varphi)=1$ only at $\varphi=\varphi_c$, we have that $1=r_a(\varphi_c(a))>r_b(\varphi_c(a))$.

Similarly, for $\varphi\in (\varphi_c(b),\pi)$, $\beta<0$ and
\begin{equation}\label{eq5.44b}
\frac{d r}{d \varphi} = -r \left(\sqrt{(\alpha/|\beta|)^2+1}-\alpha/|\beta|\right)<0,
\end{equation}
where
\begin{equation*}
\frac{\alpha}{\beta}=\frac{1}{\tan\varphi} - \frac{1-r^2}{2 r \sin(\varphi)(r \cos(\varphi)-1+2b)}.
\end{equation*}
From (\ref{eq5.44}) together with the fact that $r(\varphi)=1$ only at $\varphi=\varphi_c$, we have that $1=r_b(\varphi_c(b))>r_a(\varphi_c(b))$.

\emph{Case (b):} (\ref{eq5.44}) and (\ref{eq5.44b}) imply that for $\varphi\in(\varphi_c(a),\varphi_c(b))$, $r_a$ is (strictly) decreasing starting from $1$ and $r_b$ is (strictly) increasing ending at $1$. Thus, by continuity of $r(\varphi)$, there is a unique intersection point of $r_a(\varphi)$ and $r_b(\varphi)$.

\emph{Case (a):} At $(r,\varphi)\in [0,1]\times (0,\varphi_c(a))$, $b\mapsto r\cos(\varphi)-1+2b$ is increasing in $b$. Thus $b\mapsto \alpha/\beta$ is decreasing, but since $x\mapsto \sqrt{x^2+1}-x$ is also decreasing, this implies that $b\mapsto r \frac{\sqrt{\alpha^2+\beta^2}-\alpha}{\beta}$ is increasing in $b$. Thus, if there is an intersection of the paths $\Pi_a$ and $\Pi_b$ for this range of angle, say at $(r_0,\varphi_0)$, then at this point
\begin{equation*}
0<\frac{d r_a}{d \varphi}<\frac{d r_b}{d \varphi}.
\end{equation*}
This would then imply that $r_a(\varphi_c(a))<r_b(\varphi_c(a))$, which is false. Therefore by contradiction we have shown part (a).

\emph{Case (c):} It is similar to case (a). At $(r,\varphi)\in [0,1]\times (\varphi_c(b),\pi)$, $b\mapsto r\cos(\varphi)-1+2b$ is increasing in $b$. Thus $b\mapsto \alpha/\beta$ is increasing as well. The map $x\mapsto -(\sqrt{x^2+1}-x)$ is increasing, which implies that $b\mapsto r \frac{\sqrt{\alpha^2+\beta^2}-\alpha}{\beta}$ is increasing in $b$. Thus, if there is an intersection of the paths $\Pi_a$ and $\Pi_b$ for this range of angle, say at $(r_0,\varphi_0)$, then at this point
\begin{equation*}
\frac{d r_a}{d \varphi}<\frac{d r_b}{d \varphi}<0.
\end{equation*}
This would then imply that $r_b(\varphi_c(b))<r_a(\varphi_c(b))$, which is false. Therefore by contradiction we have shown part (c) as well.
\end{proof}

\begin{figure}
\psfrag{w1}[r]{\small $\Omega(1,b)$}
\psfrag{w2}[l]{\small $\Omega(1,a)$}
\psfrag{Ppa}[l]{$\Pi_+(a)$}
\psfrag{Ppb}[l]{$\Pi_+(b)$}
\psfrag{Pma}[r]{$\Pi_-(a)$}
\psfrag{Pmb}[r]{$\Pi_-(b)$}
\includegraphics[height=5cm]{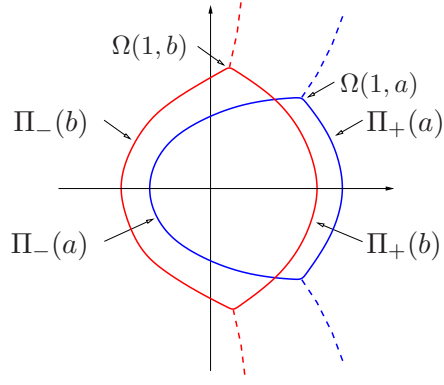}
\caption{An illustration of Lemma~\ref{lem5.16}. For $a=0.2$ (blue) and $b=0.4$ (red) we plot the lines where $\Re H(W)=0$. Inside the closed contours $\Re H(W)>0$. Along the dashed contours, the $\Re H(W)=0$ as well.}
\label{FigNonIntContours}
\end{figure}

Here is a corollary.
\begin{corollary}\label{CorIntersections}
If $\Pi_+(a)\cap c\Pi(b)\neq \varnothing$, then $\Pi_-(a)\cap c\Pi(b)=\varnothing$.
\end{corollary}
\begin{proof}
There are a few cases to be considered:\\
(a) $\Pi_+(a)\cap c\Pi_+(b)\neq \varnothing$ and $b>a$ can not occur together. Otherwise, by continuity $\Pi_+(a)\cap \Pi_+(b)\neq\varnothing$ as well, which contradicts Lemma~\ref{lem5.16}(a).\\
(b) $\Pi_+(a)\cap c\Pi_+(b)\neq \varnothing$ and $b<a$. In this case by Lemma~\ref{lem5.16} there is no intersection of $\Pi_-(a)$ and $\Pi(b)$, which implies that $\Pi_-(a)\cap c\Pi(b)=\varnothing$ as well.\\
(c) $\Pi_+(a)\cap c\Pi_+(b)=\varnothing$ but $\Pi_+(a)\cap c\Pi_-(b)\neq \varnothing$. In this case also $\Pi_+(a)\cap \Pi_-(b)\neq \varnothing$, meaning by Lemma~\ref{lem5.16} that $b<a$, which in turns implies $\Pi_-(a)\cap \Pi(b)=\varnothing$, thus also $\Pi_-(a)\cap c\Pi(b)=\varnothing$.
\end{proof}

\subsubsection{Analysis of the fast manifolds}

For a given $W$ on $\Pi$, we deform the integration paths for $X$ and $U$ so that they pass through some critical points $X_\pm(W)$ and $U_\pm(W)$. Through which critical points the paths must pass is determined by the fact that $X$ needs to start at $0$ and end at $\infty$ and $U$ is closed loop around $0$ not including $-W$. There is a difference on whether $W$ is in the part of $\Pi$ between $\overline{W}_c$ and $W_c$ or not. As we shall see, in the first case, the paths needs to pass through both critical points, while in the second case $X$ passes only through $X_+(W)$ and $U$ through $U_-(W)$.

\begin{proposition}\label{Prop5.12}
Let us denote by $\Gamma_X:=\{X\in\C \,|\, \Re F(b,W,X)=\Re F(b,W,X_\pm(W))\}$ and similarly $\Gamma_U:=\{U\in\C \,|\, \Re G(b,W,U)=\Re G(b,W,U_\pm(W))\}$. Parameterize $W=r(\varphi) e^{\I \varphi}$ on $\Pi$ as in Proposition~\ref{prop5.9}. Let $\gamma_X\subset\C$ be any deformation of the path from $0$ to $\infty$ of the real line and $\gamma_U\subset\C$ any simple counterclockwise oriented loop around $0$ but not including the point $U=-W$.
\begin{enumerate}
 \item For $\varphi\in [0,\varphi_c]$, any path $\gamma_X$ which satisfies $\Re F(b,W,X)\leq \Re F(b,W,X_\pm(W))$ need to pass through both critical points $X_+(W)$ and $X_-(W)$.
 \item For $\varphi\in [\varphi_c,\pi]$, any path $\gamma_X$ which satisfies $\Re F(b,W,X)\leq \Re F(b,W,X_\pm(W))$ need to pass through critical points $X_+(W)$ but not through $X_-(W)$.
 \item For $\varphi\in [0,\varphi_c]$, any path $\gamma_U$ which satisfies $\Re G(b,W,U)\leq \Re G(b,W,U_\pm(W))$ need to pass through both critical points $U_+(W)$ and $U_-(W)$.
 \item For $\varphi\in [\varphi_c,\pi]$, any path $\gamma_U$ which satisfies $\Re G(b,W,U)\leq \Re G(b,W,U_\pm(W))$ need to pass through critical points $U_-(W)$ but not through $U_+(W)$.
\end{enumerate}
\end{proposition}
\begin{proof}
First we consider the extreme cases, $\varphi=0$ and $\varphi=\pi$.

\emph{Case 1: $\varphi=\pi$}. In this case $W=-r\in\R_-$ and thus $\Delta=\sqrt{(1+r)^2-4br}>1-r=1+W$ (since $b\leq 1$), but also $\Delta=\sqrt{(1-r)^2+4(1-b)r}<1+r=1-W$. Thus $X_-=(1+W-\Delta)/2\in (W,0)$ and $X_+=(1+W+\Delta)/2>0$. Thus we have obtained
\begin{equation*}
W<X_-<0<X_+.
\end{equation*}
We know that the function $\Re F(b,W,X)=-\infty$ at $X=W,0,\infty$ and it goes to $\infty$ when going to infinity in the directions with negative real part. For $W$ real $\Gamma_X$ is symmetric with respect to complex conjugation. Further there are only two points where branches of $\Gamma_X$ comes together, namely at $X=X_\pm$. Thus two of the branches that leave the point $X_-$ need to close on the left of $W$, the other two have to go around $0$ and meet at $X_+$, from which point they open up asymptotically in the vertical direction. See Figure~\ref{FigConstantReHCases1and2} (left) for an illustration of $\Gamma_X$.
\begin{figure}
\centering
\begin{subfigure}
 \centering
\psfrag{W}[c]{$W$}
\psfrag{0}[c]{$0$}
\psfrag{Xp}[c]{$X_+$}
\psfrag{Xm}[c]{$X_-$}
\psfrag{Gx}[c]{$\Gamma_X$}
\includegraphics[height=4cm]{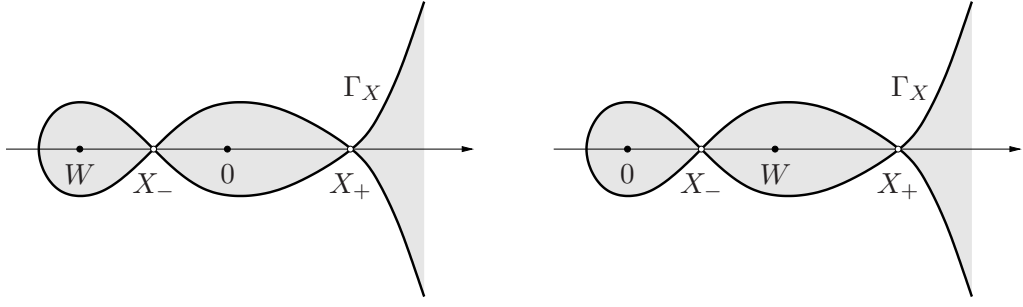}
\end{subfigure}
\qquad
\begin{subfigure}
 \centering
\psfrag{W}[c]{$0$}
\psfrag{0}[c]{$W$}
\psfrag{Xp}[c]{$X_+$}
\psfrag{Xm}[c]{$X_-$}
\psfrag{Gx}[c]{$\Gamma_X$}
\includegraphics[height=4cm]{FigConstantReHCase1.eps}
\end{subfigure}
\caption{Illustration of the path $\Gamma_X$ for the case $W\in\R_-$ (left) and $W\in\R_+$ (right). In the shaded regions $\Re F(X)< \Re F(X_\pm)$.}
\label{FigConstantReHCases1and2}
\end{figure}

\emph{Case 2: $\varphi=0$}. In this case $W=r\in\R_+$ and thus $\Delta=\sqrt{(1+r)^2-4(1-b)r}<1+r=1+W$, but also $\Delta>1-W$. Thus $X_-=(1+W-\Delta)/2\in (0,W)$ and $X_+=(1+W+\Delta)/2>W$. Thus we have obtained
\begin{equation*}
0<W<X_-<X_+.
\end{equation*}
The same argument as in Case 1, but with the roles of $W$ and $0$ interchanged leads to $\Gamma_X$ as in Figure~\ref{FigConstantReHCases1and2} (right).

\emph{Case 3: $\varphi=\varphi_c$}. In this case $W=W_c$ and thus $\Delta=0$. It is the situation where the two critical points comes together, see Figure~\ref{FigConstantReHCase3}.
\begin{figure}
\centering
\psfrag{W}[c]{$W$}
\psfrag{0}[c]{$0$}
\psfrag{Xpm}[c]{$X_\pm$}
\psfrag{Gx}[c]{$\Gamma_X$}
\includegraphics[height=4cm]{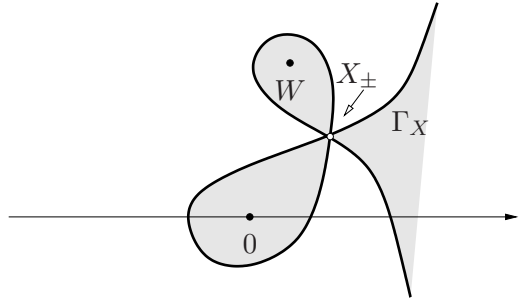}
\caption{Illustration of the path $\Gamma_X$ for the case $W=W_c$, for which $X_+=X_-$. In the shaded regions $\Re F(X)< \Re F(X_\pm)$.}
\label{FigConstantReHCase3}
\end{figure}

By continuity, to go from the topological situation of Case 1 to the one of Case 2, there is only one possibility, namely when $X_+$ and $X_-$ comes together. But this happens only at the value of $\varphi$ of Case 3. Statements (1) and (2) then follows from these observations.

(3) and (4) are proved similarly to (1) and (2).
\end{proof}

\begin{remark}
The paths $\Gamma_X$ and $\Gamma_U$ are just a shift by $W$ of each other. Further, at $X_\pm$ the branches of $\Gamma_X$ form a cross with $\pi/2$ degrees, except for the case when $X_+=X_-$ where the angles are $\pi/3$ between each branch.
\end{remark}

\begin{proof}[Proof of Theorem~\ref{thm5.1}]
Let us recall the formula for the covariance that we want to compute, namely (\ref{eqCovBulk}):
\begin{multline}
{\rm Cov}\Big(\zeta^{(N)}_{N-aN+1}(1),\zeta^{(M)}_{M-bM+1}(1)\Big)=\frac{N}{(2\pi\I)^2}\oint_{\Gamma_0} dW\oint_{\Gamma_{0,W}}dZ \frac{1}{Z-W}\\
\times\bigg(\int_0^\infty dX \frac{e^{N F(c,b,W,X)}}{X-W}\frac{1}{2\pi\I} \oint_{\Gamma_0} dU \frac{e^{N G(c,b,W,U)}}{W+U}\bigg)
\bigg(\int_0^\infty dY \frac{e^{N F(1,a,Z,Y)}}{Y-Z}\frac{1}{2\pi\I} \oint_{\Gamma_0} dV \frac{e^{N G(1,a,Z,V)}}{Z+V}\bigg).
\end{multline}

\textbf{First case: $\Pi_+(a)\cap c\,\Pi(b) = \varnothing$.} First we consider the case where the path $\Pi_+(a)$ for $Z$ does not cross the path for $W$ ($\Pi_\pm(b)$ are defined in Proposition~\ref{prop5.9}).

We divide the integrations in four terms:
\begin{equation}\label{eq5.58}
\oint_{\Gamma_0} dW \oint_{\Gamma_{0,W}} dZ = \int_{W_c}^{\overline W_c}dW \int_{Z_c}^{\overline Z_c} dZ+\int_{\overline W_c}^{W_c}dW \int_{Z_c}^{\overline Z_c} dZ+\int_{W_c}^{\overline W_c}dW \int_{Z_c}^{\overline Z_c} dZ+\int_{\overline W_c}^{W_c}dW \int_{Z_c}^{\overline Z_c} dZ
\end{equation}
where $W_c=\Omega(c,b)$ and $Z_c=\Omega(1,a)$. Here the integration path from $\overline W_c$ to $W_c$ passes to the right of the origin, while from $W_c$ to $\overline W_c$ to the left of the origin (similarly for $Z$).

By Proposition~\ref{prop5.9} we can choose the path for $Z$ from $Z_c$ to $\overline Z_c$ satisfying $\Re H(Z)<0$ and staying outside $c\Pi(b)$ (the initial choice of the path for $W$). Thus in this case we do not have to deal with potential residue terms arising from the factor $1/(Z-W)$ in the integrand.

Now we discuss which contributions remain in the $N\to\infty$ limit from the integration over $W$. The same argument goes through for the contributions for the integration over $Z$. For simplicity we consider $c=1$, as in the generic case one has just to replace $W\to c W$ in the final expression.

\medskip
\textbf{(a) Contributions for $W\in \Pi_-(b)$ and $Z\in\Pi_-(a)$ (or a deformation of them).} As discussed in Proposition~\ref{Prop5.12}, the paths for $X$ and $U$ can be chosen to pass by only one of the critical points, namely $X_+$ and $U_-$, see Figure~\ref{FigIntegrationCase2} for an illustration.
\begin{figure}
\centering
\psfrag{W}[c]{$W$}
\psfrag{-W}[c]{$-W$}
\psfrag{0}[c]{$0$}
\psfrag{Xp}[c]{$X_+$}
\psfrag{Xm}[c]{$X_-$}
\psfrag{gx}[c]{$\gamma_X$}
\psfrag{Up}[l]{$U_+$}
\psfrag{Um}[c]{$U_-$}
\psfrag{gu}[c]{$\gamma_U$}
\includegraphics[height=5cm]{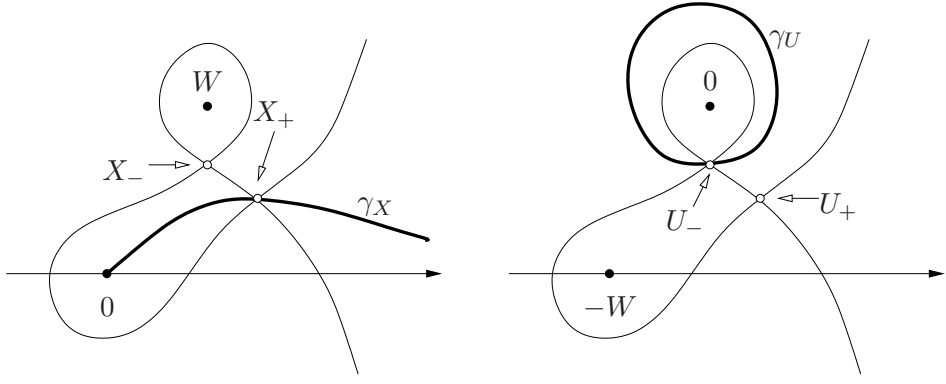}
\caption{Illustration of the integration paths $\gamma_X$ and $\gamma_U$ for the case $\varphi\in (\varphi_c,\pi]$.}
\label{FigIntegrationCase2}
\end{figure}

Then there are paths $\gamma_X$ (resp.\ $\gamma_U$) such that $\Re F(b,W,X)$ reaches its maximum at $X=X_+$ (resp.\ $\Re G(b,W,U)$ reaches its maximum at $U=U_-$). Further, due to the linear term in $X$, when $X\to\infty$ along $\gamma_X$, $\Re F(b,W,X)$ decreases linearly in the real part of $X$. Thus $\gamma_X$ and $\gamma_U$ can be taken to be steep descent paths and the leading contribution to the integrals comes from a $\delta$-neighborhood of $X_+$ and $U_-$; the contribution from the remainder is $\Or(e^{-N c(\delta)})$ for some $c(\delta)>0$ (with $c(\delta)\sim \delta^2$ as $\delta\to 0$ away from the critical points and $c(\delta)\sim \delta^3$ as $\delta\to 0$ at the critical point). This does not change if we deform the path for $W$ to a path $\widetilde \Pi_-(p)$ to stay outside $\Pi_-(p)$ (and to the left of the critical points, i.e., to the left of the dashed lines in Figure~\ref{FigSlowManyfold}).

It remains to determine the contribution coming from the neighborhood of $(X_+,U_-)$ (as outside a $\delta$-neighborhood the contribution is exponentially small in $N$). As $X_+$ and $U_-$ are (simple) critical points of $F(b,W,X)$ and $G(b,W,U)$ respectively, they are determined by the Gaussian integrals up to smaller order in $N$. We have
\begin{equation}
F(b,W,X)\simeq F(b,W,X_+)-\frac12 \sigma_+ (X-X_+)^2,\quad G(b,W,U)\simeq G(b,W,U_-)+\frac12 \sigma_- (U-U_-)^2,
\end{equation}
where $\sigma_\pm$ is defined by
\begin{equation}\label{eq5.26}
\sigma_{\pm}(W)=- \frac{\partial^2}{\partial X^2} F(b,W,X)\big|_{X=X_\pm}
=\frac{1-b}{X_\pm^2}+\frac{b}{U_\pm^2}.
\end{equation}
Thus, the integrals over $\gamma_X$ and $\gamma_U$ will be dominated by the Gaussian integrals around the critical points, with the result
\begin{equation}\label{eq5.26b}
\frac{1}{N}\frac{e^{N \Re F(b,W,X_+)}}{(X_+-W) \sqrt{\sigma_+(W)}}\frac{e^{N \Re G(b,W,U_-)}}{(W+U_-) \sqrt{\sigma_-(W)}}(1+\Or(N^{-1/2})).
\end{equation}
The error term $\Or(N^{-1/2})$ is the correction to the Gaussian integrals coming from the higher order non-zero terms. The term in the exponential is $N H(W)$, where $H(W)$ was defined in Lemma~\ref{lem5.3}.

Let $\widetilde \Pi_-^\e(b)$ the region of $\widetilde \Pi_-(b)$ at distance $\e$ from $W_c=\Omega(1,b)$. For $W$ at distance $\e$ from the critical points, by Lemma~\ref{lem5.7} we get $\Re H(W)\leq -C(\e)<0$ with $C(\e)\sim\e^{3/2}$ as $\e\to 0$ (the reason being that $\Delta^2\simeq 4\I \sqrt{b(1-b)} (W-W_c)$ as $W\to W_c$). Thus the contribution of the integral over $\widetilde \Pi_-(b)\setminus \widetilde \Pi_-^\e(b)$ is exponentially small in $N$, bounded by $\Or(e^{- c \e^{3/2}})$ for some $c>0$.

To estimate the integral over $\widetilde \Pi_-^\e(b)$, we use that
\begin{equation}\label{eq5.24}
\sigma_\pm(W)\simeq \frac{2\sqrt{-\I}(1-b)^{3/4} b^{-1/4}}{X_+(W_c)^4}\sqrt{W-W_c}
\end{equation}
as $W\to W_c$. Since $\int_0^\e e^{-N x^{3/2}} x^{-1/2} dx=\Or(N^{-1/3})$, the contribution of the integral over $W\in \widetilde \Pi_-^\e(b)$ of (\ref{eq5.26b}) is of order $N^{-1} N^{-1/3}$. We shall see that the leading terms as $N\to\infty$ is of order $N^{-1}$. Thus this contribution is at least $N^{-1/3}$ times smaller and it becomes irrelevant as $N\to\infty$.

\medskip
\textbf{(b) Contributions for $W\in \Pi_+(b)$.} As discussed in Proposition~\ref{Prop5.12}, the paths for $X$ and $U$ have to be chosen to pass by both critical points, see Figure~\ref{FigIntegrationCase1} for an illustration.
\begin{figure}
\centering
\psfrag{W}[c]{$W$}
\psfrag{-W}[c]{$-W$}
\psfrag{0}[c]{$0$}
\psfrag{Xp}[c]{$X_+$}
\psfrag{Xm}[c]{$X_-$}
\psfrag{gx}[c]{$\gamma_X$}
\psfrag{Up}[c]{$U_+$}
\psfrag{Um}[c]{$U_-$}
\psfrag{gu}[c]{$\gamma_U$}
\includegraphics[height=5cm]{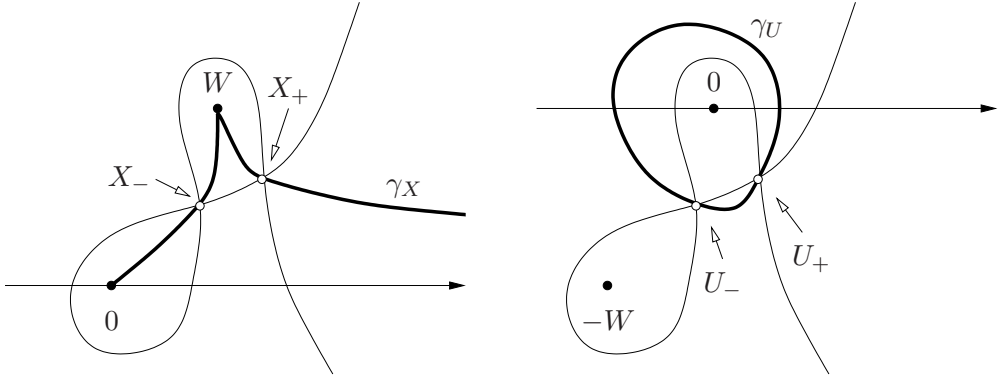}
\caption{Illustration of the integration paths $\gamma_X$ and $\gamma_U$ for the case $\varphi\in [0,\varphi_c)$.}
\label{FigIntegrationCase1}
\end{figure}

We split the integrand over $X$ into two pieces: the first going from $0$ to $W$ and the second from $W$ to $\infty$.
We thus have
\begin{equation}\label{eq5.27}
\begin{aligned}
\int_0^\infty dX \frac{e^{N F(b,W,X)}}{X-W}\frac{1}{2\pi\I} \oint_{\Gamma_0} dU \frac{e^{N G(b,W,U)}}{W+U}
&=\int_0^W dX \frac{e^{N F(b,W,X)}}{X-W}\frac{1}{2\pi\I} \oint_{\Gamma_0} dU \frac{e^{N G(b,W,U)}}{W+U}\\
&+\int_W^\infty dX \frac{e^{N F(b,W,X)}}{X-W}\frac{1}{2\pi\I} \oint_{\Gamma_0} dU \frac{e^{N G(b,W,U)}}{W+U}.
\end{aligned}
\end{equation}
At this point we can deform the integration path for $W$, namely $\Pi_+(b)$, in different ways depending on whether we want to estimate the first or the second term in (\ref{eq5.27}).

\smallskip
\textbf{(b.1) Contribution of the first term of (\ref{eq5.27})}. The integration path for $X$ passes by $X_-$, while the integration path for $U$ passes by both $U_+$ and $U_-$.
The leading term of the integrand comes from the integrals around the critical points. At $(X_-,U_-)$ the term in the exponent is $N F(b,W,X_-)+ N G(b,W,U_-)=0$, while the term in the exponent at $(X_-,U_+)$ is $N F(b,W,X_-)+N G(b,W,U_+)=-N H(W)$. Thus we deform the path $\Pi_+(b)$ into $\widetilde \Pi_+(b)$ such that $H(W)>0$ along $\widetilde \Pi_+(b)$. To do this is enough to deform $\Pi_+(b)$ slightly in the interior of $\Pi(b)$. Then, the contribution of the integral coming from the neighborhoods of $(U_+,X_-)$ can be estimated exactly as we did for the previous case (i.e., for \textbf{(a) Contributions for $W\in \Pi_-(b)$}) and in the $N\to\infty$ limit will become irrelevant.

It remains to determine the contribution coming from the neighborhood of $(X_-,U_-)$ (as outside a $\delta$-neighborhood the contribution is exponentially small in $N$). As $X_-$ and $U_-$ are (simple) critical points of $F(b,W,X)$ and $G(b,W,U)$ respectively, they are determined by the Gaussian integrals up to smaller order in $N$. We have
\begin{equation}
F(b,W,X)\simeq F(b,W,X_-)-\frac12 \sigma_- (X-X_-)^2,\quad G(b,W,U)\simeq G(b,W,U_-)+\frac12 \sigma_- (U-U_-)^2,
\end{equation}
where $\sigma_-$ is defined in (\ref{eq5.26}). Thus
\begin{equation}\label{eq5.29second}
\int_{|X-X_-|\leq \delta} dX \frac{e^{N F(b,W,X)}}{X-W}\frac{1}{2\pi\I} \int_{|U-U_-|\leq \delta} dU \frac{e^{N G(b,W,U)}}{W+U}=
\frac{1}{2\pi\I}\frac{\sqrt{2\pi}}{(X_- - W) \sqrt{\sigma_- N}}\frac{-\I\sqrt{2\pi}}{(U_- + W) \sqrt{\sigma_- N}}
\end{equation}
up to smaller terms in $N$ (of order $N^{-3/2}$ instead of $N^{-1}$).

\smallskip
\textbf{(b.2) Contribution of the second term of (\ref{eq5.27})}. This case is analogous to the previous one, except that now the integration over $X$ passes by $X_+$. This time the term in the exponent at $(X_+,U_-)$ is $N F(b,W,X_+)+N G(b,W,U_-)=N H(W)$ and therefore we deform the contour for $W$ to stay slightly outside $\Pi_+(b)$ (without crossing the dashed lines in Figure~\ref{FigSlowManyfold}). The leading contribution is then
\begin{equation}\label{eq5.30}
\int_{|X-X_+|\leq \delta} dX \frac{e^{N F(b,W,X)}}{X-W}\frac{1}{2\pi\I} \int_{|U-U_+|\leq \delta} dU \frac{e^{N G(b,W,U)}}{W+U}
= \frac{1}{2\pi\I}\frac{\sqrt{2\pi}}{(X_+ - W) \sqrt{\sigma_+ N}}\frac{\I\sqrt{2\pi}}{(U_+ + W) \sqrt{\sigma_+ N}}
\end{equation}
up to smaller terms in $N$.

Summing up, we have obtained that, up to smaller order terms in $N$, the leading contribution of (\ref{eq5.27}) is given by (\ref{eq5.29second}) plus (\ref{eq5.30}). An explicit computation gives
\begin{equation}\label{eq5.29b}
{\rm r.h.s.}~(\ref{eq5.29second})+{\rm r.h.s.}~(\ref{eq5.30}) = \frac{4}{N \sqrt{(W-W_c)(W-\overline W_c)}}\equiv \frac{4}{N \sqrt{(W-\Omega(1,b))(W-\overline\Omega(1,b))}}.
\end{equation}

\smallskip
To resume, the integral over $W$ becomes in the $N\to\infty$ limit the integral over $W$ from $\overline\Omega(1,b)$ to $\Omega(1,b)$ (passing to the right of the origin).

\medskip
The analysis of the integral over $Z$ is almost verbatim as the one for $W$ and thus we are not going to repeat it. The result is that the integral over $Z$ becomes the integral over $Z$ from $\overline\Omega(1,a)$ to $\Omega(1,a)$ passing to the right of the origin and to the right of $W$, namely we get
\begin{equation}\label{eq5.65}
\begin{aligned}
&\lim_{N\to\infty} N {\rm Cov}\Big(\zeta^{(N)}_{N-aN+1}(1),\zeta^{(M)}_{M-bM+1}(1)\Big)\\
&=\frac{16}{(2\pi\I)^2}\int_{\overline \Omega(c,b)}^{\Omega(c,b)} dW \int_{\overline \Omega(1,a)}^{\Omega(1,a)}\frac{1}{Z-W}\frac{1}{\sqrt{(W-\Omega(c,b))(W-\overline\Omega(c,b))}\sqrt{(Z-\Omega(1,a))(Z-\overline\Omega(1,a))}},
\end{aligned}
\end{equation}
which is what we had to prove for $d=1$, see (\ref{eqCov5.31}).

\textbf{Second case: $\Pi_-(a)\cap c\,\Pi(b) \neq \varnothing$.} This situation is illustrated in Figure~\ref{FigResidue}.
\begin{figure}
\begin{center}
\psfrag{w1}[l]{\small $\Omega(c,b)$}
\psfrag{w2}[r]{\small $\Omega(1,a)$}
\psfrag{z}[c]{\small $\Theta$}
\includegraphics[height=5cm]{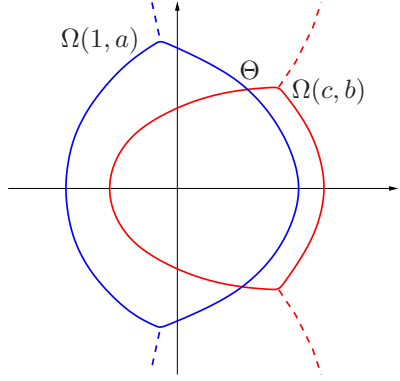}
\caption{An illustration of the case where $\Pi_-(a)\cap c\,\Pi(b) \neq \varnothing$. The closed paths are $\Pi(a)$ (blue) and $c\Pi(b)$ (red). Here the parameters are $a=0.6$, $b=0.2$ and $c=0.8$. The intersection in the upper half plane is denoted by $\Theta$.}
\label{FigResidue}
\end{center}
\end{figure}

The same argument as for the first case can be applied for most of the contributions. However, it remains to deal with the residue obtained for $Z=W$ for the portion of $W$ from $\overline \Theta\to \overline \Omega(c,b) \to \Omega(c,b)\to \Theta$. What we are going to show below is that the mixed terms (i.e., the ones involving $H(W)$ or $\tilde H(W)$) all give vanishing contributions in the $N\to\infty$ limit.

Assuming that this is shown and using the fact that for the portion of the integration from $\overline\Theta\to\overline\Omega(c,b)$ and as well from $\Omega(c,b)\to\Theta$ only mixed terms are present, the final result will be (\ref{eq5.65}) with the path $Z$ passing to the left of $W$ plus the residue at $Z=W$. Deforming the contours back to have $Z$ passing to the right of $W$ leads to the claimed result.

Thus now let us verify that the mixed terms in the residue at $Z=W$ for the integration from $\overline \Theta$ to $\Theta$ all asymptotically vanish. We have the following four possible mixed terms (we write only the exponential part, of course for each term with an $H$ there are also the denominators as in (\ref{eq5.26b})):
\begin{equation*}
e^{N (H_{b,c}(W)+H_{a,1}(W))},\quad e^{N (H_{b,c}(W)-H_{a,1}(W))},\quad e^{N (-H_{b,c}(W)+H_{a,1}(W))},\quad e^{N (-H_{b,c}(W)-H_{a,1}(W))},
\end{equation*}
where $H_{b,c}(W):=c H_{b,1}(W/c)$ with $H_{b,1}(W)=H(W)$ of Lemma~\ref{lem5.3}.

(1) The contribution containing $e^{N (H_{b,c}(W)+H_{a,1}(W))}$ vanishes because from $\overline \Theta$ to $\Theta$ we can choose a path satisfying $\Re (H_{b,c}(W)+H_{a,1}(W))<0$. Indeed, we could just use $W\in c\Pi_+(b)$ as on this path $\Re H_{b,c}(W)=0$ and $\Re H_{a,1}(W)<0$, or any other path in the shaded region of Figure~\ref{FigResiduePMandPP}(left) passing to the right of the origin. Then, the integration over $W$ gives a term of order $N^{-1}$ smaller than the leading term (in total $\Or(N^{-2}))$.

(2) The contribution with $e^{N (-H_{b,c}(W)-H_{a,1}(W))}$ vanishes as well. Indeed, by Corollary~\ref{CorIntersections}, we can choose a path from $\overline\Theta$ to $\Theta$ in the intersection of the interiors of $\Pi(a)$ and $c\Pi(b)$, where both terms are negative.
\begin{figure}
\centering
\begin{subfigure}
\centering
\includegraphics[height=5cm]{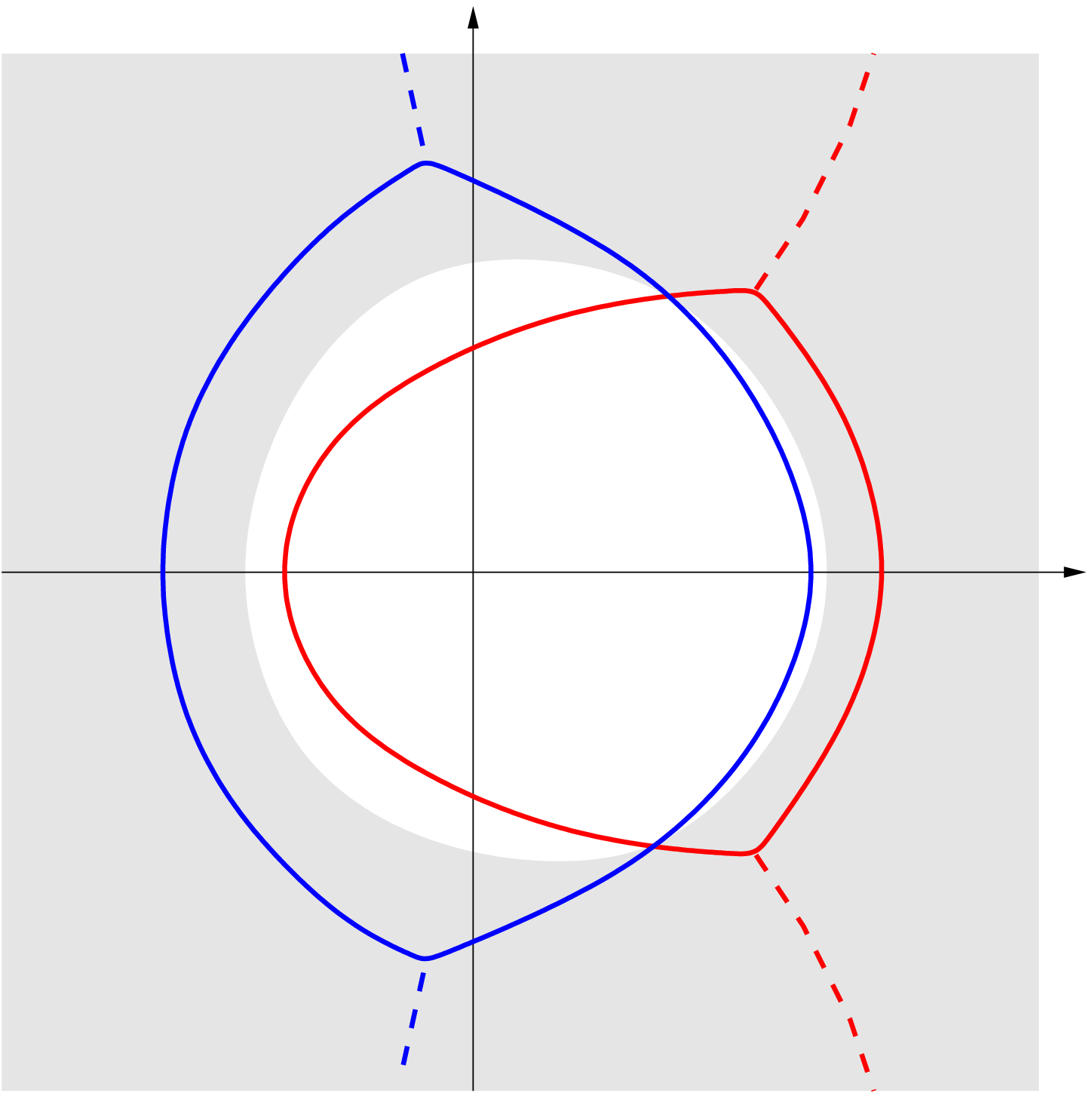}
\end{subfigure}
\qquad
\begin{subfigure}
\centering
\includegraphics[height=5cm]{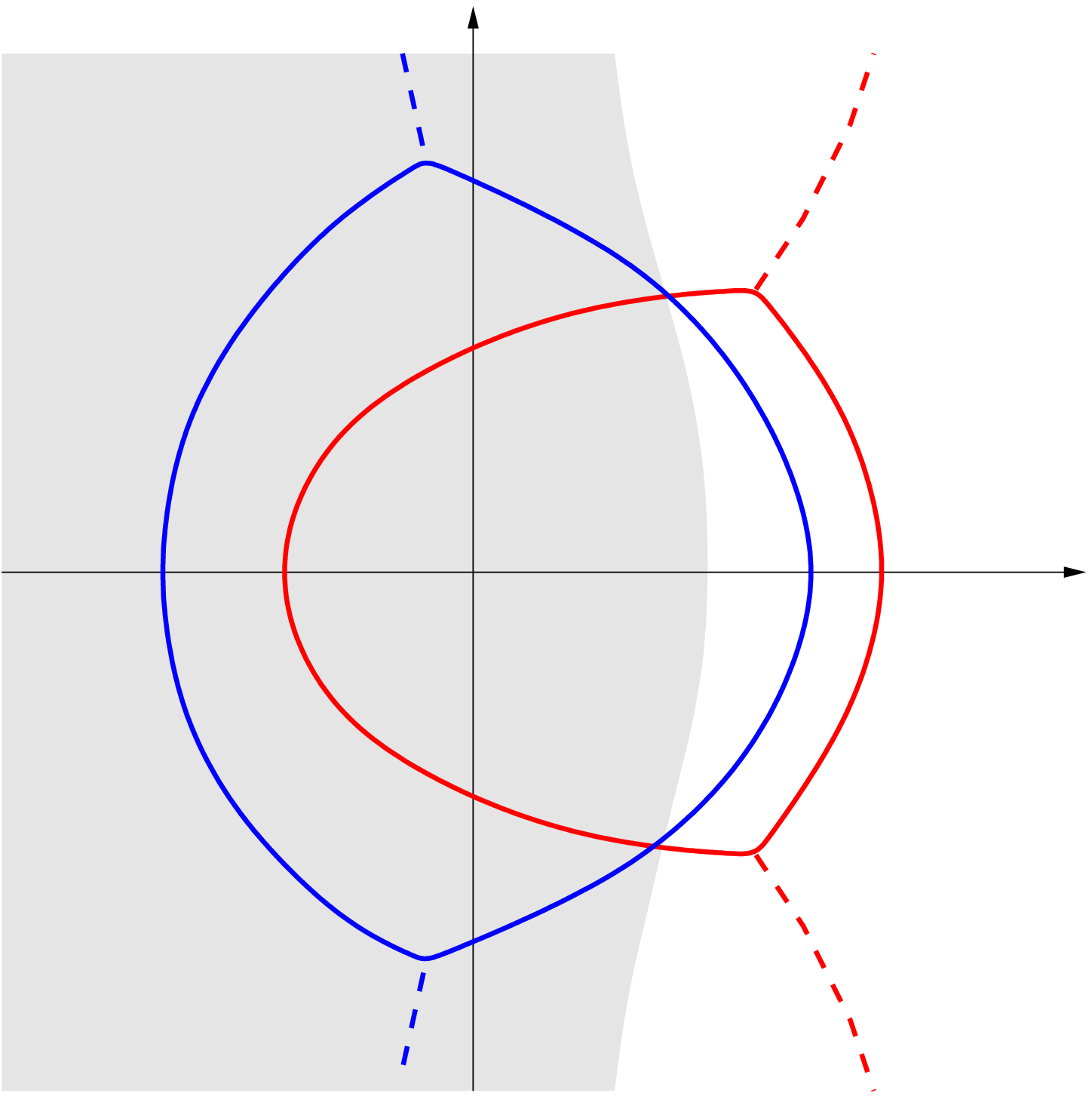}
\end{subfigure}
\caption{Illustrations with the same parameters as in Figure~\ref{FigResidue}. (Left) The shaded region is where $\Re(H_{b,c}(W)+H_{a,1}(W))<0$. (Right) The shaded region is where $\Re(H_{b,c}(W)-H_{a,1}(W))<0$.}
\label{FigResiduePMandPP}
\end{figure}

(3) The contribution including $e^{N (H_{b,c}(W)-H_{a,1}(W))}$ also vanishes. First notice that $W=0$ is not a pole of $e^{N (H_{b,c}(W)-H_{a,1}(W))}$. Indeed, by looking at the asymptotic behavior at $W=0$, we have that $H_{b,c}(W)-H_{a,1}(W)\sim (1-c)\ln(W)+O(1)$, thus $e^{N (H_{b,c}(W)-H_{a,1}(W))}\sim W^{(1-c)N}$.
We have $\Re(H_{b,c}(W)-H_{a,1}(W))=0$ for $W\in\{\Theta,\overline\Theta\}$, $\Re(H_{b,c}(W)-H_{a,1}(W))<0$ on $W\in c\Pi(b)\big|_{(\Theta,\overline\Theta)}$ since it is in the interior of the region with $\Re H_{a,1}(W)\geq 0$ (by Corollary~\ref{CorIntersections}), and $\Re (-H_{b,c}(W)+H_{a,1}(W))<0$ on $W\in\Pi(a)\big|_{\overline\Theta,\Theta}$. Taking this as integration path we get that the contribution of this term is also vanishing. See Figure~\ref{FigResiduePMandPP} for an illustration.

(4) The contributions with $e^{N (-H_{b,c}(W)+H_{a,1}(W))}$. It is similar to case (3), except that now $\Re (-H_{b,c}(W)+H_{a,1}(W))<0$ on $W\in\Pi(a)\big|_{(\overline\Theta,\Theta)}$.

Summing up, as soon as $|\Omega(1,a)-\Omega(c,b)|$ are bounded away from zero (which implies that the same holds for $|\Omega(1,a)-\Theta|$ as well), all the mixed terms (i.e., the ones containing $H(W)$) are exponentially small in $N$ and thus they become irrelevant in the $N\to\infty$ limit. This completes the proof of the result.
\end{proof}

For what we are going to do later, it is important to understand when, in the proof of Theorem~\ref{thm5.1}, the hypothesis that $|\Omega(1,a)-\Omega(c,b)|$ bounded away from zero is used. The hypothesis was not used when determining the leading contribution coming from integrating out the $X,U$ (resp.\ $Y,V$) variables as those depend only on $\Omega(c,b)$ (resp.\ $\Omega(1,a)$). We obtained
\begin{equation}\label{eq5.31}
\int_0^\infty dX \frac{e^{N F(c,b,W,X)}}{X-W}\frac{1}{2\pi\I} \oint_{\Gamma_0} dU \frac{e^{N G(c,b,W,U)}}{W+U} = \frac{4 (1+\Or(N^{-1/2})) \Id_{W\in c\Pi_+(b)}}{N \sqrt{(W-\Omega(c,b))(W-\overline\Omega(c,b))}}+\frac{\Id_{W\in c\Pi_-(b)}}{N^{1+1/3}},
\end{equation}
and similarly for the integration over $Z$.

As a consequence, the integrals over $W,Z$ in ${\rm Cov}(\zeta^{(dN)}_{(1-a)dN}(1),\zeta^{(cN)}_{(1-b)cN}(1))$ in which we do not consider the residue terms give
\begin{equation}\label{eq5.32}
\begin{aligned}
\frac{1}{N}\frac{1}{(2\pi\I)^2}\oint_{c\Pi(b)} dW \oint_{\Pi(a)} dZ \frac{1}{Z-W}&\bigg(\frac{4 (1+\Or(N^{-1/2})) \Id_{W\in c\Pi_+(b)}}{\sqrt{(W-\Omega(c,b))(W-\overline\Omega(c,b))}}+\frac{\Id_{W\in c\Pi_-(b)}}{N^{1/3}}\bigg)\\
\times&\bigg(\frac{4 (1+\Or(N^{-1/2})) \Id_{Z\in \Pi_+(a)}}{\sqrt{(Z-\Omega(1,a))(Z-\overline\Omega(1,a))}}+\frac{\Id_{Z\in \Pi_-(a)}}{N^{1/3}}\bigg).
\end{aligned}
\end{equation}

For the residue terms, as mentioned above, only between $\overline\Omega(c,b)$ and $\Omega(c,b)$ there are terms which are not mixed. These give a contribution to the covariance equal to
\begin{equation}\label{eq5.33}
\frac{1}{N}\frac{1}{2\pi\I}\int_{\bar \Omega(c,b)}^{\Omega(c,b)} dW
\frac{4 (1+\Or(N^{-1/2}))}{\sqrt{(W-\Omega(c,b))(W-\overline\Omega(c,b))}}\frac{4 (1+\Or(N^{-1/2}))}{\sqrt{(W-\Omega(1,a))(W-\overline\Omega(1,a))}}.
\end{equation}

Finally, let us consider the mixed residue terms. When the two critical points are close to each other (of distance going to $0$ as $N\to\infty$), the statements mentioned above are still true, but the exponential decay in $N$ when integrating over $W$ is less strong. Indeed, by Lemma~\ref{lem5.7} and the fact used already just above (\ref{eq5.24}), if we move in a steep descent direction of $H(W)$ starting from $W=\Theta$ (which is of distance $\Or(N^{-1/2})$ from the double critical point), then $0\geq\Re H_{b,c}(W)\simeq (W-\Theta)^{3/2}$ for small $W-\Theta$. Moreover, since $\Theta$ is close to the double critical point, we have to be a little bit careful with the term in the denominator of (\ref{eq5.26b}), but by (\ref{eq5.24}) we know that $\sigma_+(W)\simeq \sqrt{W-\Omega(b,c)}$. Since we have the contribution from $W$ and the one from the residue at $Z=W$, overall the denominator is just of order $N^{-1}/|\Omega(1,a)-\Omega(c,b)|$.
Further, since for $\gamma>0$, $\int_0^\infty dx e^{-\gamma N x^{3/2}}=\Or(N^{-2/3})$, the full contribution for each of the possible four cases for the mixed terms will be $\Or(N^{-1-2/3}/|\Omega(1,a)-\Omega(c,b)|)$. As a consequence, if we consider $|\Omega(1,a)-\Omega(c,b)|\to 0$ as $N\to\infty$, we get
\begin{equation}\label{eq5.34}
{\rm Cov}\Big(\zeta^{(dN)}_{(1-a)dN}(1),\zeta^{(cN)}_{(1-b)cN}(1)\Big) = \frac{1}{N}\Or\Big(\frac{1}{N^{2/3}|\Omega(1,a)-\Omega(c,b)|}\Big) + (\ref{eq5.32})+ (\ref{eq5.33}).
\end{equation}
In particular, if $|\Omega(1,a)-\Omega(c,b)|=\Or(N^{-1/2})$, the error terms from the mixed terms in the residue are $\Or(N^{-1/6})$. Moreover, as we will see below, in this case the leading contribution to the integrals comes actually from the two neighborhoods of the double critical points only, i.e., from the regions where $Z-W$ is small.

\subsection{Logarithmic correlations at short distances}
In Theorem~\ref{thm5.1} we have obtained the limiting covariance. Here we consider the limiting covariance and investigate its behavior at short distances, i.e., in the limit when $|\Omega(d,a)-\Omega(c,b)|=\delta\to 0$. We obtain a logarithmic divergence as $\delta\to 0$.
\begin{proposition}\label{propLogCov}
Let $a\in(0,1)$ and $d>0$ be fixed. Then
\begin{equation}\label{eqCov}
\begin{aligned}
&\frac{16}{(2\pi\I)^2}\int_{\overline \Omega(c,b)}^{\Omega(c,b)} dW \int_{\overline \Omega(d,a)}^{\Omega(d,a)} dZ \frac{1}{Z-W}\frac{1}{\sqrt{(W-\Omega(c,b))(W-\overline\Omega(c,b))}\sqrt{(Z-\Omega(d,a))(Z-\overline\Omega(d,a))}}\\
&=\frac{-4}{\pi}\frac{\ln(|\Omega(d,a)-\Omega(c,b)|)}{\sqrt{\Im\Omega(d,a)}\sqrt{\Im\Omega(c,b)}}+\Or(1)
\end{aligned}
\end{equation}
as $|\Omega(d,a)-\Omega(c,b)|\to 0$.
\end{proposition}
\begin{proof}
Let us set $\delta=|\Omega(d,a)-\Omega(c,b)|$. Let us denote $\Omega_1= \Omega(d,a)$, $\Omega_2=\Omega(c,b)$ and the unit vector $\vec e=|\Omega_1-\Omega_2|^{-1} (\overline \Omega_1-\overline \Omega_2)$. Because of the assumptions, $c_0:=\Im \Omega_1=2\sqrt{a(1-a)}>0$ and thus for $\delta$ small enough, $\Im\Omega_2>c_0/2$ as well. Then in (\ref{eqCov}) we choose the paths for $W,Z$ as follows: \\
(1) $Z$ has a first straight piece of length $c_0/4$ in the direction $\vec e$,\\
(2) $W$ has a first straight piece of length $c_0/4$ in the direction $-\vec e$,\\
(3) for the rest, on $\{z\in\C, \Im z\leq 0\}$, the contours stay at a distance at least $c_0/4$ of each other and from $\overline \Omega_1,\overline \Omega_2$,\\
(4) the contours are chosen symmetric with respect to complex conjugation.

Pick a small positive $\e<c_0/2$. We divide the integral in (\ref{eqCov}) into $I_\e:=\{|Z-\overline\Omega_1|+|W-\overline\Omega_2|\leq \e\}$, $\tilde I_\e=\{|Z-\Omega_1|+|W-\Omega_2|\leq \e\}$ and $J_\e=(I_\e\cup\tilde I_\e)^c$.
The integral over $J_\e$ is bounded by $\Or(1/\e)$, since $x\mapsto 1/\sqrt{x}$ is integrable around $0$. For the integral over $I_\e$, we consider the parametrization: $Z=\overline \Omega_1+y\vec e$ and $W=\overline \Omega_2-x\vec e$, where $x+y\leq \e$. Plugging in these we get, for the integral over $I_\e$,
\begin{equation}\label{eq5.35b}
\frac{16}{(2\pi\I)^2} \int dZ \int dW (\cdots) = \frac{16}{(2\pi)^2\I\sqrt{(\overline \Omega_1-\Omega_1)(\overline \Omega_2-\Omega_2)}}\int dx \int dy \frac{1}{\sqrt{x}\sqrt{y}(\delta+x+y)}+\Or(1),
\end{equation}
where the integral is over $\{x,y\geq 0 | x+y\leq \e\}$. The $\Or(1)$ error term comes from replacing $1/\sqrt{Z-\Omega_1}=(1 + \Or(Z-\overline\Omega_1))/\sqrt{\overline \Omega_1 - \Omega_1}$ and $1/\sqrt{W-\Omega_1}=(1 + \Or(W-\overline\Omega_2))/\sqrt{\overline \Omega_2 - \Omega_2}$. Indeed, doing these replacements, due to the fact that $|W-\overline\Omega_2|<|W-Z|$ and similarly $|Z-\overline\Omega_1|<|Z-W|$ on $I_\e$, the integrals with these error terms are bounded. Doing the change of variables $u=x-y$, $v=x+y$ we get
\begin{equation*}
\int dx \int dy \frac{1}{\sqrt{x}\sqrt{y}(\delta+x+y)} = \int_0^\e dv \frac{1}{\delta+v}\int_{-v}^v du \frac{1}{\sqrt{v^2-u^2}} = \pi \ln(\delta+\e) - \pi \ln(\delta).
\end{equation*}
Therefore the integral over $I_\e$ leads to
\begin{equation*}
\frac{2}{\pi\sqrt{\Im\Omega_1 \Im\Omega_2}}\ln(1/\delta) + \Or(1),
\end{equation*}
since we can take $\e$ small but fixed, the statement to be proven holds. The contribution of the integral over $\tilde I_\e$ is the same. This finishes the proof.
\end{proof}

Remark that the logarithmic correlations at small distances follow also from the expression in Proposition~\ref{PropCovElliptic} together with the asymptotic behavior (see Eq.~13.8(10) in~\cite{Bat53})
\begin{equation}\label{eq5.35}
\mathbb{K}(\kappa)=\ln(4/\sqrt{1-\kappa^2})+\Or((1-\kappa^2)\ln(1-\kappa^2)),\quad\textrm{as }\kappa\to 1.
\end{equation}

\subsection{Two-time covariance in the bulk scaling limit (and slow decorrelation)}
We have seen that the covariance of $N^{1/2}\zeta^{(cN)}_{(1-b)cN}(T=1)$ has a non-trivial limit as $N\to\infty$. Now we argue that if we consider the covariance of the process at different times, we effectively see the covariance observed at fixed time. Recall that the covariance at two times is obtained by applying the propagator (\ref{eqPropagator}) to the fixed time covariance (see (\ref{eq5.29})), namely
\begin{equation*}
Y^{T_0}(T) {\rm Cov}(T_0).
\end{equation*}
Thus, to study the space-time covariance we have to understand which regions at time $T_0$ are correlated with a given point at time $T$. This is done by studying the propagator (\ref{eqPropagator}). In Lemma~\ref{lemmaPropAsympt} we show that the correlated region is around the characteristic and it is of order $\sqrt{N}$ (linear scale) only, independently of $T>T_0$. By Theorem~\ref{thm5.1} we know that the limiting covariance changes over $\Or(N)$ only. Thus, as stated in Proposition~\ref{PropCovTwoTimes}, the two-time covariance is asymptotically the same as the fixed time covariance between points scaled linearly in $T$. The reason is that along the ray of fixed direction (characteristic lines for our system) the correlation persists forever, which is at first unexpected.

\begin{lemma}\label{lemmaPropAsympt}
Consider the propagator (\ref{eqPropagator}), namely,
\begin{equation*}
\left[Y^{T_0}(T)\right]_{(k,n),(k',n')}=\left(\frac{T_0}{T}\right)^{n-1}\left(\frac{T-T_0}{T_0}\right)^{n-n'} \binom{k-1}{k'-1}\binom{n-k}{n'-k'}.
\end{equation*}
For $T>1=T_0$, we set
\begin{equation}\label{eq5.82}
\begin{aligned}
k=(1-a)d T N,&\quad k'=(1-a)d N + \sigma_1 \sqrt{(1-a) d N},\\
n=d T N,&\quad n'=d N+\sigma_1\sqrt{(1-a) d N}+\sigma_2 \sqrt{a d N}.
\end{aligned}
\end{equation}
Then, as $N\to\infty$,
\begin{equation}\label{eqPropagatorAsympt}
\left[Y^{T_0}(T)\right]_{(k,n),(k',n')} =\frac{1}{2\pi (T-1)/T}\frac{1}{\sqrt{a(1-a)}dN} \exp\left(-\frac{\sigma_1^2+\sigma_2^2}{2 (T-1)/T}\right)(1+o(1)).
\end{equation}
The result still holds if $T$ goes to infinity with $N$ at any speed, in which case $(T-1)/T$ is replaced by $1$.
\end{lemma}
\begin{proof}
Using Stirling formula we have
\begin{equation*}
\binom{k}{k'}= \sqrt{\frac{k}{2\pi k'(k-k')}} e^{B(k,k')} (1+o(1)),\quad B(k,k')=k\ln(k)-k'\ln(k')-(k-k')\ln(k-k').
\end{equation*}
Also, $\binom{k-1}{k'-1}\frac{k'}{k}=\binom{k}{k'}$. Thus
\begin{equation*}
\left[Y^{T_0}(T)\right]_{(k,n),(k',n')} = C(k,k',n,n',T) e^{-n\ln(T)}e^{(n-n')\ln(T-1)} e^{B(k,k')} e^{B(n-k,n'-k')} (1+o(1)),
\end{equation*}
where
\begin{equation*}
C(k,k',n,n',T)=T \sqrt{\frac{k'}{2\pi k(k-k')}} \sqrt{\frac{n-k}{2\pi (n'-k')(n-n'-k+k')}}.
\end{equation*}
A computation gives
\begin{equation*}
-n\ln(T)+(n-n')\ln(T-1)+B(k,k')+B(n-k,n'-k')= -\frac{T(\sigma_1^2+\sigma_2^2)}{2 (T-1)}+\Or(N^{-1/2}),
\end{equation*}
and
\begin{equation*}
C(k,k',n,n',T) = \frac{T}{2\pi \sqrt{a(1-a)}d (T-1)N}+\Or(N^{-3/2}).
\end{equation*}
This gives the result for $T$ independent of $N$. However, inspecting the computations one realizes that if $T\to\infty$ as $N\to\infty$, then the same result holds, where of course one replaces $T/(T-1)$ by its limit, which is $1$.
\end{proof}

As a consequence of this result we have
\begin{proposition}\label{PropCovTwoTimes}
Take any $a,b\in (0,1)$, $d>0$ and $c\in (0,d]$. Then for any $T>1$,
\begin{equation}\label{eqCov5.93}
\begin{aligned}
&\lim_{N\to\infty} N {\rm Cov}\Big(\zeta^{(dN T)}_{(1-a)dN T}(T),\zeta^{(cN)}_{(1-b)cN}(T=1)\Big)\\
&=\frac{16}{(2\pi\I)^2}\int_{\overline \Omega(c,b)}^{\Omega(c,b)} dW \int_{\overline \Omega(d,a)}^{\Omega(d,a)} dZ \frac{1}{Z-W}\frac{1}{\sqrt{(W-\Omega(c,b))(W-\overline\Omega(c,b))}\sqrt{(Z-\Omega(d,a))(Z-\overline\Omega(d,a))}},
\end{aligned}
\end{equation}
where the path $Z$ stays to the right of $W$. Here $T$ can even go to infinity as $N\to\infty$.
\end{proposition}
\begin{proof}
First of all, notice that the prefactor $\frac{1}{\sqrt{a(1-a)}dN}$ in (\ref{eqPropagatorAsympt}) is the volume element of the change of variables $(k',n')\to (\sigma_1,\sigma_2)$, necessary to turn the sum in
\begin{equation*}
\sum_{k',n'}\left[Y^{T_0}(T)\right]_{(k,n),(k',n')} {\rm Cov}(T_0)_{(k',n'),(k'',n'')}
\end{equation*}
into an integral over $(\sigma_1,\sigma_2)$. We can write (using $(k',n')$ as in (\ref{eq5.82}))
\begin{equation}\label{eq5.91}
\begin{aligned}
&N {\rm Cov}\left(\zeta^{d N T}_{(1-a)d N T}(T),\zeta^{c N}_{(1-b)c N}(1)\right)\\
& = \int d\sigma_1 d\sigma_2 \sqrt{a(1-a)}dN \left[Y^{1}(T)\right]_{((1-a)d N T,d N T),(k',n')}
N {\rm Cov}\left(\zeta^{n'}_{n-k'}(1),\zeta^{c N}_{(1-b)c N}(1)\right),
\end{aligned}
\end{equation}
where the integrand is thought to be piece-wise constant so as to coincide with the sum (i.e., we just have rescaled $(k',n')$ but not taken any limit).

For any $R>0$, consider the integral of (\ref{eq5.91}) restricted to $\{|\sigma_1|,|\sigma_2|\leq R\}$. From Lemma~\ref{lemmaPropAsympt}, it is a convolution of a Gaussian kernel (with variance $\sqrt{(T-1)/T}$) and of $N {\rm Cov}\left(\zeta^{n'}_{n-k'}(1),\zeta^{c N}_{(1-b)c N}(1)\right)$. By Theorem~\ref{thm5.1} the latter is independent of $\sigma_1$ and $\sigma_2$ in the $N\to\infty$ limit. Thus the contribution of (\ref{eq5.91}) restricted to $\{|\sigma_1|,|\sigma_2|\leq R\}$ is given by (\ref{eqCov5.93}) times $M_R=\left({\rm Erf}(R/\sqrt{2-2/T})\right)^2$ as $R\to\infty$ (which is the mass of the Gaussian in the integration domain). Since $M_R\to 1$ as $R\to\infty$, this will be the full contribution to the two-times covariance in the large-$N$ limit.

The contribution of (\ref{eq5.91}) outside $\{|\sigma_1|,|\sigma_2|\leq R\}$ is of order \mbox{$\Or(1)(1-M_R)\to 0$} as $R\to\infty$. To see this, first observe that the limiting covariance (multiplied by $N$) is uniformly bounded as soon as (a) the $\Omega$'s for $(k',n')$ and for $(bcN,cN)$ away from each other or (b) away from $\Im\Omega=0$. Case (a) is violated only for some $\sigma_1,\sigma_2$ of order $\sqrt{N}$, the region where anyway the propagator vanishes as exponentially in $\sigma_1,\sigma_2$ (on top of it, from Proposition~\ref{propLogCov}, when two points $(k,n)$ and $(k',n')$ are at distance $\delta N$ with small $\delta$, then their covariance is only diverging like $\ln(\delta)$, which can still be integrated.). Case (b) is also not a problem, since this breaks down as well for $\sigma_1,\sigma_2$ of order $\sqrt{N}$ and when $\Im\Omega\to 0$, the covariance might diverge but only polynomially in $N$, which is dominated by the Gaussian decay of the propagator. Thus taking $N\to\infty$ and then $R\to\infty$ one establishes the result.
\end{proof}

\subsection{Correlations close to the characteristic lines}
In Proposition~\ref{PropCovTwoTimes} we showed that there is slow-decorrelation and the time correlations equal at first order the correlations at fixed time. This holds in the case when the projections along the characteristic at fixed time of the space-time points under focus are at distance of order $\Or(N)$ of each other. Now we want to consider the correlations close to the characteristic lines. (\ref{eqPropagatorAsympt}) suggests that non-trivial correlation are present when we consider space-time points at distance $\Or(\sqrt{N})$ from a given characteristic line. This is what we show in the next proposition.
\begin{proposition}\label{PropCorrCaracteristics}
Let $d>0$, $a\in (0,1)$, and $T>1$ be fixed. For any given $\xi_1,\xi_2\in\R$, consider the scaling
\begin{equation*}
k=(1-b) c N=(1-a) d N + \xi_1 \sqrt{(1-a) d N},\quad n=c N=d N+\xi_1\sqrt{(1-a) d N}+\xi_2 \sqrt{a d N}.
\end{equation*}
Then,
\begin{equation}
\begin{aligned}
&\lim_{N\to\infty} N {\rm Cov}\Big(\zeta^{(dN T)}_{(1-a)dN T}(T),\zeta^{(cN)}_{(1-b)cN}(T=1)\Big)- \frac{\ln(N/d)}{\pi d \sqrt{a(1-a)}}=\\
&+\frac{-1}{\pi d \sqrt{a(1-a)}}\frac{1}{2\pi (T-1)/T} \int_{\R^2} d\sigma_1 d\sigma_2 e^{-(\sigma_1^2+\sigma_2^2)/(2 (T-1)/T)} \ln[(\sigma_1-\xi_1)^2+(\sigma_2-\xi_2)^2].
\end{aligned}
\end{equation}
\end{proposition}

\begin{remark}
In particular, at equal time $T=1$, setting $c=1+\xi_1\sqrt{(1-a)/N}+\xi_2\sqrt{a/N}$, and $b=a+(\xi_2 \sqrt{1-a}-\xi_1\sqrt{a})\sqrt{a(1-a)/N}$, then we have
\begin{equation}
\lim_{N\to\infty} N {\rm Cov}\Big(\zeta^{(N)}_{(1-a)N}(1),\zeta^{(cN)}_{(1-b)cN}(1)\Big)- \frac{\ln(N)}{\pi\sqrt{a(1-a)}}=\frac{-1}{\pi \sqrt{a(1-a)}} \ln[\xi_1^2+\xi_2^2].
\end{equation}
\end{remark}

\begin{proof}[Proof of Proposition~\ref{PropCorrCaracteristics}]
The covariance to be computed is given explicitly by (\ref{eq5.91}) with
\begin{equation}
k'=(1-a')d'N=(1-a)d N + \sigma_1 \sqrt{(1-a) d N},\quad n'=d'N=d N+\sigma_1\sqrt{(1-a) d N}+\sigma_2 \sqrt{a d N},
\end{equation}
in which we can insert the asymptotics of the propagator (\ref{eqPropagatorAsympt}) (the control for large $\sigma_1,\sigma_2$ is as in the proof of Proposition~\ref{PropCovTwoTimes} and thus we do not repeat the details). The difference with respect to Proposition~\ref{PropCovTwoTimes} is that the distance between the double critical points $\Omega(d,a)$ and $\Omega(d',a')$ scales as $N^{-1/2}$ and therefore we need to be a bit more careful and use (\ref{eq5.34}) which holds in that regime. An explicit computation gives
\begin{equation}\label{eq5.60}
|\Omega(c,b)-\Omega(d',a')|=\sqrt{(\xi_1-\sigma_1)^2+(\xi_2-\sigma_2)^2}\sqrt{d/N}(1+\Or(N^{-1/2})).
\end{equation}
Let us verify that the error terms are all negligible in the $N\to\infty$ limit. The $\Or(N^{-1/2})$ error terms in (\ref{eq5.32}) are all neglibible: (a) when $W\in c\Pi_+(b)$ and $Z\in\Pi_+(a)$, then we can do the approximation of the integrals used to prove Proposition~\ref{propLogCov} and the $\Or(1)$ in (\ref{eqCov}) is multiplied by $\Or(N^{-1/2})$; (b) when either $W\in c\Pi_+(b)$ or $Z\in\Pi_+(a)$, then with the same strategy of the proof of Proposition~\ref{propLogCov} we get that the double integral is bounded times $\Or(N^{-1/3})$. The reason is that in this case, see (\ref{eq5.32}), one of the inverse square root term is not present, which means that in (\ref{eq5.35b}) either $1/\sqrt{y}$ or $1/\sqrt{x}$ is absent; (c) when $W\in c\Pi_-(b)$ and $Z\in\Pi_-(a)$, then the term is $\Or(N^{-2/3})$ since $1/(Z-W)$ is integrable in two-dimensions. The same holds for the error terms in (\ref{eq5.33}). Finally, the error term in (\ref{eq5.34}) is clearly integrable in two dimensions and it goes to $0$ as $N\to\infty$.

What remains to be done is to determine the asymptotics of (\ref{eq5.91}) in which we consider only the non-error terms in (\ref{eq5.34}), namely the l.h.s.\ of (\ref{eqCov}). However, Proposition~\ref{propLogCov} is not good enough, since we want to prove the limiting covariance up to $\Or(1)$. Instead, for the l.h.s.\ of (\ref{eqCov}) we can as well use the exact expression contained in Proposition~\ref{PropCovElliptic}. The asymptotics (\ref{eq5.35}) of the complete elliptic integral gives us $o(1)$ instead of $\Or(1)$ in (\ref{eqCov}). Thus, up to $o(1)$ terms, $N {\rm Cov}(\zeta^{(dN T)}_{(1-a)dN T}(T),\zeta^{(cN)}_{(1-b)cN}(T=1))$ is given by
\begin{equation}
\frac{-1}{2\pi (T-1)/T} \int_{\R^2} d\sigma_1 d\sigma_2 e^{-(\sigma_1^2+\sigma_2^2)/(2 (T-1)/T)} \frac{4}{\pi}\frac{\ln[|\Omega(c,b)-\Omega(d',a')|]}{\sqrt{\Im(\Omega(c,b))}\sqrt{\Im(\Omega(d',a'))}}.
\end{equation}
Using $\Im(\Omega(c,b))=2 d \sqrt{a(1-a)}+\Or(N^{-1/2})$, $\Im(\Omega(d',a'))=2d\sqrt{a(1-a)}+\Or(N^{-1/2})$, and (\ref{eq5.60}) we obtain the claimed result.
\end{proof}

\subsection{Identifying the additive stochastic heat equation}\label{secidentify}
One might ask what is the behaviour of the limiting double integral in Proposition~\ref{PropCorrCaracteristics}. We use $\tau=(T-1)/T$ and use polar coordinates, $\xi_1=R\sqrt{\tau} \cos(\phi)$ and $\xi_2=R\sqrt{\tau} \sin(\phi)$, the change of variables $\sigma_1=\xi_1+\lambda\sqrt{\tau}\cos(\theta+\phi)$ and $\sigma_2=\xi_2+\lambda\sqrt{\tau}\sin(\theta+\phi)$ leads to
\begin{equation}\label{eq5.49}
\frac{1}{2\pi} \int_{\R^2} d\sigma_1 d\sigma_2 e^{-(\sigma_1^2+\sigma_2^2)/2} \ln[(\sigma_1-\xi_1)^2+(\sigma_2-\xi_2)^2]=\ln(\tau)+C(R)
\end{equation}
with
\begin{equation}\label{CReq}
C(R)=\frac{1}{\pi} \int_{\R_+} d\lambda\, \lambda \ln(\lambda) e^{-(\lambda^2+R^2)/2} \int_{-\pi}^\pi d\theta e^{\lambda R \cos(\theta)}=2e^{-R^2/2} \int_{\R_+} d\lambda\, \lambda \ln(\lambda)e^{-\lambda^2/2} I_0(\lambda R).
\end{equation}
This is a function that goes from $C(0)=\ln(2)-\gamma_{\rm Euler}$ to $C(R)\simeq 2 \ln(R)$ as $R\to\infty$. It can be expressed in terms of incomplete Gamma function,
\begin{equation*}
\Gamma(s,x) = \int_x^{\infty} t^{s-1}e^{-t}dt,
\end{equation*}
as follows
\begin{equation*}
C(R)=\Gamma(0,R^2/2)+\ln(R^2).
\end{equation*}
This can be readily proved by subsisting the series expansion for $I_0(2z) = \sum_{k=0}^{\infty} z^{2k}/(k!)^2$ into the last integral in (\ref{CReq}), interchanging the integral in $\lambda$ with the sum in $k$ (easily justified) and then evaluating the resulting $\lambda$ integrals
$$
\int_0^{\infty} d\lambda \,\lambda^{2k+1} \log(\lambda) e^{-\lambda^2/2}e^{-\lambda^2/2} = 2^{k-1} k!\big(\log 2 + \psi(1+k)\big),
$$
where $\psi$ is the digamma function. The resulting sum can be identified with a series expansion~\cite[Equation (8.19.9)]{NIST:DLMF} for the incomplete Gamma function (equivalently the exponential integral $E_1$).

For the general statement below, we define a parameter $\tau$ depending on $T>S>0$, and we will work with the following function for $r\in(0,\infty)$:
$$
G_{\tau}(r) = -\Gamma\big(0,\tfrac{r^2}{2\tau}\big) - \ln(r^2),
$$
where $\xi=(\xi_1,\xi_2)$, $r= |\xi|$.

\begin{corollary}\label{cor530}
Fix $d>0$, $a\in (0,1)$, $T>S>0$. For $\eta=(\eta_1,\eta_2)$ let $\zeta(T,\eta;N) = N^{1/2} \zeta^{(n)}_k(T)$
$$
\textrm{where} \quad n=\left(dN+\left(\eta_1\sqrt{(1-a)d}+\eta_2\sqrt{ad}\right)\sqrt{N}\right)T, \qquad k= \left((1-a)dN+\eta_1\sqrt{(1-a)d}\sqrt{N}\right)T.
$$
Then for $\eta,\lambda,\mu,\nu\in \R^2$ (all different),
\begin{align*}
&\lim_{N\to\infty} {\rm Cov}\left(\zeta(T,\eta;N)-\zeta(T,\lambda;N),\, \zeta(S,\mu;N)-\zeta(S,\nu;N)\right)\\
&\qquad = \frac{S}{\pi d\sqrt{a(1-a)}} \Big(G_{\tau}(|\eta-\mu|)-G_{\tau}(|\eta-\nu|)-G_{\tau}(|\lambda-\mu|)+G_{\tau}(|\lambda-\nu|)\Big).
\end{align*}
\end{corollary}
\begin{proof}
Due to the scaling in Proposition~\ref{PropLargeLsde},
$$
{\rm Cov}\Big(\zeta\big(T,\eta;N\big),\zeta\big(S,\nu;N\big)\Big) =S\, {\rm Cov}\Big(\zeta\big(T/S,\eta;N\big),\zeta\big(1,\nu;N\big)\Big).
$$
By this and linearity, it suffices to show that
$$
\lim_{N\to\infty} {\rm Cov}\Big(\zeta\big(T,\eta;N\big),\zeta\big(1,\nu;N\big)\Big)- \frac{\ln(N/d)}{\pi d \sqrt{a(1-a)}} = \frac{1}{\pi d\sqrt{a(1-a)}} G_{\tau}\big( |\eta-\nu|\big),
$$
where $\tau$ is defined with $S=1$ (i.e. $\tau=(T-1)/T$).
This follows from the result of Proposition~\ref{PropCorrCaracteristics}. In that result, replace for the moment the notation $a,d$ by $\tilde{a},\tilde{d}$. Then, there are four free variables $\tilde{a},\tilde{d},\xi_1,\xi_2$. If we make the following substitutions
$$
\tilde{a} = a+ \frac{\eta_1 - a\eta_1 - \eta_2}{dT}N^{-1/2},\quad
\tilde{d} = d + \frac{\eta_1}{T} N^{-1/2},\quad
\xi_1 = \frac{\nu_2 - \eta_2/T}{\sqrt{d(1-a)}},\quad
\xi_2 = \frac{-\eta_1+\eta_2 + T(\nu_1-\nu_2)}{T\sqrt{ad}}.
$$
then up to negligible error, we turn the expression in Proposition~\ref{PropCorrCaracteristics} into the desired left-hand side formula above.
\end{proof}

The following calculations are formal. Consider the $(2+1)$-dimensional EW equation (\ref{EWeq}). We will calculate the covariance formally, ignoring the fact that $u$ is not function valued (see, for example,~\cite{Hairer}). We start by computing the invariant measure which turns out to be a Gaussian free field. Assuming ergodicity, the invariant solution can be written as
$$
u(t,x) = \int_{-\infty}^t \int_{\R^2} p_{t-s}(x-y)\, \xi(ds,dy)
$$
where the heat kernel is as above. From this, and the delta covariance of $\xi$ we find that
\begin{align*}
{\rm Cov}\big[u(t,x),u(\tilde t,\tilde x)\big] &= \int_{-\infty}^{t} \int_{-\infty}^{\tilde t} \int_{\R^2} \int_{\R^2}\,p_{t-s}(x-y)\,p_{\tilde{t}-\tilde{s}}(\tilde{x}-\tilde{y})\, {\rm Cov}\big[\xi(ds,dy),\,\xi(d\tilde{s},d\tilde{y})\big] \\
&=\int_{-\infty}^{\min(t,\tilde{t})} ds \int_{\R^2} dy\, p_{t-s}(x-y)\,p_{\tilde{t}-s}(\tilde{x}-y).
\end{align*}
From this we can calculate the full space-time covariance structure. First assume $t=\tilde{t}$ in which case the above integral evaluates to
$$
{\rm Cov}\big[u(t,x),\,u(t,\tilde x)\big] =(4\pi)^{-1} \Gamma\left(0,\frac{|x-\tilde x|^2}{4s}\right)\Bigg|_{s=0}^{s=\infty}.
$$
Notice that for $s$ near infinity,
$$(4\pi)^{-1} \Gamma\left(0,\frac{|x-\tilde x|^2}{4s}\right) \approx (4\pi)^{-1} \big(\ln s - \ln|x-\tilde x| + c\big)$$
for some constant $c$, whereas for $s$ near zero, the expression goes to 0. Therefore, if we look at the covariance of differences, we can remove the $\ln s$ divergence and the constant $c$. This implies that
$$
{\rm Cov}\big[u(t,x)-u(t,y) ,u(t,\tilde x) -u(t,\tilde y)\big] = -(2\pi)^{-1} \Big(\ln|x-\tilde x| -\ln|x-\tilde y| - \ln|y-\tilde x| + \ln|y-\tilde y|\Big),
$$
which shows that for any fixed time $t$, $u(t,x)$ is a Gaussian free field in $x\in \R$~\cite{She03,BF08}.

To compute the space-time covariance, we can use the above formulas. Alternatively, by the Duhamel principle we can write for $t>0$,
$$
u(t;x) = \int_{\R^2}dy\, p_t(x-y) \,u_0(y) + \int_{0}^t \int_{\R^2}p_{t-s}(x-y)\,\xi(ds,\,dy).
$$
For initial data $u_0(x)$ given by the Gaussian free field with the above covariance, using the independence of the noise $\xi$ with $u_0$ we compute
$$
{\rm Cov}\big(u(t,x),u(0,0)) = \int_{\R^2}dy \,p_t(x-y)\,{\rm Cov}[u_0(y),u_0(0)] = \frac{-1}{2\pi} \int_{\R^2}dy \,p_t(x-y)\,\ln(|y|) = \frac{G_{t}(|x|)}{4\pi}.
$$
By translation invariance we thus conclude that for $t>\tilde t$
$$
{\rm Cov}\big[u(t,x)-u(t,y) ,u(\tilde t,\tilde x) -u(\tilde t,\tilde y)\big] = (4\pi)^{-1} \Big(G_{t-\tilde t}(|x-\tilde x|) -G_{t-\tilde t}(|x-\tilde y|) - G_{t-\tilde t}(|y-\tilde x|) + G_{t-\tilde t}(|y-\tilde y|)\Big).
$$

\appendix
\section{Generalities of Gaussian processes}\label{AppGaussians}
Here we recall some basics of Gaussian processes as we will use them (see e.g.~\cite[Section VIII.6]{Ka07}) An $n$-dimensional diffusion process $X_t$ with linear drift, say drift $\mu=A X_t$ for a given matrix $A$ (possibly time-dependent), and space-independent dispersion matrix $\sigma$ is a solution of a system of SDE's
\begin{equation*}
dX^k_t= A(t) X_t dt+\sigma(t) dW_t
\end{equation*}
with $W_t$ being a standard $n$-dimensional Brownian motion. Then, the probability density $P(x,t)$ that the process $X$ is at $x$ at time $t$ satisfies the Fokker-Planck equation
\begin{equation*}
\frac{d P(x,t)}{dt} = -\sum_{i,j} A_{i,j} \frac{d}{dx_i} (x_j P(x,t)) +\frac12 \sum_{i,j} B_{i,j} \frac{d^2 P(x,t)}{dx_i dx_j},
\end{equation*}
where $B=\sigma^{\mathrm{T}}\sigma$ is the diffusion matrix.

In particular, if one starts with $\delta$-initial condition at $x(0)$, i.e., $P(x,0)=\prod_{i=1}^r \delta(x_i-x_i(0))$, then
\begin{equation*}
\EE(X(t))= Y(t) X(0),
\end{equation*}
where $Y(t)$ is the evolution matrix satisfying
\begin{equation}\label{eqAppPropagator}
\frac{d Y(t)}{dt} = A(t) Y(t),\quad Y(0)=\Id.
\end{equation}
Further, the solution of the Fokker-Plank equation is given by
\begin{equation*}
P(x,t)=\frac{1}{[(2\pi)^n\det(\Xi)]^{1/2}} \exp\left(-\frac12 (x-\EE(X(t))^{\mathrm{T}} \Xi^{-1} (x-\EE(X(t))\right)
\end{equation*}
where $\Xi(t)$ is the covariance matrix given by
\begin{equation*}
\Xi(t)=\int_0^t ds Y(t) Y^{-1}(s) B(s) Y^{-{\mathrm{T}}}(s) Y^{\mathrm{T}}(t).
\end{equation*}
Further, $\Xi$ can be characterized as the solution of the equation
\begin{equation*}
\frac{d\Xi}{dt}=A \Xi + \Xi A^{\mathrm{T}}+B,\quad \Xi(0)=0.
\end{equation*}

We compute the two-time distribution when such a Gaussian transition probability is applied to a Gaussian distribution with covariance matrix $C_1:=\Xi(t_1)$, i.e., with density given by \mbox{${\rm const}\cdot e^{-\frac12 x^{\mathrm{T}} C_1^{-1} x}$}. Denote as well $C_{1,2}=\Xi(t_2,t_1)$ and $t=t_2-t_1$. Then,
\begin{equation*}
\PP(x(t_1)\in dx, x(t_2)\in dy)={\rm const} \exp\left(-\frac12 \left[x^{\mathrm{T}} C_1^{-1} x+(y-Y(t))^{\mathrm{T}} C_{1,2}^{-1} (y-Y(t)x)\right]\right)dx dy.
\end{equation*}
The quadratic form in the parenthesis can be written as
\begin{equation*}
(x^{\mathrm{T}},y^{\mathrm{T}})
\left(\begin{array}{cc}
C_1^{-1}+ Y(t)^{\mathrm{T}} C_{1,2}^{-1} Y(t) & -Y^{\mathrm{T}}(t)\\
-C_{1,2}^{-1} Y(t) & C_{1,2}^{-1}
\end{array}\right)
\left(\begin{array}{c}
x \\ y
\end{array}\right).
\end{equation*}
We use the block matrix inversion formula
\begin{equation*}
\left(\begin{array}{cc}
a & b \\ c & d
\end{array}\right)^{-1}
=
\left(\begin{array}{cc}
-m^{-1} & m^{-1} b d^{-1} \\
d^{-1} c m^{-1} & d^{-1}-d^{-1} c m^{-1} b d^{-1}
\end{array}\right)
\end{equation*}
with $m=b d^{-1} c -a$ and obtain
\begin{equation*}
\left(\begin{array}{cc}
C_1^{-1}+ Y(t)^{\mathrm{T}} C_{1,2}^{-1} Y(t) & -Y^{\mathrm{T}}(t)\\
-C_{1,2}^{-1} Y(t) & C_{1,2}^{-1}
\end{array}\right)^{-1}
=
\left(\begin{array}{cc}
C_1 & C_1 Y^{\mathrm{T}}(t) \\
Y(t) C_1 & C_{1,2}+Y(t)C_1Y^{\mathrm{T}}(t).
\end{array}\right)
\end{equation*}
Thus this is the covariance matrix for the two-time distribution. In particular, the covariance between a site $x(t_1)$ and $y(t_2)$ is given by the application of the propagator $Y(t_2-t_1)$ to the covariance $C_1$ at time $t_1$.

\section{Additional $q$-Whittaker dynamics}\label{appother}

\subsection{Alpha dynamics}

In addition to the push-block alpha dynamics, there is also an RSK type dynamic (see~\cite[Section 6.2]{PM15}, or~\cite{P16}). Considering interlacing partitions $\bflambda$, we define a Markov transition matrix $P^{\rm{RSK}}_{\vec{a};\alpha}\big(\bflambda \to \bfnu\big)$ to a new set of interlacing partition $\bfnu$ according to the following update procedure. For $k=1,\ldots, N$ choose independent random variables $v_k$ distributed according to the $q$-geometric law with parameter $\alpha a_k$ (see Section~\ref{secqdef}). For $1\leq k\leq n-1$, let $c_k = \nu^{(n-1)}_{k}-\lambda^{(n-1)}_k$. Choose $w_1,\ldots, w_{n-1}$ independently so that $w_k\in \{0,1,\ldots, c_k\}$ is distributed according to
$$
\PP(w_k = s)= \varphi_{q^{-1}, q^{\lambda^{(n)}_k - \lambda^{(n-1)}_{k}}, q^{\lambda^{(n-1)}_{k-1}-\lambda^{(n-1)}_{k}}} \big(s | c_k\big),
$$
where we recall the convention that $\lambda^{(n)}_0=+\infty$ for all $n$.

Now update
$$
\nu^{(n)}_1 = \lambda^{(n)}_1 + w_1 +v_n,\qquad\textrm{and for $k\geq 2$ }\quad \nu^{(n)}_k = \lambda^{(n)}_{k} +w_k + c_{k-1}-w_{k-1}.
$$

\begin{proposition}
Define a Markov process indexed by $t$ on interlacing partitions $\bflambda(t)$ with packed initial data and Markov transition between time $t-1$ and $t$ given by $P^{\rm{RSK}}_{\vec{a};\alpha_{t}}\big(\bflambda(t-1) \to \bflambda(t)\big)$. Then, for any $t\in \{0,1,\ldots\}$, $\bflambda(t)$ is marginally distributed according to the $q$-Whittaker measure $\PP_{\vec{a};\vec{\alpha}(t)}$ with $\vec{\alpha}(t) = (\alpha_1,\ldots, \alpha_t)$.
\end{proposition}
\begin{proof}
This follows from~\cite[Theorem~6.4]{PM15}.
\end{proof}

In Section~\ref{secalphaode} we probe the limiting difference equation which arises from the push-block dynamics as $q\to 1$. There should be difference equations for the RSK type dynamics above, though we do not pursue this here. We also do not pursue any fluctuation limits.

\subsection{Plancherel dynamics}
We consider two additional dynamics besides the push-block Plancherel dynamics.

The ``right-pushing dynamics'' were introduced as~\cite[Dynamics 9]{BP13b}. For $2\leq k\leq n\leq N$, each $\lambda^{(n)}_{k}$ evolves according to the push-block dynamics. The only difference is in the behavior of $\lambda^{(n)}_1$. For $1\leq n\leq N$, each $\lambda^{(n)}_1$ jumps (i.e., increases value by one) at rate $a_n q^{\lambda^{(n-1)}_{1}-\lambda^{(n)}_{2}}$. When $\lambda^{(n)}_{1}$ jumps it deterministically forces $\lambda^{(n+1)}_1,\ldots, \lambda^{(N)}_1$ to likewise increase by one. It is clear that these dynamics preserve the interlacing structure of $\bflambda$.

The ``RSK type dynamics'' were introduced as~\cite[Dynamics 8]{BP13b} (see also~\cite{P16}). For $1\leq n\leq N$, each $\lambda^{(n)}_1$ has its own independent exponential clock with rate $a_n$. When the clock rings, the particle $\lambda^{(n)}_1$ {\it jumps} (i.e., increases value by one). These are all of the independent jumps, however there are certain triggered moves. When a particle $\lambda^{(n-1)}_k$ jumps, it triggers a single jump from some coordinate of $\lambda^{(n)}$. Let $\xi(k)$ represent the maximal index less than $k$ for which increasing $\lambda^{(n)}_{\xi(k)}$ by one does not violate the interlacing rules between the new $\lambda^{(n)}$ and $\lambda^{(n-1)}$. With probability
$$
q^{\lambda^{(n)}_k-\lambda^{(n-1)}_k} \frac{1- q^{\lambda^{(n-1)}_{k-1}-\lambda^{(n)}_k}}{1- q^{\lambda^{(n-1)}_{k-1}-\lambda^{(n-1)}_k}}
$$
(recall the convention that $\lambda^{(n)}_0\equiv +\infty$) $\lambda^{(n)}_{\xi(k)}$ jumps; and with complementary probability $\lambda^{(n)}_{k+1}$ jumps. It is clear that these dynamics maintain the interlacing structure of $\bflambda$.

\begin{proposition}
For each of the right-pushing and RSK type dynamics described above, define continuous time Markov processes, all denoted by $\bflambda(\gamma)$, started from packed initial data. Then, for any $\gamma>0$, $\bflambda(\gamma)$ is marginally distributed according to the Plancherel specialized $q$-Whittaker process $\PP_{\vec{a};\gamma}$.
\end{proposition}
\begin{proof}
This follows from a combination of~\cite[Proposition 8.2 and Theorem 6.13]{BP13b}.
\end{proof}

In the same manner as in Section~\ref{secplancode}, we can derive ODEs for the LLN from the above continuous time dynamics. In the right-pushing case, the dynamics are the same as in the push-block case aside from $\lambda^{(n)}_1$. Hence by the same reasoning that for $k\geq 2$, (\ref{eqnpushode}) still holds.
From the right-pushing rule, we similarly deduce that the following should hold
$$
\frac{d x^{(n)}_{1}(\tau)}{d\tau} = a_n e^{-x^{(n-1)}_1(\tau) + x^{(n)}_2(\tau)} + \sum_{\ell=1}^{n-1}\frac{d x^{(\ell)}_1(\tau)}{d\tau}.
$$

For the RSK type dynamics, let us assume that all particles are well-spaced (as they surely will be after a short amount of time). Then we need not worry about transferring jumps in the RSK type dynamics. Thus, by similar reasoning as in the push-block case we find that
$$
\frac{d}{d\tau} x^{(1)}_{1}(\tau) = a_1, \qquad \frac{d}{d\tau} x^{(n)}_{1}(\tau) = a_n + \frac{d}{d\tau}x^{(n-1)}_1(\tau) \, \cdot \, e^{-x^{(n)}_1(\tau)+x^{(n-1)}_{1}(\tau)}
$$
and for $k\geq 2$,
\begin{align*}
 \frac{d}{d\tau} x^{(n)}_{k}(\tau) &= \frac{d}{d\tau} x^{(n-1)}_{k}(\tau) e^{-x^{(n)}_k(\tau)+x^{(n-1)}_k(\tau)}\, \frac{1-e^{-x^{(n-1)}_{k-1}(\tau)+x^{(n)}_k(\tau)}}{1-e^{-x^{(n-1)}_{k-1}(\tau)+x^{(n-1)}_k(\tau)}}\\
 &\,\,\,+ \frac{d}{d\tau} x^{(n-1)}_{k-1}(\tau) \bigg(1- e^{-x^{(n)}_{k-1}(\tau)+x^{(n-1)}_{k-1}(\tau)}\, \frac{1-e^{-x^{(n-1)}_{k-2}(\tau)+x^{(n)}_{k-1}(\tau)}}{1-e^{-x^{(n-1)}_{k-2}(\tau)+x^{(n-1)}_{k-1}(\tau)}}\bigg).
\end{align*}
In these equations, on the right-hand side the differential term $d/d\tau$ gives the rate of jumps from below whereas the terms multiplying that correspond to the proportion of this jump rate which is transferred to $x^{(n)}_k$.

We do not pursue these alternative dynamics any further, though note that they may yield different fluctuation SDEs than in the push-block case (though they will all have the same marginals).

\subsection{Positivity in determinantal expressions}
Recall that in Corollary~\ref{CorPolynomials}, equation (\ref{eq3.16p}) provides the determinantal formula

$$e^{-(x^{(n)}_{n}(\tau) +\cdots + x^{(n)}_{n-r+1}(\tau))} = e^{-\tau r} \det\Big[G_{r,\tau}(n+1-r+j-i)\Big]_{i,j=1}^r.$$

We show below that
$$ e^{-\tau r} \det\Big[G_{r,\tau}(n+1-r+j-i)\Big]_{i,j=1}^r =e^{-\tau r} p^{n}_r(\tau),$$
where $p^{n}_r(\tau)$ is a polynomial in $\tau$ with \emph{positive coefficients}. This positivity is surprising and its origins warrants further investigation. 

The above representation is shown by realizing the determinant as a partition function for a certain system of non-intersecting paths with positive weights. We explain this for the above determinant, equation (\ref{eq3.16p}) in the main text, as well as the alpha version, equation (\ref{eq3.16}) in the main text.

\begin{figure}
\centering
\begin{subfigure}
 \centering
\psfrag{0}[r]{$0$}
\psfrag{rt}[r]{$r+t$}
\psfrag{r}[r]{$r$}
\psfrag{1}[c]{$1$}
\psfrag{n}[c]{$n$}
 \includegraphics[height=5cm]{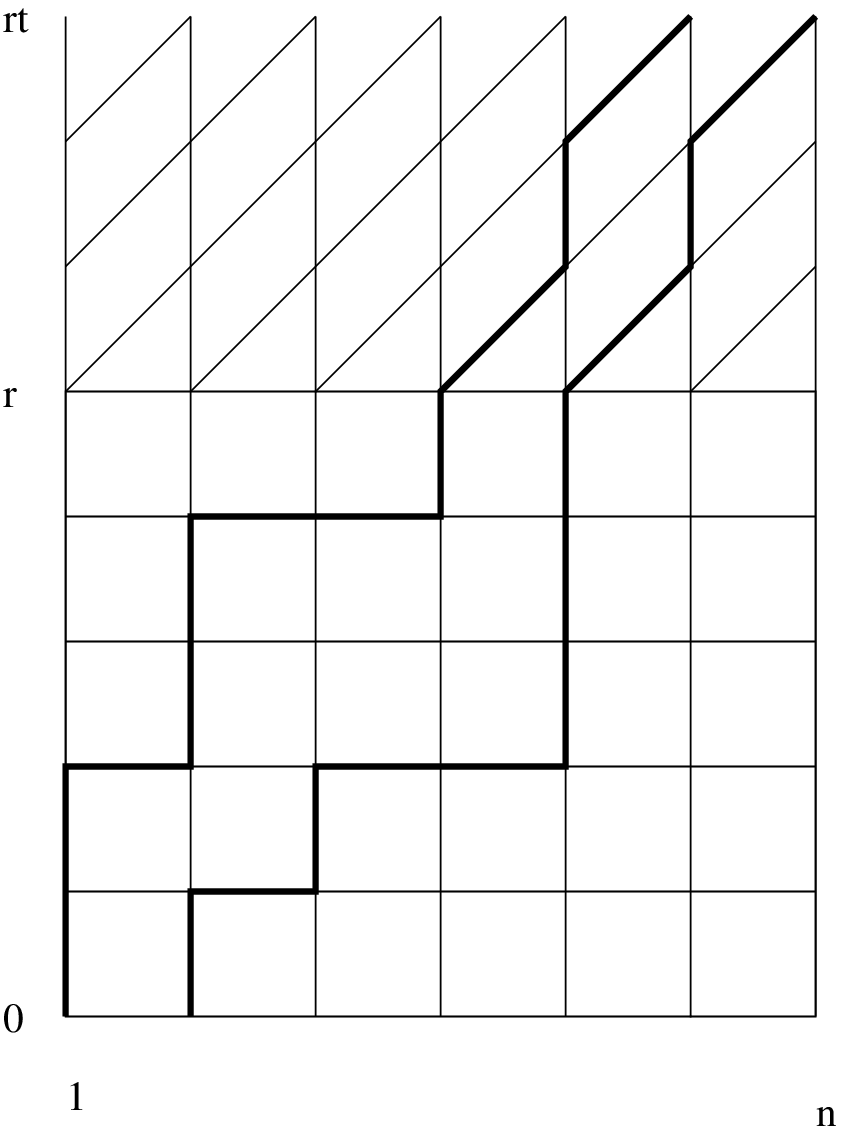}
\end{subfigure}\hspace{2cm}
\begin{subfigure}
 \centering
\psfrag{0}[r]{$0$}
\psfrag{r-1}[r]{$r-1$}
\psfrag{r}[r]{$r$}
\psfrag{1}[c]{$1$}
\psfrag{n}[c]{$n$}
 \includegraphics[height=5cm]{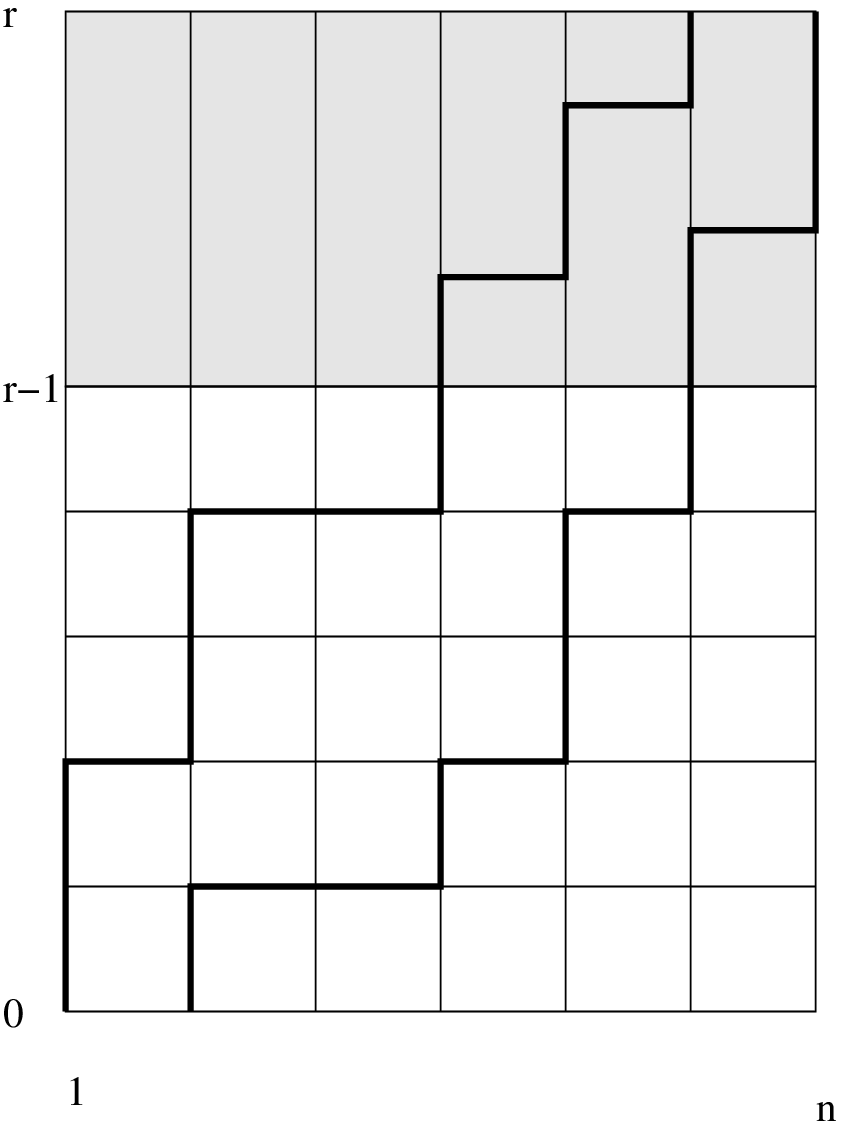}
\end{subfigure}
\caption{Karlin-McGregor interpretation of the determinant in (\ref{eq3.16}) and (\ref{eq3.16p}).}
\label{FigKarlinMcGregor}
\end{figure}

It is possible to represent the determinant in (\ref{eq3.16}) in terms of the partition function for a collection of $r$ non-intersecting paths on a certain weighted lattice. Consider the lattice on the left of Figure~\ref{FigKarlinMcGregor} which has width $n$ and height $r+t$. The bottom $r$ portion of the lattice is the standard square lattice, and every edge (horizontal and vertical) has a weight of 1. The top $t$ portion of the lattice is composed of vertical edges and diagonal up-right edges. Each diagonal edge between level $r+\ell$ and $r+\ell+1$ has weight $1-\alpha_{\ell}$ and each vertical edge between those levels has weight $\alpha_\ell$. The weight of a directed path (only taking up or right edges for the first $r$ levels and then up or up-right edge for the remaining levels) from level 1, position $i$ to level $r+t$, position $j$ ($1\leq i\leq j\leq n$) is the product of the weights along the path. The partition function is the sum of these weights over all such paths and is readily computed as
\begin{equation}\label{eq3.19}
\sum_{c=i}^{j}\binom{r}{c-i} e_{j-c}(1-\vec{\alpha};\vec{\alpha}) = \sum_{\ell=0}^{j-i}e_{\ell}(1-\vec{\alpha};\vec{\alpha}) \frac{(r)_{j-i-\ell}}{(j-i-\ell)!}=G_{r,j-i}(j-i+1).
\end{equation}
The Lindstr\"om-Gessel-Viennot theorem implies that the partition function for a collection of $r$ non-intersecting paths is written as an $r$-by-$r$ determinant. In particular, taking the starting points on level one of the $r$ paths to be $(1,\ldots, r)$ and the ending points on level $r+t$ to be $(n+1-r,\ldots, n)$ we find that this partition function is exactly $\det[G_{r,t}(n+1-r+j-i)]_{i,j=1}^r$.

Similarly, in the Plancherel case of (\ref{eq3.16p}) (see the right part of Figure~\ref{FigKarlinMcGregor}) we consider $r$ non-intersecting paths from positions $(1,\ldots,r)$ to $(n+1-r,\ldots,n)$ such that in the first part they either go up or to the right until reaching level $r-1$, in the second part they perform one-sided continuous simple random walk with jump rate $\tau$ during a time span of $1$. A combination of the Lindstr\"om-Gessel-Viennot and Karlin-McGregor theorems imply that the probability that these $r$ paths do not intersect is proportional to $e^{-\tau r} \det[G_{r,\tau}(n+1-r+j-i)]_{i,j=1}^r$. On the other hand, for a single path, the probability of going from a fixed starting point to fixed ending point is proportional to $e^{-\tau}$ times a polynomial in $\tau$ with positive coefficients. Therefore the probability of the $r$ non-intersecting paths also takes the form $e^{-\tau r}$ times a polynomial in $\tau$ with positive coefficients, which shows the positivity of the polynomial $p^{n}_r(\tau)$.

\section{Proof of Proposition~\ref{propMatrixIdentities}}\label{AppMatrixIdentites}
Let us first prove (\ref{eqApp12}). Doing linear combinations of columns and using the relation (\ref{eqApp11}), the determinants in (\ref{eqApp12}) can be rewritten as follow:
\begin{equation}
\begin{aligned}
Q_1&:=\det[B_{i,j}]_{1\leq i,j\leq M+1} =
\left|\left(
 \begin{array}{cccc}
 C_{1,1} & \cdots & C_{1,M} & B_{1,M+1} \\
 \vdots & \ddots & \vdots & \vdots \\
 C_{M+1,1} & \cdots & C_{M+1,M} & B_{M+1,M+1} \\
 \end{array}
 \right)
\right|,\\
Q_2&:= \det[C_{i+1,j+1}]_{i,j=1}^M,
\end{aligned}
\end{equation}
and
\begin{equation}
\begin{aligned}
Q_3&:=\det[C_{i,j}]_{i,j=1}^{M+1}, \\
Q_4&:=\det[B_{i+1,j+1}]_{i,j=1}^M =
\left|\left(
 \begin{array}{cccc}
 C_{2,2} & \cdots & C_{2,M} & B_{2,M+1} \\
 \vdots & \ddots & \vdots & \vdots \\
 C_{M+1,1} & \cdots & C_{M+1,M} & B_{M+1,M+1} \\
 \end{array}
 \right)
\right|,
\end{aligned}
\end{equation}
and finally
\begin{equation}
\begin{aligned}
Q_5&:=\gamma \det[B_{i,j+1}]_{i,j=1}^{M+1} =
-\left|\left(
 \begin{array}{cccc}
 C_{1,2} & \cdots & C_{1,M+1} & B_{1,M+1} \\
 \vdots & \ddots & \vdots & \vdots \\
 C_{M+1,2} & \cdots & C_{M+1,M+1} & B_{M+1,M+1} \\
 \end{array}
 \right)
\right|,\\
Q_6&:=\det[C_{i+1,j}]_{i,j=1}^M.
\end{aligned}
\end{equation}
With these notations we have $(\ref{eqApp12})=Q_1 Q_2 -Q_3 Q_4 + Q_5 Q_6$.

Let us define the following $(2M+1)\times (2M+1)$ matrix,
\begin{equation}
Q=\left(
 \begin{array}{ccccccccc}
 C_{1,2} & \cdots & C_{1,M} & C_{1,M+1} & B_{1,M+1} & C_{1,1} &0 & \cdots & 0\\
 \vdots & \ddots & \vdots & \vdots & \vdots& \vdots &\vdots & \ddots & \vdots \\
 C_{M+1,2} & \cdots & C_{M+1,M} & C_{M+1,M+1} & B_{M+1,M+1} &C_{M+1,1} &0 & \cdots & 0\\
 0 & \cdots & 0 & C_{2,M+1} & B_{2,M+1} & C_{2,1} &C_{2,2} & \cdots & C_{2,M} \\
 \vdots & \ddots & \vdots & \vdots & \vdots & \vdots &\vdots & \ddots & \vdots \\
 0 & \cdots & 0 & C_{M+1,M+1} & B_{M+1,M+1} &C_{M+1,1} &C_{M+1,2} & \cdots & C_{M+1,M} \\
 \end{array}
 \right).
\end{equation}
Next, notice that for a square block matrix of the form $\left(\begin{array}{cc} \alpha & 0\\ 0 &\beta\end{array}\right)$, its determinant is always zero unless $\alpha$ (and thus $\beta$) are square matrices. Adding the block of the last $M-1$ columns to the first $M-1$ columns and then subtracting rows $1+j$ from $M+1+j$, $j=1,\ldots,M$, we obtain a matrix block matrix with zeroes but with $\alpha$ of size $(M+2)\times (M+1)$. Thus $\det(Q)=0$.

In $Q$ there are three columns without zero entries. Call the first block with $C$'s above the zeroes as $A_1$ and the last block below the zeroes as $A_2$. Then we can write $Q$ in the following form
\begin{equation}
Q=\left(
 \begin{array}{ccccc}
 A_1 & U_1 & U_2 & U_3 &0\\
 0 & L_1 & L_2 & L_3 & A_2
 \end{array}
 \right),
\end{equation}
where $U_i$ are $(M+1)$-vectors and $L_i$ are $M$-vectors. By multi-linearity of the determinant, we have that $\det(Q)$ equals the sum of the determinants of the matrices obtained by replacing for each pair $(U_i,L_i)$, $i=1,2,3$, one of the elements by the vector of zeroes. Thus obtained matrices are of the block matrix form with zero corners as described above but, except if one sets exactly one of the $U_i=0$, the $\alpha$ matrix is not square. Thus we have
\begin{equation}
\begin{aligned}
0=\det(Q)&=\det\left(
 \begin{array}{ccccc}
 A_1 & 0 & U_2 & U_3 &0\\
 0 & L_1 & 0 & 0 & A_2
 \end{array}
 \right)+
 \det\left(
 \begin{array}{ccccc}
 A_1 & U_1 &0 & U_3 &0\\
 0 & 0& L_2 & 0 & A_2
 \end{array}
 \right)\\
 &+\det\left(
 \begin{array}{ccccc}
 A_1 & U_1 & U_2 &0 &0\\
 0 & 0& 0 & L_3 & A_2
 \end{array}
 \right) = -Q_1 Q_2 + Q_3 Q_4 - Q_5 Q_6 = - (\ref{eqApp12}).
\end{aligned}
\end{equation}

Next we prove (\ref{eqApp13}). In the first step, using linear combinations of columns, we can replace in the determinants of (\ref{eqApp13}), the $B_{i,j}$'s with $C_{i,j}$'s except for the last column. This gives
\begin{equation}
\begin{array}{ll}
P_1:=\gamma \det[B_{i,j}]_{i,j=1}^{M+1} = Q_1, &P_2:=\det[C_{i+1,j+1}]_{i,j=1}^{M-1},\\[0.5em]
P_3:=\det[C_{i,j}]_{i,j=1}^{M}, &P_4:=\det[B_{i+1,j+1}]_{i,j=1}^{M}=Q_4,
\end{array}
\end{equation}
and finally
\begin{equation}
\begin{aligned}
P_5&:=\det[B_{i,j+1}]_{i,j=1}^{M} =
\left|\left(
 \begin{array}{cccc}
 C_{1,2} & \cdots & C_{1,M} & B_{1,M+1} \\
 \vdots & \ddots & \vdots & \vdots \\
 C_{M,2} & \cdots & C_{M,M} & B_{M,M+1} \\
 \end{array}
 \right)
\right|,\\
P_6&:=\det[C_{i+1,j}]_{i,j=1}^M.
\end{aligned}
\end{equation}
We have to prove $(\ref{eqApp13})=P_1 P_2 - P_3 P_4 + P_5 P_6 =0$. Now we have written all the $P_i$'s in terms of $C_{i,j}$'s and sometimes one single column of $B_{k,M+1}$'s. Let us show that the factor multiplying $B_{k,M+1}$ equals zero for all $k=1,\ldots,M+1$. First, for $k=1$ (resp.\ $k=M+1$), it is immediate to see that the factors are zero, since there are only two contributing terms: one from $P_1$ and the other one from $P_4$ (resp.\ $P_5$). Now take a fixed $k\in\{2,\ldots,M\}$. Then, the factor in (\ref{eqApp13}) which multiplies $B_{k,M+1}$ is given by the sum of these three terms:
\begin{equation}
A_1=\left|\left(
 \begin{array}{cccccc}
 C_{2,2} & \cdots & C_{2,M} & & &\\
 \vdots & \ddots & \vdots & & 0 &\\
 C_{M+1,2} & \cdots & C_{M+1,M} & & &\\
 & & & C_{1,1} & \cdots & C_{1,M}\\
 & 0 &  & \vdots &\textrm{No }C_{k,\cdot} & \\
 & & & C_{M,2} & \cdots & C_{M,M}\\
 \end{array}
 \right)
\right|,
\end{equation}
where $\textrm{No }C_{k,\cdot}$ means that the row with the $C_{k,j}$'s is missing,
\begin{equation}
A_2=-\left|\left(
 \begin{array}{cccccc}
 C_{1,1} & \cdots & C_{1,M} & & &\\
 \vdots & \ddots & \vdots & & 0 &\\
 C_{M,1} & \cdots & C_{M,M} & & &\\
 & & & C_{2,2} & \cdots & C_{2,M}\\
 & 0 &  & \vdots &\textrm{No }C_{k,\cdot} & \\
 & & & C_{M+1,2} & \cdots & C_{M+1,M}\\
 \end{array}
 \right)
\right|,
\end{equation}
and
\begin{equation}
A_3=\left|\left(
 \begin{array}{cccccc}
 C_{2,1} & \cdots & C_{2,M} & & &\\
 \vdots & \ddots & \vdots & & 0 &\\
 C_{M+1,1} & \cdots & C_{M+1,M} & & &\\
 & & & C_{1,2} & \cdots & C_{1,M}\\
 & 0 &  & \vdots &\textrm{No }C_{k,\cdot} & \\
 & & & C_{M,2} & \cdots & C_{M,M}\\
 \end{array}
 \right)
\right|.
\end{equation}
We need to show that $A_1+A_2+A_3=0$. Define the matrix
\begin{equation}
P=\left|\left(
 \begin{array}{cccccc}
 C_{1,1} & \cdots & C_{1,M} & C_{1,2} & \cdots & C_{1,M}\\
 \vdots & \ddots & \vdots & & 0 &\\
 C_{k,1} & \cdots & C_{k,M} & C_{k,2} & \cdots & C_{k,M}\\
 \vdots & \ddots & \vdots & & 0 &\\
 C_{M+1,1} & \cdots & C_{M+1,M} & C_{M+1,2}& \cdots& C_{M+1,M}\\
 & & & C_{2,2} & \cdots & C_{2,M}\\
 & 0 &  & \vdots &\textrm{No }C_{k,\cdot} & \\
 & & & C_{M,2} & \cdots & C_{M,M}\\
 \end{array}
 \right)
\right|,
\end{equation}
where the top-right $(M+1)\times (M-1)$ block contains only three non-zero rows. It is easy to verify that $\det(P)=0$, since by linear combinations of rows and columns we can delete the three lines in the just mentioned block. Expanding the determinant by multi-linearity, for each of the three rows of $P$ without zeroes, we can decide whether to keep the first $M$ terms or the last $M-1$. Only when we replace with zeroes exactly one set of the $M$ terms and the other two sets of $M-1$ terms we get a non-zero determinant by the same argument with block matrices with zero corners as used above. Up to reordering the columns, we have a block diagonal determinant, leading (up to a $(-1)^M$ factor) to $A_1$ when we keep $(C_{k,2},\cdots,C_{k,M})$, to $A_2$ when we keep $(C_{M+1,2},\cdots,C_{M+1,M})$ and $A_3$ when we keep $(C_{1,2},\cdots,C_{1,M})$. This finishes the proof of the identity (\ref{eqApp13}).


\end{document}